\documentclass[12pt,reqno]{amsart}
\usepackage{amssymb,mathrsfs,graphicx,color}
\usepackage{mathtools}
\usepackage{tikz}
\usepackage{comment}
\usetikzlibrary{shapes,arrows, decorations.pathreplacing,patterns}
\usepackage[margin=3cm]{geometry}
-\marginparwidth \oddsidemargin -9mm \evensidemargin \oddsidemargin
\usepackage{latexsym}
\advance\hoffset by 5mm

\mathtoolsset{showonlyrefs=true}
\numberwithin{equation}{section}

\def\R{\mathbb{R}}
\def\Z{\mathbb{Z}}
\def\S{\mathbb{S}}

\def\pa{\partial}

\def\m{\mathbf{m}}

\def\Pm{\mathbb{P}}
\def\xb{\bar{x}}
\def\ux{\delta} 
\def\Sm{S_m} 
\def\Smc{S_m^c} 
\def\Smp{S_{m+1}} 
\def\Smcp{S_{m+1}^c} 
\def\Dmis{\Omega}  
\DeclareMathOperator{\arccot}{arccot}
\def \Rm {\mathbb R}
\def\G{\mathcal G}
\def\B{\mathcal B}
\def\D{\mathcal D}
\def\phim{\varphi_{\min}}
\def\sigmaz{\sigma_0}
\def\sigmasec{\sigma_1}
\def\C{C}
\def\Q{Q}
\def\M{Z}

\newtheorem{theorem}{Theorem}[section]

\newtheorem{lemma}[theorem]{Lemma}
\newtheorem{proposition}[theorem]{Proposition}
\newtheorem{corollary}[theorem]{Corollary}

\newtheorem{remark}[theorem]{Remark}

\title[The Flow of Polynomial Roots]{The Flow of Polynomial Roots Under Differentiation}
\author{Alexander Kiselev}
\thanks{Department of
Mathematics, Duke University, Durham NC 27708, USA;
email: kiselev@math.duke.edu}

\author{Changhui Tan}
\thanks{Department of Mathematics, University of South Carolina, Columbia SC 29208, USA;
email: tan@math.sc.edu  }

\begin{document}


\begin{abstract}
The question about the behavior of gaps between zeros of polynomials under differentiation is classical and goes back to Marcel Riesz.
In this paper, we analyze a nonlocal nonlinear partial differential equation formally derived by Stefan Steinerberger \cite{Steinerberger2018} to model dynamics of roots
of polynomials under differentiation.
Interestingly, the same equation has also been recently obtained formally by Dimitri Shlyakhtenko and
Terence Tao as the evolution equation for free fractional convolution of a measure \cite{STao} - an object in free probability that is also related to minor processes
for random matrices.
The partial differential equation bears striking resemblance to
hydrodynamic models used to describe the collective behavior
of agents (such as birds, fish or robots) in mathematical biology.
We consider periodic setting and show global regularity and exponential in time convergence to uniform density for solutions corresponding to strictly positive smooth initial data.
In the second part of the paper we connect rigorously solutions of the Steinerberger's PDE and evolution of roots under differentiation for a class of trigonometric
polynomials. Namely, we prove that the distribution of the zeros of the derivatives of a polynomial and the corresponding solutions of the PDE remain close for all times.
The global in time control follows from the analysis of the propagation of errors equation, which turns out to be a nonlinear fractional heat equation
with the main term similar to the modulated discretized fractional Laplacian $(-\Delta)^{1/2}$.
\end{abstract}

\subjclass[2010]{26C10,\,\,35Q70;\,\,44A15,\,\,46L54,\,\,35Q92,\,\,35Q35,\,\,60B20}
\keywords{Zeros of polynomials, gaps between roots under differentiation, convergence to equilibrium, flocking, mixing, nonlocal transport, global regularity, modulus of continuity, free fractional convolution of measures, minor process, random matrices}

\maketitle

\section{Introduction}\label{intro}

The analysis of the relation between the zero set of a polynomial or an entire function and the zero set of its derivative has rich history.
The Gauss-Lucas theorem \cite{Gauss,Lucas,Marden} says that for a polynomial on complex plane, the zero set of the derivative lies in the convex hull of the
zero set. This direction remains very active, see e.g. \cite{CM,Malamud,R1,St2,Tot2} for some recent advances and further references.
Classic conjectures by Polya and Wiman \cite{Polya1,Polya2,Wiman} dealt with the question
of disappearance (or appearance) of complex roots under differentiation for a class of entire functions;
see \cite{CCS,Sheil-Small} for some resolutions. Closer to our focus in this paper, the question about behavior of gaps between the roots of a real-valued
polynomial under differentiation goes back to Marcel Riesz. A result attributed to him \cite{Riesz}
shows that the smallest gap between the roots can only increase after differentiation, providing an indication that differentiation tends to ``even out"
distances between roots (see also \cite{SzN,Walker1,Walker2} for later related works). A rigorous proof of ``crystallization" under repeated differentiation - convergence of roots to an ideal lattice - has been
established for a class of trigonometric polynomials in \cite{FY}. Similar results were also discussed for a class of entire functions in \cite{FR},
and established for some random entire functions in \cite{PS}.
We also mention a recent related series of papers studying distribution of critical points of a random
or deterministic polynomial given the distribution of its roots \cite{PR,PS,Sub,Hanin,Kabluchko,ORW,KS,BLR,Tot1}, where further references can be found.

For a trigonometric polynomial, the process of ``crystallization" under differentiation is not difficult to understand on an elementary level: repeated differentiation
leads to larger factors appearing in front of the leading terms than lower order terms. If the polynomial has order $n,$ after differentiating $\sim An$ times,
the leading terms gain at least a constant $\sim e^A$ factor compared to all lower order terms. After sufficient number of differentiations, the leading term will
dominate and this will affect the location of roots, enforcing crystallization. One can think of the differentiation as gradually creating a spectral gap, which makes
contact with the celebrated Sturm-Hourwitz theorem \cite{Sturm1,Sturm2,Hurwitz} (see an excellent review \cite{BH} for the history, including contributions by Lord Raleigh and
Liouville). This theorem provides the estimate on a number of roots for trigonometric polynomials with spectral gap - except that in our case here the limiting
polynomial is very simple and so we can say more about the roots. Of course, specific and sufficiently strong bounds on convergence to ideal lattice can be much more
subtle to prove. Even harder question is to understand in more detail how the distribution of roots evolves under differentiation.

Recently, Steinerberger \cite{Steinerberger2018} proposed a partial differential equation to describe the evolution of roots for polynomials on the real axis.
The equation takes form
\begin{equation}\label{maineq}
\partial_t u + \frac{1}{\pi} \partial_x \left( \arctan \left( \frac{Hu}{u} \right) \right) =0,
\end{equation}
where $Hu = \frac{1}{\pi} P.V. \int_{\R}\frac{u(y)}{x-y}\,dy$ is the Hilbert transform of $u.$ The formal derivation of this PDE in \cite{Steinerberger2018} makes certain assumptions - which will be recalled in more detail in Section~\ref{derivation} -
that suggest that the PDE should approximate the dynamics of zeroes for polynomials of sufficiently high degree $n$, provided that their roots are distributed according to a smooth
density and maintain this property under repeated differentiation. The unit of time in \eqref{maineq} corresponds to $n$ differentiations,
so the evolution becomes trivial for $t >1.$ Some interesting explicit solutions, for example corresponding to the semicircle law and to Marchenko-Pastur distribution, are also described in
\cite{Steinerberger2018}, making links to well known asymptotic laws for roots of orthogonal polynomials \cite{ET,EF,U,VA}.
An equation similar to \eqref{maineq} was formally derived by O'Rourke and Steinerberger \cite{ORS} for the case of complex random polynomials with radial distribution of roots.

Interestingly, the equation \eqref{maineq} is also relevant in free probability and random matrices. In a very recent work of Shlyakhtenko and Tao \cite{STao} this
equation (under a simple change of variables, see \cite{St2020}) was formally obtained as a PDE for the evolution of free fractional convolution of a probability measure on $\R.$.
The free convolution of two probability measures $\mu \boxplus \nu$  is an object in free probability (see e.g. \cite{VDN}).
One can define $\mu \boxplus \nu$ to be the law of $X + Y$, where $X,$ $Y$ are freely
independent noncommutative random variables with law $\mu$ and $\nu$ respectively. One can then define the integer free convolution $\mu^{\boxplus k}= \mu \boxplus \dots \boxplus \mu$
to be the free convolution of the $k$ copies of $\mu.$ This (properly rescaled) object plays a key role in the free analog of the central limit theorem established by Voiculescu \cite{V}, where the limiting law is a Wigner semicircle
distribution. It turns out that the notion of the integer free convolution  $\mu^{\boxplus k}$ can be in a natural way extended to real $k \geq 1$ \cite{BV,NS}.
Under some additional assumptions, in particular that $d\mu^{\boxplus k} = f_k(x)\,dx$ are absolutely continuous, the equation \eqref{maineq} has been formally derived for $f_k(x)$ in \cite{STao}, Theorem 1.7 and Section 4; the variable
$1-\frac1k$ plays the role of time. The connection between free fractional convolution and the behavior of roots of polynomials under differentiation can be interpreted through the relation of both these processes to
minor process in random matrix theory. Roughly speaking, a minor process consists of a sequence of monotone minors of a random matrix ensemble. The connection between the fractional free convolution of $\mu^{\boxplus k}$ and the law
of a projection composed with non-commutative random variable $X$ with law $\mu$ has been established in \cite{NS}, see also \cite{STao} for a self-contained argument. In this case time is linked with the size of the minor.
On the other hand, \cite[Lemma 1.16]{MSS} establishes a link between differentiation and minor process, by showing that the expected
characteristic polynomial of a random restriction of a matrix is proportional to a derivative of its characteristic polynomial.
See also \cite{Malamud} on a direct link between the spectrum of sub-matrices and roots of the derivatives of characteristic polynomials.
In addition to these observations, \cite{HK} establishes a direct link between evolution of roots of a
polynomial under differentiation and free fractional convolution. The result of \cite{HK} applies for each fixed time in a limit of $n \rightarrow \infty,$ and does not directly involve
the PDE \eqref{maineq}. We point out several more recent papers that are also related to this circle of ideas, either linking evolution of roots under differentiation and minor process \cite{HS}, establishing connections between
limiting distributions of Bessel and Dunkl processes modeling particle systems and free convolutions \cite{VW} or proving a version of ``crystallization" for a class of random matrix ensembles \cite{GK}.

Our goal in this paper is twofold. First, we will consider the equation \eqref{maineq}  with
 periodic initial data in $\Rm,$ or, equivalently, set on a circle $\S=(-\pi, \pi],$ and prove global regularity for the case of sufficiently smooth positive initial data.
Let us rewrite \eqref{maineq} as
\begin{equation}\label{maineq1}
\partial_t u + \frac{1}{\pi}\frac{u \Lambda u -  Hu \partial_x u}{u^2+ Hu^2} =0.
\end{equation}
Here $\Lambda u= (-\Delta)^{1/2}u = \partial_x Hu.$ Recall that
\begin{equation}\label{fraclap}
\Lambda u(x) = \frac{1}{\pi}  P.V. \int_{\Rm}\frac{u(x)-u(y)}{|x-y|^2}\,dy
=\frac{1}{4\pi} P.V. \int_{\S}\frac{u(x)-u(y)}{\sin^2(\frac{x-y}{2})}\,dy.
\end{equation}
The equation \eqref{maineq1} bears striking resemblance to models that have been developed in mathematical biology to describe the
flocking behavior of animals, cells, and microorganisms,
and in particular to the 1D Euler alignment model (see e.g. \cite{carrillo2016critical,DKRT,ha-liu2009,karper-mellet-trivisa,ST1,ST2,tadmor2014critical}),
and to some models in fluid mechanics \cite{CaV1,CaV2,CaV3,CCCF} - see \cite{GB} for the more detailed discussion of the latter connection.
To make the comparison more explicit, note that in a particular case of singular interaction kernel and a specific class of initial data the 1D Euler alignment
system reduces to a single equation for the density $\rho$ of the form
\begin{equation}\label{flock}
\partial_t \rho +  \partial_x \Lambda^{\alpha-2} \rho \partial_x \rho + \rho \Lambda^\alpha \rho =0,
\end{equation}
$0 < \alpha < 2$ (see \cite{DKRT,ST1,ST2}). A similar equation in the whole space setting has been studied in \cite{CaV1,CaV2,CaV3} as a model
of fluid flow in porous media.
Since $\partial_x \Lambda^{-1} \rho = - H \rho$, the parallel between \eqref{flock} and \eqref{maineq1} when $\alpha=1$ is clear. In particular,
both equations have critical scaling - the strength of dissipation exactly balances the advection term. Global regularity and exponential
convergence to uniform state for \eqref{flock} have been proved in \cite{DKRT,ST1,ST2}.
In this paper, we prove similar results in the case of \eqref{maineq1}.
However, there is a crucial difference between \eqref{maineq1} and \eqref{flock} that makes the proof far from simple extension.
The argument is fairly sophisticated due to the factor in the denominator of the nonlinear terms in \eqref{maineq1}.
It is not clear how to get a-priori bounds on $Hu$ that would allow to apply general H\"older regularity results in the spirit
of \cite{SS}. In addition, an attempt at direct application of the methods of \cite{DKRT,ST1,ST2} fails.
We still follow the general plan using the nonlocal maximum principle similar to the one used in \cite{DKRT},
building upon \cite{KNV}.
But to achieve the result, we have to meaningfully upgrade this approach compared to earlier works using variants of \cite{KNV}.

The first main result we prove here is as follows.
\begin{theorem}\label{mainthm1}
The equation \eqref{maineq} with $H^s,$ $s > 3/2$ periodic initial data such that $u_0(x)>0$ for all $x \in \S$ has a unique
global smooth solution $u(x,t).$ The $H^s$ norm of the solution is bounded uniformly in time, and all the derivatives of this solution
are bounded uniformly in time on any interval $[t_0,\infty)$, $t_0>0.$

Moreover, we have exponential in time convergence to equilibrium 
\begin{align}\label{meanconv1117}
\|u(\cdot, t) - \bar u \|_\infty \leq C_0 e^{-\sigma t},
\end{align}
and exponential in time decay of all derivatives
\begin{align}\label{higherder1117}
\|\partial^k_x u(\cdot, t)\|_\infty \leq C_k e^{- \sigma t}
\end{align}
for all integer $k \geq 1,$ with $\sigma = \frac{2}{\pi^2 \bar u}$ and with constants $C_k,$ $k=0,\dots$ that may only depend on $u_0.$
Here $\bar u$ denotes the mean of $u(x,t)$ that is conserved in time; \eqref{higherder1117} holds for all $t$ if $u_0$
has necessary regularity, and starting from any fixed $t_0>0$ otherwise.
\end{theorem}
Local regularity for the initial data $u_0 \in H^2(\S)$, global regularity with a certain smallness condition,
and exponential convergence to equilibrium in $L^\infty$ and in appropriate Wiener spaces under a smallness condition on the
initial data have been proved in \cite{GB}.

Our second major goal is to establish a rigorous connection between evolution of roots under differentiation and the PDE \eqref{maineq} in the periodic setting.
We will consider a class of trigonometric polynomials
\begin{align} p_{2n} (x) = \sum_{j=1}^{n} \left( a_j \cos jx + b_j \sin jx \right) = \prod_{j=1}^{2n} \sin \frac{x-x_j}{2} \end{align}
that will be assumed to have exactly $2n$ distinct roots $x_j \in \S$, $j=1 \dots, 2n.$
Note that by Rolle's theorem and a simple calculation that we outline in Section~\ref{derivation}, all derivatives of $p_{2n}$ belong in the same class.
Denote $\bar x_j = \frac{x_j+x_{j+1}}{2}$ the
midpoints of the gaps between the roots (we think here $x_j$ as angular coordinates of the roots).
We measure closeness between a discrete set
of roots $\{x_j\}_{j=1}^{2n}$ and a continuous distribution $u(x)$ in the following way.
Define the error
\begin{equation}\label{Ej1122}
  E_j=x_{j+1}-x_j-\frac{1}{2nu(\xb_j)},\quad j=1,\cdots, 2n.
\end{equation}
We may assume that initially, the roots of the polynomial $p_{2n}$ obey \eqref{Ej} for the initial density $u_0$
with some reasonably small errors. 
For the subsequent steps, we track  
\[ \|E^t\|_\infty = {\rm max}_j |E_j^t|, \]
with $t = \frac{k}{2n}$ and
\begin{align}\label{errtime1122}
E^t_j = x_{j+1}^{t} - x_j^t -\frac{1}{2nu(\xb_j^t,t)},\quad j=1,\cdots, 2n.
\end{align}
Here $x_j^t$ are the roots of the $k$th derivative of $p_{2n}$ and $\xb_j^t$ are the corresponding midpoints.

\begin{theorem}\label{mainthm2}
Let $u_0 \in H^s(\S),$ $s >7/2.$ Suppose that $u_0(x)>0$ for all $x \in \S,$ and $\int_\S u_0(x)\,dx =1.$
Let $u(x,t)$ be solution of \eqref{maineq} with the initial data $u_0,$ and let $p_{2n}$ be any trigonometric
polynomial that at the initial time obeys \eqref{Ej1122} with $u=u_0$ and $\|E^0\|_\infty \leq \M_0 n^{-1-\epsilon}$
for some $\epsilon>0.$
Then there exist positive constants $C(u_0)$ and $n_0(u_0,\M_0,\epsilon)$
such that if $n \geq n_0(u_0,\M_0,\epsilon),$ the following estimate holds true for times $t \geq 0:$
\begin{equation}\label{mainerrest1122}
\|E^{t}\|_\infty \leq C \left(\M_0 n^{-1-\epsilon} + n^{-3/2}(1-e^{-(\sigma -1)t})\right) e^{-t(1+O(n^{-\epsilon/2}))}.
\end{equation}
\end{theorem}
\begin{remark}
1. The largest initial errors $E^0$ that we can handle have size $\sim n^{-1-\epsilon}$ for arbitrary small $\epsilon>0,$ but the result applies for smaller errors:
$\epsilon$ can be large or $\M_0$
can be taken just zero for the initial perfect fit. \\
2. In the context of Theorem \ref{mainthm2}, we will sometimes refer to the triple $(u_0,\M_0,\epsilon)$ as the initial data, and 
use short cut notation $\tilde u_0$ for this triple, for example in $n_0(\tilde u_0) \equiv n_0 (u_0,\M_0,\epsilon).$ \\
3. The error $O(n^{-\epsilon/2})$ in the exponent in \eqref{mainerrest1122} can be replaced with $O((\log n)^2 n^{-\epsilon}).$
We choose the former form for the sake of simplicity.
\end{remark}

To the best of our knowledge, our work is the first one to rigorously connect evolution of polynomial roots under differentiation with a mean field partial differential equation.
The surprising aspect of Theorem~\ref{mainthm2} is that we are able to maintain control on the error for all times - and in fact, it is even improving past certain time 
that only depends on $u_0.$
The philosophical reason for this is that both evolution of roots and the solution $u(x,t)$ tend to uniform distribution, so they have no reason
to diverge for large times. However, there is quite a bit of distance between such observation and the estimate \eqref{mainerrest1122}. On a more detailed level, the
estimate \eqref{mainerrest1122} is enabled by careful analysis of the propagation of the error equation. Amazingly, it turns out to have form
\begin{align}\label{rougherrevol1122} \frac{E^{t+\Delta t}-E^t}{\Delta t} = {\mathcal L}^t E^t  + lower \,\,\,order \,\,\,terms, \end{align}
where ${\mathcal L}^t$ is a nonlinear operator of diffusive type that in the main order is similar to a modulated discretized fractional Laplacian $-\Lambda.$
In fact,
in the limit of large $n$ and large time ${\mathcal L}^t$ converges to exactly the dissipative term \[ -\frac{u \Lambda }{\pi(u^2 +Hu^2)} \sim -\frac{1}{\pi \bar u} \Lambda \] of \eqref{maineq1},
see Theorem~\ref{lem:dpositive} and remark after it for details.
Thus the propagation of the error equation turns out to be essentially a nonlinear fractional heat equation. 
The dissipative nature of \eqref{rougherrevol1122} is crucial for maintaining control of $\|E^t\|_\infty$ even for a finite time.
Global in time bound requires further ingredients, in particular favorable estimates on the leading lower order terms that take advantage
of the decay of derivatives \eqref{higherder1117}. Together, \eqref{rougherrevol1122} and results on the evolution of \eqref{maineq}
stated in Theorem~\ref{mainthm1} describe rigorously specific and delicate mechanisms that modulate the evolution
of roots of polynomials under differentiation.

There are many further natural questions.
We believe that our scheme of the proof can be useful in establishing the whole line result, as well - but there is an extra issue that one has to handle.
The assumption $u_0(x)> 0$ that actually implies $u_0(x) \geq a>0$ in the compact case is crucial for the proof of global regularity for \eqref{maineq}.
There is no reason to believe that this result holds in the whole line case when $u_0(x)$ has compact support - one would expect the solution to be just
H\"older regular near the edges. Thus one has to understand the associated free boundary problem to make progress in such situation (for analysis of the free boundary
for the equation \eqref{flock}, see \cite{CaV1,CaV2,CaV3}).
Another reasonable question is whether the current approach can be extended to more general sets of polynomials and some classes of entire functions.
Our approach shows how large scale $\sim 1$ or micro scale $\lesssim n^{-1-\epsilon}$ imbalances in root spacings are evened out, but does not apply to intermediate scale irregularities.
One can wonder whether repeated differentiation might eventually bring any trigonometric polynomial into the class that we handle here - but different methods and ideas
are needed to carefully analyze such conjecture.
Potential applications
to free probability and random matrices are also an exciting direction. In particular, looking further ahead, the structure of \eqref{rougherrevol1122}
hints at possibility of perhaps more intuitive continuous analogs for the free fractional convolutions. This may potentially have interesting applications
in a variety of directions in free probability and random matrix theory.

The paper is organized as follows. In Section~\ref{derivation}, we recall the formal derivation of \eqref{maineq}, recasting it for the periodic case.
In Section~\ref{local}, we establish local regularity, instant regularization property, and conditional regularity criteria for \eqref{maineq}.
Section~\ref{globreg} is dedicated to the proof of global regularity. In Section~\ref{asymptotic}, we derive estimates on convergence to the mean and
on decay of derivatives. In Sections~\ref{rigdersetup}, \ref{rigderoneroot} we set up and prove estimates on how a single root moves under differentiation,
essentially sharpening the estimates in all steps of the formal derivation. In Sections~\ref{rigdertworootssetup}, \ref{nearfieldtworoots} and \ref{farfieldtworoots}
we set up and prove estimates on how the pairs of neighboring roots move under differentiation, which is used for derivation of the error propagation equation.
In Section~\ref{properr} we further analyze evolution of the error, and prove that it takes form \eqref{rougherrevol1122}. In Section~\ref{proofT2}, we complete the proof
of Theorem~\ref{mainthm2}. 

\section{Formal derivation of the PDE}\label{derivation}

This section follows the original argument from \cite{Steinerberger2018}, recasting it for
trigonometric polynomials.
Recall that we consider the following space of periodic functions $\mathbb{P}_{2n}$
\[\Pm_{2n}=\left\{p_{2n}~\left|~
    \begin{split}&\exists~ \{a_j, b_j\}_{j=1}^n\in\R,
      ~~p_{2n}(x)=\sum_{j=1}^na_j\cos(jx)+b_j\sin(jx),\\
     &p_{2n}\text{ has
     }2n\text{ distinct roots }\{x_j\}_{j=1}^{2n}
     \text{ in }\S=(-\pi,\pi]\end{split}\right.
  \right\}.\]
The parameter $n$ is assumed to be large, and the distribution of the roots of $p_{2n}$ is
assumed to be close to a smooth function $u(x)$, which is
$2\pi$-periodic.

The derivative of a function in $\Pm_{2n}$ lies in
$\Pm_{2n}$ as well. It also has $2n$ distinct roots, one in each
interval $(x_j,x_{j+1})$. Let us denote
\[ \bar x_j = \frac{x_j +x_{j+1}}{2} \]
the midpoints of these intervals.

\begin{lemma}[An identity on roots and the derivative] Let $p_{2n}\in\Pm_{2n}$, and denote
  $\{x_j\}_{j=1}^{2n}$ be the roots of $p_{2n}$. Then, the following
  identity holds
  \begin{equation}\label{rootid}
  p_{2n}'(x)=\frac12~p_{2n}(x)\sum_{j=1}^{2n}\cot\frac{x-x_j}{2}.
  \end{equation}
\end{lemma}
\begin{proof}
It is not hard to show that a trigonometric polynomial $p_{2n} \in \Pm_{2n}$ that has roots at $2n$ distinct points
$x_1, \dots, x_{2n} \in (-\pi,\pi]$ satisfies
\[ p_{2n}(x) = c \prod_{j=1}^{2n} \sin \frac{x-x_j}{2}. \]
Direct differentiation yields \eqref{rootid}.
\end{proof}

Let $y_m\in (x_m, x_{m+1})$ be the roots of $p_{2n}'$,
$m=1,\cdots,2n$.
From the identity \eqref{rootid}, we know
\[\sum_{j=1}^{2n}\cot\frac{y_m-x_j}{2}=0.\]

Split the sum into two parts:
\[\sum_{j=1}^{2n}\cot\frac{y_m-x_j}{2}=
  \sum_{|x_j-y_m|\leq n^{-1/2}}\cot\frac{y_m-x_j}{2}+
   \sum_{|x_j-y_m|> n^{-1/2}}\cot\frac{y_m-x_j}{2}=I_m+II_m.\]

For the near field $I_m$, take the Taylor expansion
\[\cot\frac{y_m-x_j}{2}=\frac{2}{y_m-x_j}+O(|y_m-x_j|).\]
Then,
\[I_m=\sum_{|x_j-y_m|\leq n^{-1/2}}\left(\frac{2}{y_m-x_j}+O(|y_m-x_j|)\right)
  = \sum_{|x_j-y_m|\leq n^{-1/2}}\frac{2}{y_m-x_j}+O(1).\]

Recall the cotangent identity
\[ \pi \cot \pi x = \frac{1}{x} + \sum_{k=1}^\infty \left(\frac{1}{x+k}+\frac{1}{x-k}\right) \]
for $x \in \R \setminus \Z.$
For the first term, since the range of $\{x_j\}$ is small, and the
distribution of $\{x_j\}$ is close to $u(x):=u(x,0)$,
we can formally approximate $\{x_j\}$ by
equally distributed nodes $\{\tilde{x}_j\}$ centered at $x_m$, separated by
distance $(2n u(\xb_m))^{-1}$. Namely, by
\begin{equation}\label{xtilde}
  \tilde{x}_j=x_m+\frac{j-m}{2n u(\xb_m)}.
\end{equation}

Then, making use of the cotangent identity, we have
\begin{align*}
  \sum_{|x_j-y_m|\leq n^{-1/2}}\frac{2}{y_m-x_j}\sim&
  \sum_{k=-2u(\xb_m) n^{1/2}}^{2u(\xb_m) n^{1/2}}\frac{2}{y_m-x_m+k(2nu(\xb_m))^{-1}}\\
  \sim&~  4\pi n u(\xb_m)\cot(2\pi n u(\xb_m)(y_m-x_m)).
\end{align*}

For the far field $II_m$, as the distribution of $\{x_j\}$ is close to
$u(x)$, we formally get
\[II_m= \sum_{|x_j-y_m|>n^{-1/2}}\cot\frac{y_m-x_j}{2}\sim
  2n\int_{|y-y_m|>n^{-1/2}}u(y)\cot\frac{y_m-y}{2}\,dy  \sim 4\pi n
  Hu(y_m).\]

Putting together the two expressions,
the leading order
$O(n)$ term reads
\[  4\pi u(\xb_m)\cot(2\pi n u(\xb_m)(y_m-x_m))+4\pi Hu(y_m)=0.\]
Simplify the equation and get
\begin{equation}\label{formalupdate}
  y_m-x_m=-\frac{1}{2n \pi
    u(\xb_m)}\arctan\left(\frac{u(\xb_m)}{Hu(y_m)}\right).
\end{equation}
To make sure $y_m\in(x_m,x_{m+1})$, we take the branch of
 $\arctan x$ with values in $(-\pi,0).$ This branch is discontinuous at $x=0$
 but as we will see $u(x,t)/Hu(x,t)$ is bounded away from zero for all times in our setting.

Take the time scale $\Delta t= (2n)^{-1}$, so that $u(x,t=1)$ represents the
distribution of the roots of $p_{2n}^{(2n)}$, the $2n$-th derivative of $p_{2n}$.
Equation \eqref{formalupdate} provides a macroscopic flux
\[ v(x_m) =  -\frac{1}{\pi
    u(\xb_m)}\arctan\left(\frac{u(\xb_m)}{Hu(y_m)}\right) \]
of $x_m$ (roots of $p_{2n}$)
to $y_m$ (roots of $p_{2n}'$).
Letting $n\to\infty$, we formally derive
\begin{align}\label{maineq1122}
  0=\partial_tu+\partial_x(uv)=\partial_tu-\frac{1}{\pi}\partial_x\left(\arctan\left(\frac{u}{Hu}\right)\right)
  =\partial_tu+\frac{1}{\pi}\partial_x\left(\arctan\left(\frac{Hu}{u}\right)\right).
\end{align}
Note that here it does not matter which branch of $\arctan$ we take as the derivative is the same.
This is indeed our main equation.

\section{Local regularity, instant regularization, and global regularity criterion}\label{local}

In this section, we discuss the local well-posedness theory of
\eqref{maineq} in $H^s$, $s>3/2$, and derive global regularity criterion.

The local regularity of $H^2$ solution has been established in
\cite{GB} using energy estimates together with the uniform
boundedness of $u$ away from zero.
We will prove a similar result in general $H^s$ spaces, $s>3/2:$
\[H^s(\S)=\left\{f : \S\to\R ~\left|~ \|f\|_{H^s}^2:=\|f\|_{L^2}^2+\|f\|_{\dot{H}^s}^2<+\infty \right.\right\}.\]
Here, the homogeneous Sobolev semi-norm is defined by
\[\|f\|_{\dot{H}^s}^2:=\|\Lambda^s f\|^2_{L^2},\]
where $\Lambda^2f(x)=-f''(x)$, and for $s\in(0,2)$, the fractional
Laplacian is given by
\[\Lambda^sf(x) := c_s\int_\R\frac{f(x)-f(y)}{|x-y|^{1+s}}\,dy,\quad
   c_s=\frac{2^s\Gamma(\frac{1+s}{2})}{\sqrt\pi|\Gamma(-\frac{s}{2})|}.\]
Equivalently, $\|\cdot\|_{\dot{H}^s}$ can be defined through Fourier
series
\[\|f\|_{\dot{H}^s}^2:=2\pi\sum_{k\in\mathbb{Z}} |k|^{2s}|\hat{f}(k)|^2.\]

Let us first state our local wellposedness result.
\begin{theorem}[Local wellposedness]\label{thm:local}
  Let $s>3/2$. Consider equation \eqref{maineq} with $H^s$ periodic
  initial data such that $u_0(x)>0$ for all $x \in \S$.
  Then, there exists a finite time $T=T(s,u_0)>0$ and a unique classical solution $u(x,t)$
  such that
  \begin{align}
    u&\in C([0,T]; H^s(\S))\cap L^2([0,T], H^{s+1/2}(\S));\label{uHs}\\
    t^ku&\in C((0,T], H^{s+k}(\S)),\quad\forall~k\geq0.\label{tku}
  \end{align}
  Moreover, the $C^\infty$ solution described above can be extended beyond a time $T>0$ if
  and only if
  \begin{equation}\label{BKM}
    \int_0^T \|\pa_xu\|_{L^\infty} \big(1+\|\pa_xu\|_{L^\infty}^4\big)
    \,dt<+\infty.
  \end{equation}
\end{theorem}
\begin{remark}
The estimate for the norm $\|t^ku\|_{C((0,T], H^{s+k}(\S))}$ that we will obtain in this section will depend on $T,$
so it does not yield any decay of the Sobolev norms for large time. This will be proved later, in Section~\ref{asymptotic}.
\end{remark}

The rest of this section is devoted to the proof of Theorem \ref{thm:local}.
We focus on the a-priori estimates that can be used to establish Theorem~\ref{thm:local} in a standard way -
for example using Galerkin approximations (see \cite{KNS} for a similar setting) or limiting viscosity (see \cite{GB}).
To proceed, we shall introduce short cut notation
\begin{equation}\label{phipsi}
  \psi=\frac{Hu}{u^2+(Hu)^2},\quad
  \varphi=\frac{u}{u^2+(Hu)^2}.
\end{equation}
The equation \eqref{maineq1} can be expressed as
\begin{equation}\label{ueq}
  \pi\pa_tu=\left(\psi\pa_xu-\varphi\Lambda u\right).
\end{equation}

Let us denote $m(t)$ and $M(t)$ the minimum and maximum
values of $u(\cdot,t)$, respectively
\begin{equation}\label{uminmax}
  m(t)=\min_{x\in\S}u(x,t),\quad M(t)=\max_{x\in\S}u(x,t).
\end{equation}

The following maximum principle can be derived in a standard way (see for instance \cite{GB}).
\begin{proposition}[Maximum principle]\label{prop:mp}
  Suppose $u$ is a smooth solution (in the sense of \eqref{uHs}) of
  \eqref{ueq}. Then,
\[m_0\leq m(t)\leq M(t)\leq M_0,\quad\forall~t\in[0,T].\]
Here $m_0=m(0)$ and $M_0=M(0).$
\end{proposition}
An improved estimate will be obtained later in Section \ref{expconv},
where it is shown that \[ M(t),m(t) \rightarrow \frac{1}{2\pi}\int_{\S} u(x,t)\,dx = \bar u \] and
$V(t):=M(t)-m(t)$ decays exponentially in time.

The maximum principle directly implies boundedness of
\[\|u(\cdot,t)\|_{L^2}\leq\sqrt{2\pi}M_0,\quad \forall~t\geq0.\]
Therefore, we only need to focus on the control of $\|u\|_{\dot{H}^s}$.

\subsection{Local wellposedness}
We perform a standard energy estimate. Applying $\Lambda^s$ to the
equation \eqref{ueq}, multiplying by $\Lambda^su$ and
integrating in $x$ we obtain
\[\frac{\pi}{2}\frac{d}{dt}\|u\|_{\dot{H}^s}^2=
  \int_\S \Lambda^su \cdot \Lambda^s(\psi\pa_xu)\,dx
  -\int_\S \Lambda^su \cdot \Lambda^s(\varphi\Lambda u)\,dx=I+II.\]
Throughout most of this section, we drop time dependence in notation to simplify presentation.

Let us treat the two terms separately. We split the transport term $I$ as follows:
\[I=\int_\S \Lambda^su \cdot \psi \cdot \Lambda^s(\pa_xu)\,dx+
  \int_\S \Lambda^su \cdot [\Lambda^s, \psi]\pa_xu\,dx=I_1+I_2.\]
$I_1$ can be bounded by standard integration by parts,
\[|I_1|=\left|\int_\S\pa_x\left(\frac12 (\Lambda^su)^2\right)\cdot\psi\,dx\right|
  =\left|-\frac12\int_\S\pa_x\psi\cdot(\Lambda^su)^2\,dx\right|
  \leq\frac{1}{2}\|\pa_x\psi\|_{L^\infty}\|u\|_{\dot{H}^s}^2.\]
$I_2$ can be bounded by the Kato-Ponce type commutator estimate \cite{KP,KPV}, see \cite[Lemma 2.5]{KenigP} for the periodic case:
\[|I_2|\leq\|u\|_{\dot{H}^s}\|[\Lambda^s, \psi]\pa_xu\|_{L^2}
  \lesssim\|u\|_{\dot{H}^s}\left(\|\pa_x\psi\|_{L^\infty}\|\Lambda^{s}u\|_{L^2}
    +\|\Lambda^s\psi\|_{L^2}\|\pa_xu\|_{L^\infty}\right).\]
Next, $\|\pa_x\psi\|_{L^\infty}$ can be estimated by
\begin{equation}\label{eq:psix}
  |\pa_x\psi|\leq\frac{|\pa_xHu|}{u^2+(Hu)^2}+\frac{|Hu|\cdot\left(2u|\pa_xu|+2|Hu|\cdot|\pa_xHu|\right)}{(u^2+(Hu)^2)^2}
  \leq\frac{|\pa_xu|+3|\Lambda u|}{m_0^2}.
\end{equation}
To estimate $\|\psi\|_{\dot{H}^s}$, we apply the fractional Leibniz rule (see e.g. \cite{GO,MS}):
\begin{align}\|\psi\|_{\dot{H}^s}\lesssim&\,
  \|Hu\|_{\dot{H}^s}\left\|\frac{1}{u^2+(Hu)^2}\right\|_{L^\infty}+
  \|Hu\|_{L^\infty}\left\|\frac{1}{u^2+(Hu)^2}\right\|_{\dot{H}^s}\nonumber\\
  \leq&\, \frac{1}{m_0^2}\|u\|_{\dot{H}^s}+\|Hu\|_{L^\infty}\left\|\frac{1}{u^2+(Hu)^2}\right\|_{\dot{H}^s}.\label{psiHsest}
\end{align}
For $\left\|\frac{1}{u^2+(Hu)^2}\right\|_{\dot{H}^s}$, we state the
following composition estimate, and give a sketch of the proof for the sake of completeness.
\begin{lemma}\label{lem:composition}
 Let $s\geq1$.  Suppose $f\in H^s(\S)$ is a positive function with
  $\min_x f(x)=f_{\min}>0$.
  There exists a constant $C=C(s,f_{\min})$ such that
  \begin{equation}\label{composition}
    \left\|\frac{1}{f}\right\|_{\dot{H}^s}
    \leq C(s,f_{\min},a) \left(1+\|f\|_{L^\infty}^{k-1}\right)
    \Big(\|f\|_{\dot{H}^s}+\|f\|_{C^{\gamma+a}}\|f\|_{\dot{H}^k}\Big),
  \end{equation}
  where we denote $k$ the integer part of $s$,
  $\gamma=s-k\in[0,1)$, $a>0$ and $\gamma+a<1.$
\end{lemma}
\begin{proof}
  We first show \eqref{composition} when $s=k$ is an integer.
  Apply the following identity on derivatives of reciprocal functions
\[\frac{d^k}{dx^k}\left(\frac{1}{f(x)}\right)=
  \sum_{j=1}^k(-1)^j{k+1 \choose j+1}\frac{1}{f(x)^{j+1}}\frac{d^k}{dx^k}\big(f(x)^j\big),\]
and the uniform lower bound $f\geq f_{\min}$ to get
\begin{align}
  \left\|\frac{1}{f}\right\|_{\dot{H}^s}=&\,
  \left\|\frac{d^s}{dx^s}\left(\frac{1}{f}\right)\right\|_{L^2}\leq \sum_{j=1}^s C(j, s, m_0)\cdot
  \|f^j\|_{\dot{H}^{s}}\\ \leq&\, \sum_{j=1}^s C(j, s, f_{\min})\cdot \|f\|_{L^\infty}^{j-1}\|f\|_{\dot{H}^{s}}\leq C(s, f_{\min})(1+\|f\|_{L^\infty}^{s-1})\|f\|_{\dot{H}^s}.\label{aux4127}
\end{align}
Here, in the second inequality, we have also used the Leibniz rule \cite{GO,MS}
\begin{equation}\label{powerLeib}
  \|f^j\|_{\dot{H}^{s}}=\left\|\prod_{l=1}^jf\right\|_{\dot{H}^{s}}\lesssim
  \|f\|_{\dot{H}^s}\prod_{l=1}^{j-1}\|f\|_{L^\infty}=\|f\|_{L^\infty}^{j-1}\|f\|_{\dot{H^s}}.
\end{equation}

If $s$ is not an integer, we can deploy a version of the commutator estimate \cite[Eq. (3.5)]{DKRT}):
\begin{align*}
  \left\|\frac{1}{f}\right\|_{\dot{H}^s}=&\,
  \left\|\frac{d^k}{dx^k}\left(\frac{1}{f}\right)\right\|_{\dot{H}^\gamma}\leq
 \sum_{j=1}^k C(j, k)\cdot\left\|\frac{1}{f^{j+1}}\frac{d^k}{dx^k}\big(f^j\big)\right\|_{\dot{H}^\gamma}\\ \leq&\,
\sum_{j=1}^k C(j, k,a)\cdot\left(\left\|\frac{1}{f^{j+1}}\right\|_{L^\infty}\|f^j\|_{\dot{H}^s}+
 \left\|\frac{1}{f^{j+1}}\right\|_{C^{\gamma+a}}\|f^j\|_{\dot{H}^k}\right).
\end{align*}
The first part can be treated in the same manner as in \eqref{aux4127}.

For the second part, observe that
\[
  \left\|\frac{1}{f^{j+1}}\right\|_{C^{\gamma+a}}
  \leq \frac{j+1}{(\min f)^{j+2}}\|f\|_{C^{\gamma+a}}.
\]
Then we have
\[\left\|\frac{1}{f^{j+1}}\right\|_{C^{\gamma+a}}\|f^j\|_{\dot{H}^k}\leq
C(s,f_{\min}) \|f\|_{C^{\gamma+a}} \|f\|_{L^\infty}^{j-1}\|f\|_{\dot{H}^k}.\]
Collecting all estimates, we obtain
\eqref{composition}.
\end{proof}

Now, we apply Lemma \ref{lem:composition} with $f=u^2+Hu^2$. Compute
\begin{align*}
  \|f\|_{L^\infty}\leq&\, M_0^2+\|Hu\|_{L^\infty}^2,\\
  \|f\|_{C^{\gamma+a}}\leq&\,
  2\|u\|_{L^\infty}\|u\|_{C^{\gamma+a}}
  +2\|Hu\|_{L^\infty}\|Hu\|_{C^{\gamma+a}}
\lesssim (M_0+\|Hu\|_{L^\infty})\|\pa_xu\|_{L^\infty},
 \end{align*}
and use fractional Leibniz rule
\[\|f\|_{\dot{H}^s}=\|u^2+(Hu)^2\|_{\dot{H}^s}
  \lesssim (M_0+\|Hu\|_{L^\infty})\|u\|_{\dot{H}^s}.\]
We arrive at the bound
\begin{equation}\label{HsuHu}
  \left\|\frac{1}{u^2+(Hu)^2}\right\|_{\dot{H}^s}\leq
  C(s, m_0,M_0)\left(1+\|Hu\|_{L^\infty}^{2s}\right)(1+\|\pa_xu\|_{L^\infty})\|u\|_{\dot{H}^s}.
\end{equation}

Then, we have
\[\|\psi\|_{\dot{H}^s}\lesssim
  \left(1+\|Hu\|_{L^\infty}^{2s+1}\right)(1+\|\pa_xu\|_{L^\infty}) \|u\|_{\dot{H}^s}.\]
Putting the estimates together, we get the bound on $I$
\begin{equation}\label{eq:est1}
  |I|\lesssim
  \left[\|\pa_xu\|_{L^\infty}\left(1+\|\pa_xu\|_{L^\infty}\right)
    \left(1+\|Hu\|_{L^\infty}^{2s+1}\right)+\|\Lambda
  u\|_{L^\infty}\right]\|u\|_{\dot{H}^s}^2.
\end{equation}

\medskip
Next, for the dissipation term $II$, we have
\[II=-\int_\S
  \Lambda^{s+\frac12}u\cdot\Lambda^{s-\frac12}(\varphi\Lambda u)\,dx
  =-\int_\S\varphi\cdot(\Lambda^{s+\frac12}u)^2\,dx
  -\int_\S\Lambda^{s+\frac12}u\cdot[\Lambda^{s-\frac12},\varphi]\Lambda u\,dx.
\]
Note that $\varphi$ satisfies the following lower bound
\[\varphi=\frac{u}{u^2+(Hu)^2}\geq\frac{m_0}{M_0^2+\|Hu\|_{L^\infty}^2}.\]
Then,
\begin{align*}
  II\leq&
  -\frac{m_0}{M_0^2+\|Hu\|_{L^\infty}^2}\|u\|_{\dot{H}^{s+\frac12}}^2
  +\|u\|_{\dot{H}^{s+\frac12}}\|[\Lambda^{s-\frac12},\varphi]\Lambda u\|_{L^2}\\
  \leq&  -\frac{m_0}{2(M_0^2+\|Hu\|_{L^\infty}^2)}\|u\|_{\dot{H}^{s+\frac12}}^2
  + \frac{M_0^2+\|Hu\|_{L^\infty}^2}{2m_0}\|[\Lambda^{s-\frac12},\varphi]\Lambda u\|_{L^2}^2.
\end{align*}
Let us apply the same commutator estimate \cite{KenigP} (using $\Lambda u = \partial_x Hu$):
\[\|[\Lambda^{s-\frac12},\varphi]\Lambda u\|_{L^2}\lesssim
  \|\pa_x\varphi\|_{L^\infty}\|u\|_{\dot{H}^{s-\frac12}}+\|\varphi\|_{\dot{H}^{s-\frac12}}\|\Lambda u\|_{L^\infty}.
\]
Similarly to the treatment of $\psi$, we can estimate $\|\pa_x\varphi\|_{L^\infty}$ via
\begin{equation}\label{eq:phix}
  |\pa_x\varphi|\leq\frac{|\pa_xu|}{u^2+(Hu)^2}+\frac{u\cdot\left(2u|\pa_xu|+2|Hu|\cdot|\pa_xHu|\right)}{(u^2+(Hu)^2)^2}
  \leq\frac{3|\pa_xu|+|\Lambda u|}{m_0^2}.
\end{equation}
Let us use fractional Leibniz rule and \eqref{HsuHu} to estimate $\|\varphi\|_{\dot{H}^{s-\frac12}}:$
\[\|\varphi\|_{\dot{H}^{s-\frac12}}\lesssim
  \frac{1}{m_0^2}\|u\|_{\dot{H}^{s-\frac12}}+
  M_0 \left\|\frac{1}{u^2+(Hu)^2}\right\|_{\dot{H}^{s-\frac12}}
  \lesssim \left(1+\|Hu\|_{L^\infty}^{2s-1}\right)(1+\|\pa_xu\|_{L^\infty})\|u\|_{\dot{H}^{s-\frac12}}.\]
Then
\[\|[\Lambda^{s-\frac12},\varphi]\Lambda u\|_{L^2}\lesssim
 \Big[\|\pa_xu\|_{L^\infty}+\|\Lambda u\|_{L^\infty}(1+\|\pa_xu\|_{L^\infty})\left(1+\|Hu\|_{L^\infty}^{2s-1}\right)\Big]\|u\|_{\dot{H}^{s-\frac12}}.
\]
We end up with the estimate on $II$
\begin{align}\label{eq:est2}
  II\leq&
-\frac{m_0}{2(M_0^2+\|Hu\|_{L^\infty}^2)}\|u\|_{\dot{H}^{s+\frac12}}^2
\\
 +&\,C(s,m_0,M_0)(1+\|Hu\|_{L^\infty}^2)\Big[\|\pa_xu\|_{L^\infty}^2+\|\Lambda u\|_{L^\infty}^2(1+\|\pa_xu\|_{L^\infty}^2)\left(1+\|Hu\|_{L^\infty}^{4s-2}\right)\Big]
\|u\|_{\dot{H}^{s-\frac12}}^2.\nonumber
\end{align}

\medskip
Combining the estimates \eqref{eq:est1} and \eqref{eq:est2},
we obtain the differential inequality on evolution of the $\dot{H}^s$ norm of solution
\begin{align}\label{eq:uHs}
  &\frac{\pi}{2}\frac{d}{dt}\|u\|_{\dot{H}^s}^2+\frac{m_0}{2(M_0^2+\|Hu\|_{L^\infty}^2)}\|u\|_{\dot{H}^{s+\frac12}}^2\\
  &\lesssim
  \Big[\|\pa_xu\|_{L^\infty}\left(1+\|\pa_xu\|_{L^\infty}\right)
    \left(1+\|Hu\|_{L^\infty}^{2s+1}\right)+\|\Lambda
  u\|_{L^\infty}\Big]\|u\|_{\dot{H}^s}^2.
 \nonumber\\
&\quad+(1+\|Hu\|_{L^\infty}^2)\Big[\|\pa_xu\|_{L^\infty}^2+\|\Lambda u\|_{L^\infty}^2(1+\|\pa_xu\|_{L^\infty}^2)\left(1+\|Hu\|_{L^\infty}^{4s-2}\right)\Big]\|u\|_{\dot{H}^{s-\frac12}}^2. \nonumber
\end{align}
For $s>3/2$, $\|Hu\|_{L^\infty}, \|u_x\|_{L^\infty}, \|\Lambda
u\|_{L^\infty}$ and $\|u\|_{\dot{H}^{s-\frac12}}$ are all controlled
by $\|u\|_{\dot{H}^s}$.
Hence, we have
\[\frac{d}{dt}\|u\|_{\dot{H}^s}^2\lesssim(1+\|u\|_{\dot{H}^s}^{4s+3})\|u\|_{\dot{H}^s}^3\]
(we do not aim for the best exponent right now).
This a-priori estimate can be used in a standard manner to derive that
the solution exists and is unique locally
in time, namely there exists a time $T>0$ such that $u\in C([0,T], H^s(\S))$ provided that $u_0 \in H^s(\S).$

Moreover, integrating \eqref{eq:uHs} in $[0,T]$, we get
\begin{align*}\int_0^{T}\|u(\cdot,t)\|_{\dot{H}^{s+\frac12}}^2\,dt
\leq&~
      \frac{2}{m_0}\left(M_0^2+\sup_{t\in[0,T]}\|Hu(\cdot,t)\|_{L^\infty}^2\right)\\
  &\times~\left(\frac{\pi}{2}\|u_0\|_{\dot{H}^s}^2+C\int_0^{T}\left(1+\|u(\cdot,t)\|_{H^s}^{4s+3}\right)\|u(\cdot,t)\|_{H^s}^3\,dt\right)<+\infty.
\end{align*}
So, $u\in L^2([0,T], H^{s+1/2}(\S))$. This completes the proof of
\eqref{uHs}.

\subsection{Instant regularization}

It is a common feature of dissipative differential equations that the solutions
instantly become $C^\infty,$ even if the initial data has limited regularity.
This can be expressed in terms of \eqref{tku}.
We include a proof here, which deploys an argument similar to
\cite[Theorem 2.4]{KNS}.

We show \eqref{tku} for all half integers $k$ by induction.
Again, we focus on the a-priori estimates, which are done on smooth functions such as Galerkin
approximations or viscous regularizations. The bounds are then inherited by the solutions via limiting procedure.
The $k=0$ case is exactly the local well-posedness result in
\eqref{uHs}.
Inductively, assume that $t^{k} u\in C((0,T], H^{s+k}(\S))$.
Replacing $s$ by $s+k$ in the energy estimate \eqref{eq:uHs} (but still using $\|u\|_{\dot{H}^s}$ to control lower order norms),
using interpolation inequality and Young's inequality, we get
\begin{align*}
  &\frac{\pi}{2}\frac{d}{dt}\|u\|_{\dot{H}^{s+k}}^2+\frac{m_0}{2(M_0^2+\|Hu\|_{L^\infty}^2)}\|u\|_{\dot{H}^{s+k+\frac12}}^2 \lesssim
    \left(1+\|u\|_{\dot{H}^s}^{4(s+k)+3}\right)\|u\|_{\dot{H}^s}
    \|u\|_{\dot{H}^{s+k}}^2\\
  &\lesssim
    \left(1+\|u\|_{\dot{H}^s}^{4(s+k)+3}\right)\|u\|_{\dot{H}^s}
    \|u\|_{\dot{H}^s}^{2(1-\alpha)}\|u\|_{\dot{H}^{s+k+\frac12}}^{2\alpha},\qquad
  \alpha=\frac{2k}{2k+1}\\
  &\leq\frac{\alpha m_0}{4(M_0^2+\|Hu\|_{L^\infty}^2)}\|u\|_{\dot{H}^{s+k+\frac12}}^2\\
  & \qquad +(1-\alpha)\left[\left(1+\|u\|_{\dot{H}^s}^{4(s+k)+3}\right)\|u\|_{\dot{H}^s}\right]^{\frac{1}{1-\alpha}}\left[\frac{4(M_0^2+\|Hu\|_{L^\infty}^2)}{m_0}\right]^{\frac{\alpha}{1-\alpha}}\|u\|_{\dot{H}^s}^2.
\end{align*}
This implies
\begin{equation}\label{spk}
  \frac{\pi}{2}\frac{d}{dt}\|u\|_{\dot{H}^{s+k}}^2+\frac{m_0}{4(M_0^2+\|Hu\|_{L^\infty}^2)}
  \|u\|_{\dot{H}^{s+k+\frac12}}^2\lesssim F(\|u\|_{\dot{H}^s}),
\end{equation}
where the nonlinear function $F$ can be chosen as
\[F(r)=\big((1+r)^{4(s+k)+3}r\big)^{2k+1}(1+r^2)^{2k}r^2,\]
which is bounded if $r$ is bounded. Therefore, from the local
wellposedness theory, $F(\|u(\cdot,t)\|_{\dot{H}^s})$
is uniformly bounded for any $t\in[0,T]$.

Now, let us fix any $t\in(0,T]$, and consider the interval $I=(t/2,t)$.
Integrating \eqref{spk} in the interval $I$, we obtain
\begin{align*}\int_{t/2}^t\|u(\cdot,\tau)\|_{\dot{H}^{s+k+\frac12}}^2\,d\tau
\leq&~
      \frac{4}{m_0}\left(M_0^2+\sup_{\tau\in[t/2,t]}\|Hu(\cdot,\tau)\|_{L^\infty}^2\right)\times\\
    &\times~\left(\|u(\cdot,t/2)\|_{\dot{H}^{s+k}}^2+C\int_{t/2}^tF(\|u(\cdot,\tau)\|_{\dot{H}^s})\,d\tau\right)\\
  \leq&~C\|u(\cdot,t/2)\|_{\dot{H}^{s+k}}^2+Ct\leq Ct^{-2k},
\end{align*}
where the induction hypotheses is used in the last inequality.
Here, we use $C$ to represent constants which only depend on $s, k,
m_0, M_0$ and $\|u_0\|_{H^s}$, and they could change from line to line.

Using mean value theorem,
we can find $\tau\in I$ such that
\[\|u(\cdot,\tau)\|_{\dot{H}^{s+k+\frac12}}^2\leq C|I|^{-1}t^{-k}\leq Ct^{-2k-1}.\]

Replacing $s+k$ in \eqref{spk} by $s+k+\frac12$ and integrating over $[\tau,t]$ yields
\begin{align}\label{aux7127} \|u(\cdot,t)\|_{\dot{H}^{s+k+\frac12}}^2\leq
  \|u(\cdot,\tau)\|_{\dot{H}^{s+k+\frac12}}^2+Ct\leq Ct^{-2k-1}.\end{align}
Hence, we obtain $t^{k+\frac12}u\in C((0,T],
H^{s+k+\frac12}(\S))$. This concludes the induction step. Note that the constant $C$ in \eqref{aux7127}
depends on $T;$ thus we do not get decay of the Sobolev norms in time in this estimate, only instant regularization.

Non-half-integer $k$ can be covered by interpolation
\[\|u(\cdot,t)\|_{\dot{H}^{s+k}}\leq
  \|u(\cdot,t)\|_{\dot{H}^s}^{1-\frac{k}{k_0}}\|u(\cdot,t)\|_{\dot{H}^{s+k_0}}^{\frac{k}{k_0}},\]
where $k_0$ is the smallest half-integer that is larger than $k$.

\subsection{Conditional global regularity criterion}
We now derive a sufficient and necessary condition \eqref{BKM} that
ensures the global existence and smoothness of solutions.

We start with the energy estimate \eqref{eq:uHs}, and simplify it.
Let us first estimate the Hilbert transform $Hu$.
\begin{lemma}\label{Katolem}
Suppose that $u$ belongs to H\"older class $C^\alpha(\S)$ with some $\alpha>0.$
Then
\begin{align}\label{Kato127}
  \|Hu\|_{L^\infty}\leq C(\alpha)\|u\|_{L^\infty}\left(1+\log_+\frac{\|u\|_{C^\alpha}}{\|u\|_{L^\infty}}\right) 
\end{align}
The norm $\|u\|_{C^\alpha}$ can be replaced by $\|\pa_x u\|_\infty$ or, by the imbedding theorem, $\|u\|_{\dot{H}^s}$ with any $s>1/2.$
\end{lemma}
\begin{proof}
This result is well known as Kato inequality \cite{Kato} (see e.g. \cite{KS1}). We provide a sketch of the proof for completeness. Write
\begin{align*}
  |Hu(x)| =&\, \frac{1}{2\pi}\left| \int_{|y|\leq R} (u(x-y)-u(x))\cot \frac{y}{2}\,dy + \int_{R<|y|\leq\pi} u(x-y) \cot \frac{y}{2}\,dy \right| \\
\leq&\, \frac{C}{\alpha}\|u\|_{C^\alpha} R^{\alpha} + C\|u\|_{L^\infty} \log \frac{\pi}{R}.
\end{align*}
Choosing $R=\min \left\{1, \left(\frac{\|u\|_{L^\infty}}{\|u\|_{C^\alpha}} \right)^{1/\alpha}\right\}$ completes the proof.
\end{proof}
Since in our case $\|u\|_{L^\infty} \geq m_0>0,$ we can omit it in the denominator under the $\log_+$ adjusting the constant $C=C(\alpha,m_0)$ if necessary.

As a consequence, we have the bound
\begin{equation}\label{Hubounds}
  \|Hu\|_{L^\infty}^{p}\lesssim
  M_0^{p}\big(1+\log_+\|\pa_xu\|_{L^\infty}\big)^{p}\leq
  C(s,M_0,p)(1+\|\pa_xu\|_{L^\infty}),
\end{equation}
for every $0<p <\infty$. We could use smaller power of $\|\pa_x u\|_{L^\infty}$ on the right but choose linear bound for simplicity.

The Kato inequality also implies that
\begin{align}\label{lambdauest}
\|\Lambda u \|_{L^\infty} \lesssim \|\pa_x u\|_{L^\infty} (1+\log_+\|u\|_{\dot{H}^s}^2),
\end{align}
with the constant involved in $\lesssim$ that depends only on $s>3/2.$

Finally, we also have
\begin{align}\label{aux128}
 \|u\|_{\dot{H}^{s-1/2}}\lesssim&\,
  \|u\|_{\dot{H}^s}^{1-\frac{1}{2s}}~\|u\|_{L^2}^{\frac{1}{2s}}
  \leq (\sqrt{2\pi}M_0)^{\frac{1}{2s}} \|u\|_{\dot{H}^s}^{1-\frac{1}{2s}}.
\end{align}

Substituting \eqref{Kato127}, \eqref{Hubounds}, \eqref{lambdauest} and \eqref{aux128} into \eqref{eq:uHs}, we obtain
\begin{align}
& \frac{d}{dt}\|u\|_{\dot{H}^s}^2 \lesssim \left(\|\pa_x u \|_{L^\infty} (1+\|\pa_x u \|_{L^\infty})^2 +\|\pa_x u\|_{L^\infty}
(1+\log_+\|u\|_{\dot{H}^s}^2)\right)\|u\|_{\dot{H}^s}^2 \\
& +\left((1+\|\pa_xu\|_{L^\infty}) \|\pa_x u\|_{L^\infty}^2 \left(1+(1+\|\pa_x u\|_{L^\infty})^2 (1+\log_+\|u\|_{\dot{H}^s}^2)^2\right)\right) \|u\|_{\dot{H}^s}^{2-\frac{1}{s}} \\
& \lesssim \|\pa_x u\|_{L^\infty} (1+\|\pa_x u\|_{L^{\infty}}^3)(1+\log_+\|u\|_{\dot{H}^s}^2)\|u\|_{\dot{H}^s}^2, \label{grc128}
\end{align}
where we used that
\[ (1+\log_+\|u\|_{\dot{H}^s}^2)^2 \|u\|_{\dot{H}^s}^{-\frac{1}{s}} \lesssim 1. \]
The bound \eqref{grc128} leads to
\begin{equation}\label{eq:doubleexpbound}
  \sup_{0\leq t\leq T}\|u(\cdot,t)\|_{\dot{H}^s}^2\leq \big(e\max\{1,\|u_0\|_{\dot{H}^s}\}\big)^{\exp\left[C\int_0^T \|\pa_xu\|_{L^\infty} \left(1+\|\pa_xu\|_{L^\infty}^4\right)\,dt\right]}.
\end{equation}
Hence, for every $s>3/2,$ $H^s$ regularity can be extended past $T$ using local well-posedness if 
\eqref{BKM} is satisfied.

\section{Global regularity}\label{globreg}

Throughout this section, we will denote by $c,$ $c_0,$ $c_1,$ etc and $C,$ $C_0,$ $C_1,$ etc the constants that do not depend on any parameters of the problem with the possible exception of $m_0$ and $M_0$;
these constants may change from expression to expression.
Our goal is to prove the global regularity part of Theorem~\ref{mainthm1}. We use the modulus of continuity method pioneered in \cite{KNV}.
We will prove that the periodic smooth solutions of \eqref{maineq1} conserve the modulus of continuity $\omega(\xi)$
that is explicitly given by
\begin{equation}\label{modcondef}
\omega(\xi) = \left\{ \begin{array}{ll} \xi-\xi^{3/2} & 0 < \xi \leq \delta \\
\delta - \delta^{3/2} + \gamma \log \left( 1+ \frac14 \log (\xi/\delta) \right) & \xi >\delta.
\end{array} \right.
\end{equation}
Here $\delta$ and $\gamma$ are small parameters that will be determined later in the proof.
We will assume throughout the argument that $\gamma = c\delta \leq \delta \ll 1,$ where $c<1$ and may only depend on $m_0$ and $M_0;$
$\delta$ may also depend only on $m_0$ and $M_0.$
More specific conditions on these parameters will be outlined in due course.
The key properties of $\omega$ is that it is increasing, concave, Lipshitz at $\xi=0$ with $\omega'(0)=1,$ $\omega''(0)=-\infty,$
and satisfies $\omega(\xi) \rightarrow \infty$ as $\xi \rightarrow \infty.$
We say that a function $f(x)$ obeys $\omega$ if $|f(x)-f(y)| \leq \omega(|x-y|)$ for all $x, y \in \Rm.$
We say that a solution $u(x,t)$ preserves the modulus $\omega$ if $u(x,t)$ obeys $\omega$ for every $t \geq 0$
provided that the initial data $u_0(x)$ obeys $\omega.$

Let us think of \eqref{maineq1} as set on the whole line with the periodic initial condition; in the case of the period $2\pi$ the problem
is equivalent to our $\S$ setting. Of course, the local regularity results of the previous section apply to any fixed period.
Notice that the equation \eqref{maineq1} in the whole line setting is invariant under scaling $x \mapsto \lambda x,$ $t \mapsto \lambda t$
for any $\lambda >0,$ so if $u(x,t)$ is a periodic solution of \eqref{maineq1}, then so is $u_\lambda (x,t) \equiv u(\lambda x, \lambda t)$
(with a different period). We have

\begin{proposition}\label{grmc}
Suppose that any periodic solution $u(x,t)$ of the equation \eqref{maineq1} corresponding to $H^s,$ $s>3/2$ initial data preserves the modulus of continuity $\omega$
defined in \eqref{modcondef}. Then, $\|\pa_xu(\cdot,t)\|_{L^\infty}$
is  bounded uniformly in time.
\end{proposition}
Note that the global regularity part of Theorem~\ref{mainthm1} follows from Proposition~\ref{grmc} due to conditional regularity criterion \eqref{BKM}.
\begin{proof}
The proof can be found in \cite{KNV}; we sketch it here for the sake of completeness. Given an arbitrary $H^s,$ $s>3/2$ initial data $u_0(x)$ with period $2\pi,$
we can find $\lambda >0$ such that $u_0(\lambda x)$ obeys $\omega.$ This follows from the fact that $\omega(\xi) \rightarrow \infty$
as $\xi \rightarrow \infty,$ and a simple argument. In fact the necessary size of $\lambda$ can be estimated explicitly in terms of $\omega$
and the initial data:
it is sufficient to require that
\[\omega \left( \frac{M_0-m_0}{\lambda\|\partial_x u_0\|_{L^\infty}} \right) >
   M_0-m_0,\quad\text{or equivalently,}~~\lambda <
   \frac{M_0-m_0}{\|\partial_x u_0\|_{L^\infty}\cdot\omega^{-1}(M_0-m_0)}.\]
From the assumption of the proposition, it follows that the
solution $u(\lambda x, \lambda t)$
preserves $\omega.$ But this is equivalent to saying that solution
$u(x,t)$ preserves $\omega_{\lambda^{-1}}(\xi) \equiv
\omega(\lambda^{-1}\xi).$
Since $\omega$ is Lipschitz near the origin, it forces $\|\pa_x
u(\cdot,t)\|_{L^\infty} \leq \lambda^{-1}$ for all times.
\end{proof}

 If we can prove the conservation of modulus $\omega$ for solutions of any period,
 a consequence of the proof of Proposition \ref{grmc} is a uniform in time bound on $\|\pa_xu(\cdot,t)\|_{L^\infty}:$
  \begin{equation}\label{uxbound}
    \sup_{t\geq0}\|\pa_xu(\cdot,t)\|_{L^\infty}\leq
    2\frac{\|\partial_x u_0\|_{L^\infty}\cdot\omega^{-1}(M_0-m_0)}{M_0-m_0}=:\Q_1. 
  \end{equation}
  Moreover, it will be useful for us to note that $\|Hu(\cdot,t)\|_{L^\infty}$ is also uniformly
  bounded in time. Indeed, in the periodic case
  \begin{align}\label{Ht1124}
  Hu(x) = \frac{1}{\pi} P.V. \int_\R \frac{u(y)}{x-y}\,dy = \frac{1}{2\pi} \int_{-\pi}^\pi u(y) \cot \frac{x-y}{2}\,dy.
  \end{align}
   Then
  \begin{equation}\label{Bbound}
    |Hu(x,t)|\leq
    \frac{\|\pa_xu(\cdot,t)\|_{L^\infty}}{\pi}\int_0^\pi
    y\cot\frac{y}{2}\,dy<2 \|\pa_xu(\cdot,t)\|_{L^\infty}\leq
    \frac{4\|\partial_x u_0\|_{L^\infty}\cdot\omega^{-1}(M_0-m_0)}{M_0-m_0}=:Q.
  \end{equation}
  Note that $Q=2\Q_1,$ but it is convenient for us to have a short hand notation for this bound.
  Clearly we could have used any H\"older norm $\|u\|_{C^\alpha}$, $\alpha>0,$ to control $\|Hu\|_{L^\infty}$ in \eqref{Bbound}.

Thus to show global regularity, it is sufficient to prove preservation of the modulus of continuity. The following lemma reduces this
to consideration of a breakthrough scenario.
\begin{lemma}\label{break}
Suppose that the periodic $H^s,$ $s>3/2$ initial data $u_0(x)$ obeys $\omega$ but the solution $u(x,t)$ no longer obeys it at some time $t>0.$
Then there exists a time $t_1>0$ and two points $x$ and $y$ such that for all $t \leq t_1$ the solution $u(x,t)$ obeys $\omega,$ and
\begin{equation}\label{breakeq}
u(x,t_1)- u(y,t_1)=\omega(|x-y|).
\end{equation}
\end{lemma}
\begin{proof}
The proof can be found in \cite{KNV}. Due to compactness provided by periodicity,
the only viable alternative to the breakthrough scenario of the lemma is
that we could have a single point $x$ at time $t_1$ where $|\pa_x u(x,t_1)|= \omega'(0)=1.$ This option can be ruled out due to our choice
of $\omega$ which satisfies $\omega''(0)=-\infty.$ Then any smooth function $f$ obeying $\omega$ must satisfy $|f'(x)| < 1$ for all $x.$
\end{proof}

Now we will consider the hypothetical breakthrough scenario of Lemma~\ref{break}, and will show that necessarily
\begin{equation}\label{breakineq}
 \left. \partial_t (u(x,t)-u(y,t)) \right|_{t=t_1} <0.
\end{equation}
This will imply that $u(x,t)$ does not obey $\omega$ for $t<t_1$ and contradict the definition of $t_1,$ thus establishing the preservation
of the modulus. To establish \eqref{breakineq}, we need to carry out some careful estimates. The first part of these estimates is
largely parallel to \cite{DKRT}, but the equation \eqref{maineq1} has additional features that will require novel considerations.
Roughly, one can describe the difficulty as stemming from lack of an a-priori upper bound on the amplitude of the Hilbert transform $Hu$
in terms of $\omega$ that is uniform in the period of solution. 
This may lead to depletion in the dissipative term, as the coefficient $\frac{u}{u^2+Hu^2}$ in front of it is no longer uniformly bounded away from zero.
This complication will require fairly subtle adjustments.

First of all, let us recall the short cut notation in \eqref{phipsi}
\begin{align}\label{phipsi1124} \psi(z) = \frac{Hu(z)}{u(z)^2+Hu(z)^2}, \quad \varphi(z) = \frac{u(z)}{u(z)^2+Hu(z)^2}.\end{align}
Since we now reserved the notation $x$ and $y$ for our special points where the solution touches the modulus,
in the remainder of this section we will usually denote the spacial variable by $z.$ We will also denote $\xi \equiv |x-y|.$
Then we have to prove that
\begin{equation}\label{breakineqexp}
\left. \pi \partial_t (u(x,t)-u(y,t))\right|_{t=t_1} = \psi(x) \partial_z u(x) - \psi(y) \partial_z u(y) -
\varphi(x) \Lambda u(x) + \varphi(y) \Lambda u (y) <0.
\end{equation}

We begin with the following lemma the proof of which can be found in \cite{KNV}.
\begin{lemma}\label{Hmod}
Suppose that $u$ obeys the modulus $\omega.$ Then $Hu$ obeys the modulus
\begin{equation}\label{Hmodeq}
\Omega(\xi) = C \left( \int_0^\xi \frac{\omega(\eta)}{\eta}\,d\eta + \xi \int_\xi^\infty \frac{\omega(\eta)}{\eta^2}\,d\eta \right).
\end{equation}
Here the constant $C$ is universal (and in particular does not depend on $\delta,\gamma$ from \eqref{modcondef}).
\end{lemma}
The lemma is valid for any modulus of continuity $\omega,$ not just \eqref{modcondef}. In \cite{KNV}, it is proved for the
Riesz transform of $u$ in two dimensions, but the proof for the Hilbert transform is identical.

Next we prove bounds on $\phi$ and $\psi.$ In the first part of our estimates that largely follows \cite{DKRT},
the role of the enforcer of regularity is played by the dissipative term $-{\rm min}(\varphi(x), \varphi(y)) (\Lambda u(x) - \Lambda u(y)).$
It is therefore convenient to introduce the notation
\begin{equation}\label{minphi}
\zeta = {\rm min}(\varphi(x), \varphi(y)).
\end{equation}
\begin{lemma}\label{phipsiest}
Suppose that $u(z,t)$ obeys the modulus $\omega$ given by \eqref{modcondef}. Let $\varphi$ and $\psi$ be given by \eqref{phipsi1124}.
 Then we have the following estimates:
\begin{equation}\label{psiest}
|\psi(x,t)-\psi(y,t)|,\,|\varphi(x,t)-\varphi(y,t)| \leq C \zeta \left(\omega(\xi)+\Omega(\xi)+\Omega(\xi)^2\right),
\end{equation}
where the constant $C$ may only depend on $m_0$ and $M_0.$ Here $x$ and $y$ can be any two points, and $\xi=|x-y|.$
In particular, the estimate \eqref{psiest} holds at the breakthrough time $t_1$ and at the breakthrough points $x$ and $y$ of Lemma \ref{break}.
\end{lemma}
\begin{proof}
For the sake of simplicity, we will not write the time variable in the calculation that follows; all estimates are valid for any two points
$x$ and $y$ at any time $t$ at which $u(z,t)$ obeys $\omega$.
We have
\begin{align}\nonumber
|\psi(x)-\psi(y)| =& \left| \frac{Hu(x) u(y)^2 - Hu(y)u(x)^2  +Hu(x)Hu(y)^2- Hu(y)Hu(x)^2}{(u(x)^2 + Hu(x)^2)(u(y)^2+Hu(y)^2)} \right| \\
\leq &~\frac{2M_0 {\rm min}(|Hu(x)|,|Hu(y)|) \omega(\xi) + M_0^2 \Omega(\xi) + |Hu(x) Hu(y)| \Omega(\xi)}{(u(x)^2 + Hu(x)^2)(u(y)^2+Hu(y)^2)}.
\label{psicalc}
\end{align}
The first summand in the numerator of \eqref{psicalc} can be estimated by factoring out $\zeta$ (multiplying and dividing by $u(y)$ or $u(x)$) and controlling
the remaining factor by $\frac{|Hu(z)|}{u(z)^2+Hu(z)^2},$ $z=x$ or $y.$ Such factor does not exceed $\frac{1}{2m_0},$ yielding the total upper bound of $M_0 \zeta \omega(\xi) m_0^{-2}.$
In a similar fashion, the second term is controlled by $M_0^2 \zeta \Omega(\xi) m_0^{-3}.$ The third term requires a little more work:
we write
\begin{align*}
|Hu(x)Hu(y)| =&~  |Hu(x)Hu(y)-{\rm min}(Hu(x)^2,Hu(y)^2)+{\rm min}(Hu(x)^2,Hu(y)^2)|\\
  \leq&~ {\rm min}(|Hu(x)|,|Hu(y)|)\Omega(\xi) +
{\rm min}(Hu(x)^2,Hu(y)^2). \end{align*}
With this, the third term in \eqref{psicalc} can be estimated by
\[ \frac{\zeta \Omega(\xi)^2}{2m_0^2}+\frac{\zeta \Omega(\xi)}{m_0}. \]
Collecting all the estimates, we arrive at \eqref{psiest}.

The estimate for $\varphi$ is proved similarly:
\begin{align}
|\varphi(x)-\varphi(y)|=&~ \left| \frac{u(x)u(y)^2 + u(x)Hu(y)^2-u(y)u(x)^2-u(y)Hu(x)^2}{(u(x)^2 + Hu(x)^2)(u(y)^2+Hu(y)^2)} \right|\label{phicalc}\\ \leq &~
\frac{u(x)u(y) \omega(\xi) + {\rm min}(Hu(x)^2,Hu(y)^2) \omega(\xi) + M_0 |Hu(x)+Hu(y)|\Omega(\xi)}{(u(x)^2 + Hu(x)^2)(u(y)^2+Hu(y)^2)}. \nonumber
\end{align}
The first two terms can be estimated from above by $\frac{2\omega(\xi)\zeta}{m_0}$ using estimates similar to the ones done above for the
$\psi$ case. To estimate the third term, write
\[ |Hu(x)+Hu(y)| \leq \Omega(\xi) +2{\rm min}(|Hu(x)|,|Hu(y)|). \]
Then the third term is controlled by
\[  \frac{M_0 \zeta \Omega(\xi)^2}{m_0^3}+\frac{M_0 \zeta \Omega(\xi)}{m_0^2}. \]
Combining the estimates gives \eqref{psiest}.
\end{proof}

The following three lemmas focus on the analysis of the fractional Laplacian in the breakthrough scenario.
The first lemma estimates the difference of the dissipative terms.
\begin{lemma}\label{dissipD}
Let $\omega$ be given by \eqref{modcondef}, with $\gamma =c \delta \leq \delta \ll 1.$
Suppose that $x$ and $y$ are as in the breakthrough scenario: $u(z)$ obeys $\omega$ while $u(x)-u(y)=\omega(\xi).$
Then
\begin{align}\label{Dpart}
&\Lambda u(x) - \Lambda u(y) \geq \D(\xi) \\ &\equiv \frac{1}{\pi} \left( \int_0^{\xi/2} \frac{2\omega(\xi) - \omega(\xi+2\eta) -\omega(\xi-2\eta)}{\eta^2} + \int_{\xi/2}^\infty \frac{ 2\omega(\xi) - \omega (\xi+2\eta)+\omega(2\eta-\xi)}{\eta^2}\right).
\nonumber
\end{align}
\end{lemma}
Note that both integrands above are nonnegative due to concavity of $\omega.$
We refer for the proof of this lemma to \cite{KNV}.
The second lemma provides control of the individual contributions of fractional Laplacians.
\begin{lemma}\label{fraclap1}
Suppose that $x$ and $y$ are as in the breakthrough scenario: $u(z)$ obeys $\omega$ while $u(x)-u(y)=\omega(\xi).$
Then
\begin{equation}\label{fraclapest123}
-\Lambda u(x), \,\,\,\Lambda u(y) \leq \left\{ \begin{array}{ll} C (1+ \log \xi^{-1}), & 0 \leq \xi \leq \delta \\
C\omega'(\xi),  & \xi > \delta. \end{array} \right.
\end{equation}
\end{lemma}
\begin{proof}
This lemma is similar to the one that appears in \cite{DKRT}, but provides a more general estimate in the range $\xi > \delta$.
Let us prove the result for the case of $\Lambda u(x),$ the other case is similar.
Recalling \eqref{fraclap}, we have
\begin{align}\label{auxfl1}
-\Lambda u(x) =&~ \frac{1}{\pi} P.V. \int_{\Rm} \frac{u(y)-u(x)-u(y)+u(z)}{|x-z|^2}\,dz \\
=&~\frac{1}{\pi} P.V. \int_\Rm \frac{\omega(|y-z|)-\omega(\xi)}{|x-z|^2}\,dz - E(x) \leq
 \frac{1}{\pi} P.V. \int_\Rm \frac{\omega(|\xi-\eta|)-\omega(\xi)}{\eta^2}\,d\eta. \nonumber
\end{align}
Here
\begin{equation}\label{Ex}
E(x) =  \frac{1}{\pi}\int_\Rm \frac{u(y)-u(z)+\omega(|y-z|)}{|x-z|^2}\,dz \geq 0.
\end{equation}
Split the last integral in \eqref{auxfl1} into four parts corresponding to integration over $(-\infty, -\xi],$ $[-\xi,\xi],$
$[\xi, 2\xi],$ and $[2\xi,\infty),$ and we denote them $I,$ $II,$ $III$ and $IV$ respectively.
Observe that
\begin{equation}\label{IIest}
II = \int_0^\xi \frac{\omega(\xi-\eta)+\omega(\xi+\eta)-2\omega(\xi)}{\eta^2}\,d\eta \leq 0
\end{equation}
for all $\xi$ due to concavity of $\omega.$
Also
\begin{equation}\label{IIIest}
III = \int_\xi^{2\xi} \frac{\omega(\eta-\xi)-\omega(\xi)}{\eta^2}\,d\eta \leq 0
\end{equation}
for all $\xi$ since $\omega$ is increasing.

The estimates for $I$ and $IV$ will be split into several cases.
Let us first look at $I$ in the $0 < \xi \leq \delta/2$ range. Here we have
\begin{align*}
 I = &~\int_{-\infty}^{-\xi} \frac{\omega(|\xi-\eta|)-\omega(\xi)}{\eta^2}\,d\eta \leq \int_\xi^\infty
 \frac{\omega(\xi+\eta)-\omega(\xi)}{\eta^2}\,d\eta\\ \leq&~
\int_\xi^{\delta-\xi} \frac{\eta}{\eta^2}\,d\eta + \int_{\delta-\xi}^\infty \frac{\gamma \log \left(1+ \frac14
\log \frac{\xi+\eta}{\delta} \right) +\delta}{\eta^2}\,d\eta\\ \leq&~ \log \xi^{-1} + 2 +
\gamma \delta^{-1} \int_{1/2}^\infty
\frac{\log \left(1+ \frac14 \log (1+s) \right)}{s^2}\,ds \leq \log \xi^{-1} + C (1+\gamma \delta^{-1}).
\end{align*}
We used the form of $\omega$ given in \eqref{modcondef}, concavity, and elementary calculations.
In the $\delta/2 < \xi \leq \delta$ range,
\[ I \leq \int_\xi^\infty \frac{\omega(\xi+\eta)}{\eta^2}\,d\eta \leq \int_{\delta/2}^\infty \frac{\delta + \gamma \log \left(
1+ \frac14 \log \left( 1+ \frac{\eta}{\delta} \right) \right)}{\eta^2}\,d\eta \leq C(1+\gamma \delta^{-1}). \]
The term $IV$ we will estimated in the whole $0 < \xi \leq \delta$ range at once:
\begin{align*}
IV \leq&~ \int_{2\xi}^\infty \frac{\omega(\eta-\xi)}{\eta^2}\,d \eta \leq \int_{2\xi}^{\xi+\delta} \frac{1}{\eta}\,d\eta +
\int_{\xi+\delta}^\infty \frac{\delta + \gamma \log \left(1+ \frac14
\log \frac{\xi+\eta}{\delta} \right)}{\eta^2}\,d\eta \\ \leq&~
\log \xi^{-1} +C(1+\gamma \delta^{-1}).
\end{align*}
Since we are assuming that $\gamma \leq \delta,$ this completes the proof of the first bound in \eqref{fraclapest123}.

Now let us consider the range $\xi > \delta.$ Here we have
\begin{align}
I = &~\int_\xi^\infty \frac{\omega(\xi+\eta)-\omega(\xi)}{\eta^2}\,d\eta = \int_\xi^\infty \frac{1}{\eta^2} \int_\xi^{\xi+\eta}
\omega'(s)\, ds d\eta \nonumber \\
=&~\int_\xi^{2\xi}\omega'(s) \int_\xi^\infty \frac{1}{\eta^2} \,d\eta ds + \int_{2\xi}^\infty \omega'(s) \int_{s-\xi}^\infty \frac{1}{\eta^2}\,d\eta ds = \frac{\omega(2\xi)-\omega(\xi)}{\xi}+ \int_{2\xi}^\infty \frac{\omega'(s)}{s-\xi}\,ds \nonumber \\
\leq&~\omega'(\xi) + 2 \int_\xi^\infty \frac{\omega'(s)}{s}\,ds \leq C\omega'(\xi). \label{Iom}
\end{align}
In computation, we used concavity of $\omega$ in estimating $\omega(2\xi)-\omega(\xi) \leq \omega'(\xi)\xi$. We also used the explicit form of $\omega'(s)$
in $s \geq \delta$ region to estimate
\begin{equation}\label{aux123c} \int_\xi^\infty \frac{\omega'(s)}{s}\,ds =  \int_\xi^\infty \frac{\gamma}{s^2 \log\left(4+\frac{s}{\delta}\right)}\,ds \leq
  \frac{\gamma}{\xi \log\left(4+\frac{\xi}{\delta}\right)} = \omega'(\xi). \end{equation}
For the term $IV,$ we get
\begin{align}
IV =&~ \int_{2\xi}^\infty \frac{\omega(\eta-\xi)-\omega(\xi)}{\eta^2} \, d\eta = \int_{2\xi}^\infty \frac{1}{\eta^2}
\int_\xi^{\eta-\xi} \omega'(s) \,ds d\eta  \nonumber  \\
=&~\int_\xi^\infty \omega'(s) \int_{\xi+\eta}^\infty \frac{1}{\eta^2}\,d\eta = \int_\xi^\infty \frac{\omega'(s)}{s+\xi}\,ds \leq
\omega'(\xi). \label{IVom}
\end{align}
This completes the proof of the lemma.
\end{proof}
We now state the lemma that is an immediate consequence of the estimates of Lemma~\ref{fraclap1}, but this specific statement will play a crucial
role at some point below.
\begin{lemma}\label{Elemma1}
Let $\omega$ be given by \eqref{modcondef}. Suppose that $x$ and $y$ are as in the breakthrough scenario: $u(z)$ obeys $\omega$ while $u(x)-u(y)=\omega(\xi).$
Then for $\xi > \delta,$ we have
\begin{equation}\label{LamxE}
\Lambda u(x) \geq  E(x) - C \omega'(\xi)
\end{equation}
and
\begin{equation}\label{LamyE}
\Lambda u(y) \leq  -E(y) + C \omega'(\xi),
\end{equation}
where $E(x)$ is given by \eqref{Ex} and
\begin{equation}\label{Ey}
E(y)= \frac{1}{\pi} \int_\Rm \frac{\omega(|x-z|)-u(x)+u(z)}{|y-z|^2}\,dz.
\end{equation}
\end{lemma}
\begin{proof}
The result for $\Lambda u(x)$ follows from \eqref{auxfl1}, \eqref{Iom}, \eqref{IVom}
and the fact that $II$ and $III$ defined in the proof of Lemma~\ref{fraclap1}
are non-positive. The proof for $\Lambda u(y)$ is completely analogous.
\end{proof}

Let us now start the consideration of the breakthrough scenario and the proof of \eqref{breakineqexp}.

1. The short range: $0 < \xi \leq \delta$. Observe that for all $\xi>0,$
\begin{equation}\label{aux123a} -\varphi(x) \Lambda u(x) +\varphi(y)\Lambda u(y) \leq -\zeta \D(\xi) + \left\{ \begin{array}{ll}
(\varphi(y)-\varphi(x)) \Lambda u(x) & {\rm if}\,\,\,\varphi(x) \geq \varphi(y) \\
 (\varphi(y)-\varphi(x)) \Lambda u(y) & {\rm if}\,\,\,\varphi(x) < \varphi(y). \end{array} \right. \end{equation}
Here $\D(\xi)$ is given by \eqref{Dpart} and we used Lemma~\ref{dissipD} in estimating
$\Lambda u(y)-\Lambda u(x) \leq -\D(\xi).$
 According to Lemma~\ref{phipsiest} and Lemma~\ref{fraclap1}, the second term on the right hand side of
 \eqref{aux123a} can be bounded from above by
\begin{equation}\label{smallxierr1}
C\zeta \left( \omega(\xi)+ \Omega(\xi) + \Omega(\xi)^2 \right)(1+\log \xi^{-1}).
\end{equation}
The first term on the right hand side of \eqref{aux123a} does not exceed
\begin{equation}\label{smallxiD}
-\zeta \D(\xi) \leq -\frac{\zeta}{\pi} \int_0^{\xi/2} \frac{2\omega(\xi)-\omega(\xi+2\eta)-\omega(\xi-2\eta)}{\eta^2}\,d\eta
\leq  -C \zeta \xi^{1/2}.
\end{equation}
Here the last step can be obtained by application of the Taylor theorem and concavity of $\omega:$
\begin{equation}\label{aux123b} 2\omega(\xi)-\omega(\xi+2\eta)-\omega(\xi-2\eta) \geq -2\omega''(\xi_1)\eta^2, \end{equation}
where $\xi_1 \in (\xi-2\eta,\xi).$

The advective flow term can be estimated by
\begin{equation}\label{smallxiflow}
|\psi(x) \partial_z u(x) - \psi(y) \partial_z u(y)| \leq |\psi(x)-\psi(y)| \omega'(\xi) \leq
C\zeta \left( \omega(\xi) + \Omega(\xi) + \Omega(\xi)^2 \right)\omega'(\xi).
\end{equation}
In the first step we used that $|\partial_z u(x)| = |\partial_z u(y)| = \omega'(\xi)$ at the breakthrough points;
otherwise one could shift $x$ or $y$ and obtain that the modulus is already not obeyed at the breakthrough  time $t_1$,
yielding a contradiction (see \cite{Kis11} for more details). The estimate \eqref{smallxiflow} is valid for all $\xi>0.$

It remains to estimate $\Omega(\xi)$. Using \eqref{Hmodeq}, we find
\begin{align} \nonumber
\Omega(\xi) \leq &~C \left( \int_0^\xi d \eta + \xi \int_\xi^\delta \frac{1}{\eta}\, d\eta + \xi
\int_\delta^\infty \frac{\delta +\gamma \log \left( 1+ \frac14 \log (\eta/\delta) \right)}{\eta^2}\,d\eta \right) \\
\leq&~C\xi\left(1+\log \xi^{-1} + \frac{\gamma}{\delta} \right) \leq C\xi (1+\log \xi^{-1}). \label{Omsmallxi}
\end{align}
To estimate the integral in the second step we rescaled the variable $\eta = \delta s;$ we also used in the estimates
that $\gamma \leq \delta \leq 1.$

Let us combine the estimates \eqref{smallxierr1}, \eqref{smallxiD}, \eqref{smallxiflow} and \eqref{Omsmallxi}.
Taking into account that $\omega'(\xi) \leq 1$ for all $\xi,$ the balance we need to verify in the short range then reads as follows:
\begin{equation}\label{smallxibal}
C \zeta \xi^{1/2} > C_1 \zeta (\xi (1+\log \xi^{-1})^2+\xi^2 (1+ \log \xi^{-1})^3).
\end{equation}
We can always ensure that \eqref{smallxibal} holds in the range $0 < \xi \leq \delta$ by choosing $\delta$ to be sufficiently small.

2. The mid-range: $\delta < \xi \leq \delta e^{\delta^{-1}}$. Let us start with an estimate on $\Omega(\xi)$ valid for
$\xi > \delta.$ We have
\begin{equation}\label{Ommidxi1}
 \Omega(\xi) \leq C \left(\int_0^\delta d\eta + \int_\delta^\xi \frac{\omega(\eta)}{\eta}\,d\eta + \xi \int_\xi^\infty \frac{\omega(\eta)}{\eta^2}\,d\eta \right)
\end{equation}
The second term on the right hand side of \eqref{Ommidxi1} can be estimated by $\omega(\xi)\log \frac{\xi}{\delta}.$
For the last term we have
\[ \xi \int_\xi^\infty \frac{\omega(\eta)}{\eta^2}\,d\eta = \omega(\xi)+\xi \int_\xi^\infty \frac{\omega'(\eta)}{\eta}\,d\eta \leq
\omega(\xi) + \xi \omega'(\xi) \leq 2\omega(\xi). \]
Here in the second step we used a computation parallel to \eqref{aux123c} and in the last step invoked concavity of $\omega.$
Combining all the estimates, we get that
\begin{equation}\label{Ommidfarxi}
\Omega(\xi) \leq C \omega(\xi)\left(4+\log \frac{\xi}{\delta}\right)
\end{equation}
for $\xi > \delta.$

For the dissipative term, we will now use the estimate that follows from \eqref{Dpart}:
\[ -\zeta \D(\xi) \leq -\frac{\zeta}{\pi} \int_{\xi/2}^\infty \frac{2\omega(\xi) - \omega (\xi+2\eta) + \omega (2\eta-\xi)}{\eta^2}\,d\eta. \]
For $\xi > \delta,$ we have
\begin{equation} \omega(2\xi) \leq \omega(\xi) + \gamma \int_\xi^{2\xi} \frac{d\eta}{4\eta} \leq \omega(\xi)+\frac{\gamma}{4} \leq \frac{3}{2}\omega(\xi) \label{xi2xi} \end{equation}
since $\gamma \leq \delta$ and $\delta$ can be taken small enough so that $\omega(\delta) \geq \delta/2.$
Thus we summarize that for $\xi > \delta,$
\begin{equation}\label{Dmidxi1}
-\zeta \D(\xi) \leq -C \zeta \frac{\omega(\xi)}{\xi}.
\end{equation}
Putting together the estimates \eqref{Dmidxi1},  \eqref{smallxiflow}, \eqref{aux123a}, \eqref{psiest} and \eqref{fraclapest123}, we obtain
that the balance we have to check in the $\delta < \xi \leq \delta e^{\delta^{-1}}$ range reads
\begin{equation}\label{balmidxi1} C \zeta  \frac{\omega(\xi)}{\xi} > C_1 \zeta \omega'(\xi)  \left(\omega(\xi)+
\Omega(\xi) +\Omega(\xi)^2 \right),  \end{equation}
where $C$ is a universal constant and $C_1$ depends only on $m_0$ and $M_0.$
Applying \eqref{Ommidfarxi}, after a short computation we see that it is sufficient to ensure that if $\delta < \xi \leq \delta e^{-\delta^{-1}},$ then
\begin{equation}\label{gamma1}
C_2 \gamma \left(1+\omega(\xi)\left(4+\log\frac{\xi}{\delta}\right)\right) \leq C_2 \gamma \left(1+ \left(\delta + \gamma \log \left(1+ \frac14 \delta^{-1}\right)\right)(4+\delta^{-1})\right) < C.
\end{equation}
This is just a condition on $\gamma$ and $\delta$ that we will enforce; it suffices to require that $\gamma \leq c\delta$ with
appropriately small $c$ (which may depend only on $m_0$ and $M_0$ through the constant $C_2$). This completes the treatment of the
$\delta \leq \xi \leq \delta e^{\delta^{-1}}$ range.

Observe that the bounds \eqref{Dmidxi1}, \eqref{Ommidfarxi}, \eqref{psiest}, \eqref{aux123a} and \eqref{smallxiflow}, and thus 
\eqref{balmidxi1} and \eqref{gamma1}, are also valid in the $\xi > \delta e^{\delta^{-1}}$ far range.
But as we can see from \eqref{balmidxi1} and \eqref{gamma1}, for unboundedly large $\xi$ we cannot establish the needed
control with the above argument. Moreover, modifying $\omega$ will not help. The reason is the presence of $\Omega(\xi)^2$ in the
balance \eqref{balmidxi1}. This summand appears there due to the structure of $\varphi$ and $\psi$ and in particular the way Hilbert transform
enters into these functions. In fact, without $\Omega(\xi)^2$ part, the balance \eqref{balmidxi1} could have been controlled over all $\xi$
very similarly to \cite{KNV,DKRT}. However, in our case the leading term condition for large $\xi$ becomes
\[  \frac{1}{\xi (\log \xi)^2} \gtrsim \omega(\xi)\omega'(\xi) \]
which leads to globally bounded $\omega$ due to integrability of the right hand side near infinity.
A bounded modulus cannot be used to control arbitrary initial data, and will
lead to global regularity only under appropriate additional smallness condition.

To control the large $\xi$ range, we will use a different argument. This is the truly novel part of this global regularity proof.
The main idea is to use sharper estimates on dissipation, with Lemma~\ref{Elemma1} as the initial step.
We now prove a key estimate on the quantities $E(x)$ \eqref{Ex} and $E(y)$ \eqref{Ey} appearing in that lemma. First, we single out a more
convenient part of integrals defining these quantities. Assume, without loss of generality, that $y <x.$ Set
\begin{align} E_1(x) =&~ \frac{1}{\pi} \int_y^x \frac{\omega(|y-z|)-u(z)+u(y)}{|x-z|^2}\,dz = \frac{1}{\pi} \int_0^\xi \frac{\omega(\eta)+u(y)-u(y+\eta)}{(\xi-\eta)^2}\,d\eta, \label{E1x}
 \\ \label{E1y} E_1(y) =&~ \frac{1}{\pi} \int_y^x \frac{\omega(|x-z|)+u(z)-u(x)}{|y-z|^2}\,dz = \frac{1}{\pi} \int_0^\xi \frac{\omega(\xi-\eta)-u(x)+u(y+\eta)}{\eta^2}\,d\eta.
\end{align}
Observe that $E_1(x) \leq E(x)$ and $E_1(y) \leq E(y).$
\begin{proposition}\label{keyEprop}
Suppose that $\xi \geq \delta e^{\delta^{-1}},$ and $\delta$ is chosen sufficiently small, to be specified later.
Assume that
\begin{equation}\label{E1ybound} E_1(y) \leq \frac{q \omega(\xi)}{\xi}, \end{equation}
where $\xi^{-1/2} \leq q \leq \frac{1}{128\pi}.$ Then 
\begin{equation}\label{E1xbound}
E_1(x) \geq \frac{\omega(\xi)}{64\pi q \xi}.
\end{equation}
The result also holds if $E_1(y)$ and $E_1(x)$ are switched places.
\end{proposition}

Before proving the proposition, let us outline the idea behind our plan to control the large $\xi$ range.
The estimate $\Lambda u(x) - \Lambda u (y) \geq \D(\xi)$ is, in general, sharp. One can see this from its proof \cite{KNV}.
The profile for $u(z)$ on which the equality is achieved is equal, up to a vertical shift by a constant,
to $u(z) = \frac12 \omega (2z) {\rm sgn}(z)$ provided
that the origin is placed at $\frac{x-y}{2}$. 
For this profile, one can compute that $\Lambda u(x) = -\Lambda u(y) = \frac12 \D(\xi).$ A computation similar to the one
that appears below shows that if there is a balance like that then one can show the breakthrough cannot happen for large $\xi$ as well.
The difficulty appears when one of the fractional Laplacians - say $\Lambda u(y)$ - has absolute value much smaller than $\D(\xi)$, while
$|Hu(x)|$ is large and, due to its appearance in the denominator of $\varphi,$ suppresses the dissipative contribution of $\Lambda u(x).$ We will use Proposition~\ref{keyEprop}
to show that in this event $\Lambda u(x)$ turns out to be much larger than $\D(\xi)$ and sufficient to control the
scenario.

\begin{proof}[Proof of Proposition~\ref{keyEprop}]
From \eqref{E1ybound} and \eqref{E1y} we can conclude that
\begin{equation}\label{init126} \int_{\xi/2}^\xi (\omega(\xi-\eta)-u(x)+u(y+\eta))\,d\eta \leq \pi q \xi \omega(\xi). \end{equation}
Therefore
\[ \int_{\xi/2}^\xi (\omega(\eta) - u(y+\eta) + u(y))\,d \eta \geq \int_{\xi/2}^\xi (\omega(\eta)+\omega(\xi-\eta)-\omega(\xi))\,d\eta
-\pi q \omega(\xi)\xi. \]
Note that the expressions under the integrals are non-negative due to concavity of $\omega$ and the fact that $u$ obeys it. Below we will find an $a>0$ such that
\begin{equation}\label{acond} \int_{\xi-a}^\xi (\omega(\eta)+\omega(\xi-\eta)-\omega(\xi))\,d\eta \geq \pi q \xi\omega(\xi). \end{equation}
Then we claim that
\begin{equation}\label{E1xlb1}
E_1(x) = \frac{1}{\pi}\int_{0}^{\xi} \frac{\omega(\eta)-u(y+\eta)+u(y)}{(\xi-\eta)^2}\,d\eta \geq \frac{1}{\pi}\int_{\xi/2}^{\xi-a} \frac{\omega(\eta)+\omega(\xi-\eta)-\omega(\xi)}{(\xi-\eta)^2}\,d\eta.
\end{equation}
Indeed, this inequality is implied by
\begin{align}\label{Eaux1124} \int_{\xi-a}^\xi
\frac{\omega(\eta) - u(y+\eta) + u(y)}{(\xi-\eta)^2}\,d\eta \geq  \int_{\xi/2}^{\xi-a} \frac{\omega(\xi-\eta)-u(x)+u(y+\eta)}{(\xi-\eta)^2}\,d\eta. \end{align}
But \eqref{Eaux1124}
follows from the monotonicity of the denominator and 
\begin{align}\label{Eaux21124}  \int_{\xi-a}^\xi (\omega(\eta) - u(y+\eta) + u(y))\,d\eta \geq \int_{\xi/2}^{\xi-a} (\omega(\xi-\eta)-u(x)+u(y+\eta))\,d\eta. \end{align}
On the other hand, \eqref{Eaux21124} is implied by \eqref{acond} and \eqref{init126}.

Let us now estimate $a.$ First, we claim that for all  $\delta$ sufficiently small,
we can take $a = 8\pi q \xi.$
Note that due to our assumptions on the range of $q,$ we have
$8 \pi q \leq \frac{1}{16}.$ Recall (see \eqref{xi2xi}) that
\begin{equation} \omega(\xi)-\omega(\eta)\leq\int_{\xi/2}^\xi
  \omega'(z) \,dz \leq
  \frac{\gamma}{4}, \label{xi2xi2} \end{equation}
for any $\eta\in[\xi-a,\xi]$, and so
\[\int_{\xi-a}^\xi (\omega(\eta)-\omega(\xi))\,d\eta\geq -\frac{\gamma a}{4}.\]
On the other hand, by concavity of $\omega$,
\[ \int_{\xi-a}^\xi\omega(\xi-\eta)\,d\eta=\int_0^a
  \omega(\eta)\,d\eta \geq \frac{a}{2}\omega(a). \]
Notice that due to our assumption that
\[q\xi\geq\xi^{1/2}\geq\delta^{1/2}e^{\delta^{-1}/2},\]
we have $a\geq\delta$ for all $\delta\leq1$, and
$\omega(a)\geq\gamma$ if we assume that $\delta$ is sufficiently small
and $\gamma\leq\delta/4$. Then we obtain
\[ \int_{\xi-a}^\xi (\omega(\eta)+\omega(\xi-\eta)-\omega(\xi))\,d\eta
  \geq \frac{a}{2}\omega(a)-\frac{\gamma a}{4}\geq \frac{a}{4}\omega(a)=2\pi q\xi\omega(8\pi q\xi).
\]

It remains to verify that
\begin{equation}\label{aux125b} 2\omega \left( 8 \pi  q \xi \right) \geq \omega(\xi). \end{equation}
Due to \eqref{modcondef}, this would follow from
\[ \log \left(1+ \frac14 \log \frac{8 \pi q \xi}{\delta} \right)^2 \geq \log \left(1+\frac14 \log \frac{\xi}{\delta}\right), \]
and thus from
\[ 2 \log \frac{8\pi q \xi}{\delta} \geq \log \frac{\xi}{\delta} \]
or
\[ \frac{64 \pi^2 q^2 \xi}{ \delta} \geq 1. \]
Since we assume that $q^2 \xi \geq 1,$ this inequality is clearly true for all $\delta \leq 1.$

Now we can estimate $E_1(x)$ using \eqref{E1xlb1}:
\begin{align}\nonumber
E_1(x) \geq&~ \int_{\xi/2}^{\xi-a} \frac{\omega(\eta)+\omega(\xi-\eta)-\omega(\xi)}{(\xi-\eta)^2}\,d\eta \\=&~
\left. \frac{\omega(\eta)+\omega(\xi-\eta)-\omega(\xi)}{\xi-\eta}\right|_{\xi/2}^{\xi-a} -\int_{\xi/2}^{\xi-a} \frac{\omega'(\eta) - \omega'(\xi-\eta)}{\xi-\eta}\,d\eta \nonumber\\ \geq&~
\frac{\omega(\xi-a)+\omega(a)-\omega(\xi)}{a} -\frac{4\omega(\xi/2)-2\omega(\xi)}{\xi}. \label{aux125}
\end{align}
In the last step we used that $\omega'(\xi-\eta) \geq \omega'(\eta)$ for $\eta \geq \xi/2$ since $\omega'$ is decreasing.
The last term in \eqref{aux125} is bounded from below by $-\frac{2 \omega(\xi)}{\xi}.$
Due to \eqref{xi2xi2} and \eqref{aux125b}, we also have
\[ \omega(\xi-a)+\omega(a)-\omega(\xi) \geq \frac12 \omega(\xi)-\frac{\gamma}{4}\geq \frac14 \omega(\xi). \]
Collecting all the estimates we conclude that, since $q \leq \frac{1}{128\pi},$
\[ E_1(x) \geq \frac{\omega(\xi)}{32 \pi q \xi} - \frac{2 \omega(\xi)}{\xi} \geq \frac{\omega(\xi)}{64 \pi q \xi}. \]
The proof of the statement with $E_1(x)$ and $E_1(y)$ reversed is analogous.
\end{proof}
Now we are ready to consider the breakthrough scenario for large $\xi$.

3. The far range: $\xi > \delta e^{\delta^{-1}}$. There are two terms in our estimates that do not allow
to prove \eqref{breakineqexp} similarly to the previous range (and to the earlier works such as \cite{KNV,DKRT}).
These are the terms leading to $\Omega(\xi)^2$ in \eqref{balmidxi1}, and they are given by
\begin{equation}\label{psibad} \frac{|Hu(x) Hu(y)|\cdot |Hu(x)-Hu(y)|}{(u(x)^2 + Hu(x)^2)(u(y)^2+Hu(y)^2)}\omega'(\xi) \end{equation}
arising from $|\psi(x)-\psi(y)|\omega'(\xi)$ (see \eqref{psicalc} and \eqref{smallxiflow}) and
\begin{equation}\label{phibad}
 \frac{CM_0 |Hu(x)+Hu(y)|\cdot|Hu(x)-Hu(y)|}{(u(x)^2 + Hu(x)^2)(u(y)^2+Hu(y)^2)}\omega'(\xi)
\end{equation}
arising from
\[ \left\{ \begin{array}{ll}
(\varphi(y)-\varphi(x)) \Lambda u(x) & {\rm if}\,\,\,\varphi(x) \geq \varphi(y) \\
 (\varphi(y)-\varphi(x)) \Lambda u(y) & {\rm if}\,\,\,\varphi(x) < \varphi(y). \end{array} \right. \]
(see \eqref{phicalc}, \eqref{aux123a} and \eqref{fraclapest123}).
Estimating Hilbert transforms in the numerator of these expressions leads to the highest order terms with only
$\lesssim \gamma \zeta \omega'(\xi) \Omega(\xi)^2$ bound that, as we discussed after the consideration of the mid-range,
is not possible to control by the dissipative term $\zeta D(\xi) \sim \zeta \omega(\xi) \xi^{-1}.$

In order to focus on the terms \eqref{psibad} and \eqref{phibad}, let us use $\frac12 (-\varphi(x)\Lambda u(x)+\varphi(y)\Lambda u(y))$
to estimate all other contributions as outlined above (producing the term \eqref{phibad} in the process);
now we will show how to use the remaining $\frac12 (-\varphi(x)\Lambda u(x)+\varphi(y)\Lambda u(y))$ to control
\eqref{phibad} and \eqref{psibad} in a different way.

Without loss of generality, suppose that $|Hu(x)| \geq |Hu(y)|;$ the other alternative is handled in an identical way.
We start by looking at the fractional Laplacian with the potentially larger pre-factor, as $\varphi(y) \geq \frac{m_0}{M_0} \varphi(x)$ if $|Hu(x)| \geq |Hu(y)|$.
Suppose that
\begin{equation}\label{goodcase1} -\frac12 \varphi (y) \Lambda u(y) > \frac{(CM_0 |Hu(x)+Hu(y)|+|Hu(x) Hu(y)|) \Omega(\xi)}{(u(x)^2 + Hu(x)^2)(u(y)^2+Hu(y)^2)}\omega'(\xi), \end{equation}
that is, the size of the fractional Laplacian term at $y$ exceeds the sum of the problematic terms.
Observe that the remaining term $-\phi(x) \Lambda u(x)$ satisfies
\begin{equation}\label{fraclapx125} -\frac12 \phi(x) \Lambda u(x) \leq C\zeta \frac{M_0}{m_0} \omega'(\xi) \end{equation}
due to \eqref{fraclapest123} and our assumption that $|Hu(x)| \geq |Hu(y)|.$
The estimate \eqref{fraclapx125} implies that the contribution of $-\phi(x) \Lambda u(x)$ can be controlled by $\zeta D(\xi)$
since
\[ C_1\frac{\omega(\xi)}{\xi} \geq \frac{C_1\left(\frac12 \delta + \gamma \log \left(1+\frac14\log \delta^{-1}\right)\right)}{\xi} >
\frac{C_2\gamma}{\xi(4+\log \frac{\xi}{\delta})}=C_2 \omega'(\xi) \]
for $\xi > \delta e^{\delta^{-1}}$ if $\gamma \leq \delta$ and $\delta$ is sufficiently small.
Therefore, if \eqref{goodcase1} holds, we are done. It remains to consider the case where
\begin{equation}\label{badcase1}
-\Lambda u(y) \leq  \frac{C_0(|Hu(x)+Hu(y)|+|Hu(x) Hu(y)|) \Omega(\xi)}{u(x)^2 + Hu(x)^2}\omega'(\xi),
\end{equation}
where $C_0$ is a constant that may only depend on $m_0$ and $M_0$ and may change from line to line below.

Due to \eqref{fraclapest123} and \eqref{LamyE}, \eqref{badcase1} implies that
\begin{equation}\label{badcase2}
 E_1(y) \leq E(y) \leq C_0 \left( 1+ \frac{(|Hu(x)+Hu(y)|+|Hu(x) Hu(y)|) \Omega(\xi)}{u(x)^2 + Hu(x)^2}\right)\omega'(\xi).
 \end{equation}
We would like to apply Proposition~\ref{keyEprop}. Let us denote
\begin{align} \nonumber  q = &~C_0 \left( 1+ \frac{(|Hu(x)+Hu(y)|+|Hu(x) Hu(y)|)\Omega(\xi)}{u(x)^2 + Hu(x)^2}\right) \frac{\xi \omega'(\xi)}{\omega(\xi)}
 \\ \label{sigmabad}=&~ C_0 \gamma \left( 1+ \frac{(|Hu(x)+Hu(y)|+|Hu(x) Hu(y)|)\Omega(\xi)}{u(x)^2 + Hu(x)^2}\right) \frac{1}{\omega(\xi)(4+\log \frac{\xi}{\delta})}.
\end{align}
We need to check that $ q$ satisfies the assumptions of Proposition~\ref{keyEprop}. First of all, due to \eqref{Ommidfarxi},
$|Hu(x)| \geq |Hu(y)|,$ and
\[ \omega(\xi)\left(4+\log \frac{\xi}{\delta}\right) \geq \delta (4+ \delta^{-1}) \geq 1 \]
in the far range,
we have that $ q \leq C_0 \gamma$ which can be made less than $\frac{1}{128\pi}$ if we select $\gamma$ sufficiently small. On the other hand, we also have
\[  q \geq \frac{C_0\gamma}{\omega(\xi)(4+\log \frac{\xi}{\delta})}. \]
The condition $ q \xi^{1/2} \geq 1$ is then satisfied if
\begin{equation}\label{sigmacond1} C_0 \gamma \xi^{1/2} \geq \omega(\xi)\left(4+\log \frac{\xi}{\delta}\right).  \end{equation}
Note that all constraints we have so far placed on $\gamma$ were of the kind $\gamma \leq c$ or $\gamma \leq c \delta$ with $c$ that may only depend on $m_0$ and $M_0$. To satisfy \eqref{sigmacond1} and all other constraints, we can first set $\gamma = c\delta$ with a sufficiently small constant $c$. The inequality \eqref{sigmacond1} then holds if for all $\xi > \delta e^{\delta^{-1}}$ we have
\begin{equation}\label{deltacon125}
C_0  \xi^{1/2} \geq  \left(c^{-1}+ \log \left(1+\frac14 \log \frac{\xi}{\delta} \right)\right) \left(4+\log \frac{\xi}{\delta}\right).
\end{equation}
It is not hard to see that \eqref{deltacon125} is true for $\xi > \delta e^{\delta^{-1}}$ if $\delta$ is small enough. Thus we can take care of
\eqref{deltacon125} as well as all other constraints we placed on $\delta$ and $\gamma$ by selecting $\delta$ to be sufficiently small.


It follows that if \eqref{badcase1} applies, then we can invoke Proposition~\ref{keyEprop} and conclude that
\[ E_1(x) \geq \frac{\omega(\xi)}{64\pi  q \xi}. \]
From Lemma~\ref{Elemma1} it follows that if $\delta$ is sufficiently small, then
\begin{equation}\label{Lxlb}
 \Lambda u(x) \geq E_1(x) - C_1 \omega'(\xi) \geq \frac{\omega(\xi)}{128 \pi q \xi}.
 \end{equation}
In the second step in \eqref{Lxlb} we used that due to \eqref{modcondef}, we have $\xi \omega'(\xi) \leq \delta \omega(\xi)$ if $\xi \geq \delta e^{\delta^{-1}}.$
Thus the critical relationship that we need to check in the far range becomes the second inequality in the equation below:
\begin{equation}\label{balfarrange} \frac12 \varphi(x) \Lambda u(x) \geq \frac{1}{u(x)^2 +Hu(x)^2} \frac{C\omega(\xi)}{ q \xi} > \frac{C_0( |Hu(x)+Hu(y)|+|Hu(x) Hu(y)|) \Omega(\xi)}{(u(y)^2+Hu(y)^2)(u(x)^2 + Hu(x)^2)}\omega'(\xi), \end{equation}
with $C$ and $C_0$ some constants that may only depend on $m_0$ and $M_0,$ and may change from step to step in the following calculations.
Using \eqref{sigmabad} and \eqref{Ommidfarxi}, we see that the inequality we need to prove reduces to
\begin{align} \nonumber
C \geq&~ \gamma C_0 q \frac{|Hu(x)+Hu(y)|+|Hu(x) Hu(y)|}{u(y)^2+Hu(y)^2} =
  \frac{|Hu(x)+Hu(y)|+|Hu(x) Hu(y)|}{u(y)^2+Hu(y)^2} \times \\
&\times  \left( 1+ \frac{(|Hu(x)+Hu(y)|+|Hu(x) Hu(y)|)\Omega(\xi))}{u(x)^2 + Hu(x)^2}\right) \frac{ C_0 \gamma^2}{\omega(\xi)(4+\log \frac{\xi}{\delta})}.
  \label{final125}
\end{align}
Let us open up the brackets and estimate the two terms on the right hand side of \eqref{final125}. The first one is equal to
\begin{align*} &\frac{|Hu(x)+Hu(y)|+|Hu(x) Hu(y)|}{u(y)^2+Hu(y)^2}\frac{ C_0 \gamma^2}{\omega(\xi)(4+\log \frac{\xi}{\delta})}
\\ &\leq\frac{2|Hu (y)| +\Omega(\xi)+ Hu(y)^2 + |Hu(y)|\Omega(\xi)}{u(y)^2 +Hu(y)^2} \frac{C_0\gamma^2}{\omega(\xi)(4+\log \frac{\xi}{\delta})} \leq C_0\gamma^2 \left[ 1+ \frac{1}{\delta (4+ \delta^{-1})} \right] \leq C_0 \gamma^2. \end{align*}
 In the penultimate step, we used \eqref{Ommidfarxi} and $\xi \geq \delta e^{\delta^{-1}}.$
Thus the first term is controlled by choosing $\gamma$ to be sufficiently small.
The second term can be estimated as follows:
\begin{eqnarray*}
\frac{(|Hu(x)+Hu(y)|+|Hu(x) Hu(y)|)^2\Omega(\xi)}{(u(y)^2+Hu(y)^2)(u(x)^2 + Hu(x)^2)}\frac{ C_0 \gamma^2}{\omega(\xi)(4+\log \frac{\xi}{\delta})} \leq C_0\gamma^2,
\end{eqnarray*}
due to \eqref{Ommidfarxi} and elementary inequalities. This completes the proof that $\omega$ is preserved by evolution and thus the
proof of global regularity.



\section{Asymptotic behavior: flocking of polynomial roots}\label{asymptotic}

In this section, we study the asymptotic behavior of the solutions of
\eqref{maineq}.
Let us recall the equivalent form \eqref{ueq}:
\begin{equation}\label{ueq5}
  \pi\pa_tu-\psi\pa_xu=-\varphi\Lambda u,\quad
  \psi=\frac{Hu}{u^2+(Hu)^2},\quad
  \varphi=\frac{u}{u^2+(Hu)^2},
\end{equation}
with positive initial data $u_0\in H^s(\S)$, for $s>3/2$.

Observe that for a regular solution in a periodic setting the total mass is conserved,
\[ \int_{\S} u(x,t)\,dx = \int_\S u(x,0)\,dx \equiv 2 \pi \bar u. \]
In this section, we will show that the solution converges to the uniform steady state, at an exponential rate in time.
We will also prove exponential decay bounds on all spacial derivatives of the solution.
These bounds will play an important role in our application to evolution of polynomial roots under differentiation.
\subsection{Convergence to equilibrium in the $L^\infty$ norm}\label{expconv}
We start with an exponential in time convergence to the uniform steady
state in the $L^\infty$ norm.

\begin{proposition}\label{linfconv}
Suppose that $u(x,t)$ is a global smooth periodic solution of \eqref{ueq5}.
Then the maximum of the solution $M(t)$ is non-increasing, the minimum $m(t)$ is non-decreasing, and
\begin{equation}\label{eq:Vdecay}
V(t):=M(t)-m(t) \leq (M_0-m_0)e^{-\sigmaz t},
\end{equation}
where the constant $\sigmaz>0$ may depend only on $m_0,$ $M_0,$ and $\|
u_0'\|_{L^\infty}$.
\end{proposition}

Consider a point $x_0$ where the maximum is reached at time $t_0.$
From \eqref{ueq5}, we find that
\begin{equation}\label{dtumax}
  \pi\partial_t u(x_0,t_0) =  -\varphi(x_0,t_0)\Lambda u(x_0,t_0).
\end{equation}

The factor $\varphi(x_0,t_0)$ can be bounded (uniformly in time) from below, thanks to the
maximum principle Proposition \ref{prop:mp} and the
uniform-in-time bound on $Hu$ in \eqref{Bbound} that follows from global regularity.  We get
\begin{equation}\label{phimin}
  \varphi(x_0,t_0)\geq\frac{m(t_0)}{M(t_0)^2+Q^2}\geq\frac{m_0}{M_0^2+Q^2}=:\phim.
\end{equation}

The following estimate is used to control $\Lambda u(x_0,t_0)$.
\begin{lemma}\label{fraclapmin}
Suppose that $f(x) \in C(\S)$ takes its maximum value $M$ at a point $x_0$, and its minimum value $m$ of $f$ at a point $\tilde x_0.$ Denote
\[ \bar f = \frac{1}{2\pi} \int_{\S} f(x)\,dx. \]
Then
\begin{align}\label{fraclapest1}
\Lambda f(x_0) >  \frac{M-m}{\pi} \cot \frac{\pi (\bar f -m)}{2(M-m)} \\
\label{fraclapest1212} \Lambda f(\tilde x_0) <  -\frac{M-m}{\pi} \cot \frac{\pi (M -\bar f)}{2(M-m)}
\end{align}
\end{lemma}
\begin{remark}
  The function $f$ may assume its maximum and minimum values at more
  than one point: the bounds \eqref{fraclapest1} and \eqref{fraclapest1212} respectively apply at every such
  point.
\end{remark}
\begin{proof}
Recall that
\[ \Lambda u(x) =\frac{1}{4\pi} P.V. \int_{-\pi}^{\pi} \frac{u(x)-u(y)}{\sin^2(\frac{x-y}{2})}\,dy. \]
Given the constraints on function $f$ and monotonicity of the kernel, it is not
hard to show that the minimum possible value of $\Lambda f(x_0)$
is achieved when $f$ is chosen to be
\[ f_{\min}(x)=\begin{cases}M& |x-x_0|\leq d\\m& |x-x_0|>d\end{cases}\qquad
\text{with }~d=\frac{\pi(\bar{f}-m)}{M-m}.\]
Such function is not continuous, but can be approximated by continuous functions.
Estimating the value of $\Lambda f_{\min}(x_0)$ from below gives
\[\Lambda f_{\min}(x_0)
  =\frac{M-m}{2\pi}\int_d^\pi\frac{1}{\sin^2(\frac{y}{2})}\,dy
  = \frac{M-m}{\pi} \cot \frac{d}{2}, \]
  resulting in \eqref{fraclapest1}. The proof of \eqref{fraclapest1212} is analogous.
\end{proof}

\begin{proof}[Proof of Proposition~\ref{linfconv}]
Applying \eqref{phimin} and Lemma~\ref{fraclapmin} to \eqref{dtumax}, we obtain
\begin{equation}\label{uminest1}
 \pi\partial_t u(x_0,t_0)  <  -\frac{(M-m)\phim}{\pi} \cot \frac{\pi (\bar f -m)}{2(M-m)} 
\end{equation}
The estimate \eqref{uminest1} holds at any point where the maximum is taken, and, by continuity, in some small time-space
neighborhoods $U \times [t_0,t_0+\tau)$ of these points. By continuity again, decreasing $\tau$ if necessary,
for $(x,t) \in ((-\pi,\pi] \setminus U)\times [t_0,t_0+\tau]$ we have $u(x,t) \leq M(t_0)- A\tau,$ where $A$ stands
for the right hand side of \eqref{uminest1}. Note that the function $M(t)$ is Lipschitz, so its derivative exists almost everywhere
and the fundamental theorem of calculus applies to $M(t).$
All together, the above arguments and \eqref{uminest1} imply that for every $t_0,$
\begin{equation}\label{maxup1}
M'(t_0) \leq -\frac{(M-m)\phim}{\pi^2} \cot \frac{\pi (\bar f -m)}{2(M-m)}.  
\end{equation}
A similar estimate can be carried out at the minimum of $u,$ yielding
\begin{equation}\label{minlb1}
m'(t_0) \geq \frac{(M-m)\phim}{\pi} \cot \frac{\pi (M -\bar f)}{2(M-m)}. 
\end{equation}
Subtracting \eqref{minlb1} from \eqref{maxup1}, we get
\begin{align}
  M'(t)-m'(t) \leq&\, -\frac{\phim}{\pi^2}(M(t)-m(t))
  \left[ \cot \frac{\pi (\bar f -m)}{2(M-m)}+ \cot \frac{\pi (M -\bar f)}{2(M-m)} \right]\nonumber\\
  \leq&\,-\frac{2\phim}{\pi^2}(M(t)-m(t)). \label{Vdynamics}
\end{align}
Here in the last step we used that
\[ \frac{\pi (\bar f -m)}{2(M-m)}+\frac{\pi (M -\bar f)}{2(M-m)}=\frac{\pi}{2}, \]
and that
\[ \cot \theta + \cot \left( \frac{\pi}{2}-\theta \right) = \cot \theta + \frac{1}{\cot \theta} \geq 2. \]

Application of Gronwall inequality finishes the proof, with the
constant
\begin{equation}\label{eq:D}
  \sigmaz=\frac{2\phim}{\pi^2}=\frac{2m_0}{\pi^2(M_0^2+Q^2)}.
\end{equation}
Here $Q$ is defined in \eqref{Bbound}, and depends only on $m_0,
M_0$ and $\|u_0'\|_{L^\infty}$.
\end{proof}
Since $M(t) \geq \bar u \geq m(t),$ \eqref{eq:Vdecay} implies exponential in time convergence of the solution
to its mean. In a later subsection, we will show a stronger bound where the rate of convergence will only depend on $\bar u.$

\subsection{Exponential decay of the first derivative}
\begin{proposition}\label{thm:ux}
  Suppose $u(x,t)$ is a global smooth periodic solution of
  \eqref{ueq5}. Then
  \begin{equation}\label{eq:uxdecay}
    \|\pa_xu(\cdot,t)\|_{L^\infty}\leq C_1e^{-\sigmaz t},
  \end{equation}
  where $\sigmaz$ is defined in \eqref{eq:D}, and $C_1$ is a constant that
  may depend only on $m_0$, $M_0$ and  $\|u_0'\|_{L^\infty}$.
\end{proposition}

The function $v:=\pa_xu$ satisfies equation
\[\pi\pa_tv-\psi\pa_xv=-\varphi\cdot\Lambda v+
  \big(\pa_x\psi\cdot v-\pa_x\varphi\cdot
  Hv\big) =:\G_1+\B_1.\]

Fix a time $t_0$, and let $x_0$ be a point where $\max_xv(x,t_0)$ is
attained. Then
\begin{equation}\label{vdecompose}
  \pi\pa_tv(x_0)=-\varphi(x_0)\cdot\Lambda v(x_0)+\big[\pa_x\psi(x_0)\cdot v(x_0)-\pa_x\varphi(x_0)\cdot
  Hv(x_0)\big]=\G_1(x_0)+\B_1(x_0).
\end{equation}

Let us first consider the good term $\G_1(x_0)$.
The dissipation term $\Lambda v(x_0)$ can be estimated by the
following enhanced nonlinear maximum principle, which takes into
account the exponential decay of $V(t)$.
It has been pioneered in \cite{CV1} and further developed
in \cite{ST2} in a slightly different form than the one we use here.

\begin{lemma}[Enhanced nonlinear maximum principle]\label{lem:nmp}
  Let $u\in C^1(\S)$ be a given function with the amplitude $V=\max u-\min
  u\geq0$, and $v=u'$. Let $x_0$ be the point where $\max v$ is attained.
  Then, the following two estimates hold:
  \begin{align}
    8\pi V \Lambda v(x_0)\geq&~v(x_0)^2; \label{est:lamv1}\\
    \Lambda v(x_0)\geq&~ \max( v(x_0)-\pi V, 0). \label{est:lamv2}
  \end{align}
 \end{lemma}
 \begin{proof}
   If $V=0$, $u$ is a constant and $v\equiv0$. The inequalities are
   trivial. We focus on the case when $V>0$; in this case $v(x_0)>0.$

   Let $\chi \geq 0$ be a smooth and decreasing test function defined in $\mathbb{R}^+$,
   such that $\chi(r)=1$ for $r<1$, $\chi(r)=0$ for $r>2$, and
   $\|\chi'\|_{L^\infty}<2$. Let $\chi_R(r)=\chi(r/R)$.
   Estimate $\Lambda v$ as follows.
   \begin{align*}
    \pi\cdot\Lambda v(x_0)\geq&~
     \int_\Rm (1-\chi_R(|y|))\cdot\frac{v(x_0)-v(x_0+y)}{|y|^2}\,dy\\
     \geq&~v(x_0)\int_{|y|>2R}\frac{1}{|y|^2}dy-
           \int_{|y|>R}(1-\chi_R(|y|))\cdot\frac{v(x_0+y)}{|y|^2}\,dy\\
     =&~\frac{v(x_0)}{R}+\int_{|y|>R}u(x_0+y)\cdot
        \pa_y\left(\frac{1-\chi_R(|y|)}{|y|^2}\right)\,dy\\
     =&~\frac{v(x_0)}{R}+\int_R^\infty(u(x_0+y)-u(x_0-y))\cdot
        \pa_y\left(\frac{1-\chi_R(y)}{y^2}\right)\,dy\\
     \geq&~\frac{v(x_0)}{R}-V\cdot\frac{2}{R^2}.
   \end{align*}
   where in the last inequality,
   \begin{align*}
     \int_R^\infty\left|\pa_y\left(\frac{1-\chi_R(y)}{y^2}\right)\right|\,dy
     \leq&\, \int_R^\infty \left(\frac{|\chi_R'(y)|}{y^2}
       +\frac{2(1-\chi_R(y))}{y^3}\right)\,dy\\
     \leq&\,\int_R^{2R}\frac{2}{R}\cdot\frac{1}{y^2}\,dy+\int_R^\infty \frac{2}{y^3}\,dy
     =\frac{2}{R^2}.
   \end{align*}
   Choose $R=\frac{4V}{v(x_0)}$ and $R=\frac{1}{\pi}$. Then we obtain
   \eqref{est:lamv1} and \eqref{est:lamv2} respectively.
 \end{proof}

  Since $\Lambda v(x_0)>0$,  the good term $\G_1(x_0)$ can be bounded by
  \begin{equation}\label{G1est}
    \G_1(x_0)=-\varphi(x_0)\cdot\Lambda v(x_0)\leq
    -\phim\cdot\Lambda v(x_0),
  \end{equation}
 where $\phim$ is a constant defined in \eqref{phimin}.

Next, we aim to control the bad term $\B_1(x_0)$.
Applying \eqref{eq:psix} and \eqref{eq:phix} we obtain
\begin{align*}
  |\B_1(x_0)|\leq&\,\frac{v(x_0)+3|Hv(x_0)|}{m_0^2}\cdot v(x_0)+
  \frac{3v(x_0)+|Hv(x_0)|}{m_0^2}\cdot |Hv(x_0)|\\
  \leq&\, \frac{4}{m_0^2}(v(x_0)^2+|Hv(x_0)|^2).
\end{align*}
The first term can be bounded by $\Lambda v(x_0)$ directly using
\eqref{est:lamv1}:
\begin{equation}\label{est:Bv}
  \frac{4}{m_0^2}v(x_0)^2\leq \frac{32\pi}{m_0^2}V\cdot\Lambda
  v(x_0).
\end{equation}
For the second term involving $|Hv(x_0)|$, we need an additional estimate.
\begin{lemma}[Control of $|Hv(x_0)|$]Under the same assumptions in
  Lemma \ref{lem:nmp},
  \begin{equation}\label{est:Hv}
    Hv(x_0)^2\leq \frac{16}{\pi} V\cdot \Lambda v(x_0).
  \end{equation}
\end{lemma}
\begin{proof}
  We split $Hv(x_0)$ into two parts.
    \[\pi\cdot Hv(x_0)=
    P.V.\int_{-R}^R\frac{v(x_0)-v(x_0+y)}{y}\,dy+
    \int_R^\infty\frac{v(x_0-y)-v(x_0+y)}{y}\,dy=I+II.\]
  The first term can be controlled by  $\Lambda v(x_0)$ using \eqref{fraclap}:
  \[|I|\leq \int_{-R}^R\frac{v(x_0)-v(x_0+y)}{|y|^2}\cdot|y|\,dy\leq
    \pi\cdot \Lambda v(x_0) \cdot R. \]
  The second term can be estimated by performing integration
  by parts.
  \begin{align*}|II|=
    &~\left|\left.\frac{2u(x_0)-u(x_0+y)-u(x_0-y)}{y}\right]_R^\infty
      -\int_R^\infty\frac{2u(x_0)-u(x_0+y)-u(x_0-y)}{y^2}\,dy\right|\\
    \leq&~\frac{2V}{R}+2V\cdot\frac{1}{R}=\frac{4V}{R}.
  \end{align*}
  Putting the two parts together, we get
  \[ \pi\cdot|Hv(x_0)|\leq \pi\cdot \Lambda v(x_0) \cdot R + \frac{4V}{R}.\]
  Taking $R=(\frac{4V}{\pi\cdot \Lambda v(x_0)})^{1/2}$, we obtain \eqref{est:Hv}.
\end{proof}

We are now ready to prove Proposition \ref{thm:ux}.
\begin{proof}[Proof of Proposition \ref{thm:ux}]
Collecting the estimates \eqref{est:Bv} and \eqref{est:Hv},
the bad term $\B_1(x_0)$ can be bounded by
\begin{equation}\label{B1est}
  |\B_1(x_0)|\leq \frac{32\pi^2+64}{\pi m_0^2}\,V\cdot\Lambda v(x_0).
\end{equation}

Coming back to \eqref{vdecompose}, we apply the estimates on
$\G_1(x_0)$ \eqref{G1est} and $\B_1(x_0)$ \eqref{B1est},
and use \eqref{eq:Vdecay}, \eqref{est:lamv2} as well as an argument similar to the one after \eqref{uminest1} to obtain
\begin{align}
  \frac{d}{dt} v_{\max}(t)\leq&\,\frac{1}{\pi}\left(-\phim+\frac{32\pi^2+64}{\pi
      m_0^2}\,V(t)\right)
      \cdot \max ((v_{\max}(t)-2\pi V(t)),0) \nonumber\\
  \leq&\,(-\sigmasec+Ce^{-\sigmaz t})v_{\max}(t)
        +C'e^{-\sigmaz t},\label{vmaxest}
\end{align}
where the coefficient
\begin{equation}\label{sigmasec}
  \sigmasec := \frac{\phim}{\pi}>\sigmaz=\frac{2\phim}{\pi^2},
\end{equation}
and the constants $C=\frac{32\pi^2+64}{\pi^2m_0^2}(M_0-m_0)$ and $C'=2\phim(M_0-m_0)$.
This directly implies the exponential decay
\begin{align*}
v_{\max}(t)\leq&\,e^{-\sigmasec t+\frac{C}{\sigmaz}(1-e^{-\sigmaz t})}
\left[v_{\max}(0)+
\int_0^t e^{\sigmasec\tau-\frac{C}{\sigmaz}(1-e^{-\sigmaz\tau})}\cdot C'e^{-\sigmaz\tau}\,d\tau\right]\\
\leq&\, \|u_0'\|_{L^\infty}e^{\frac{C}{\sigmaz}}\cdot e^{-\sigmasec t}+
 \frac{C' e^{\frac{C}{\sigmaz}}}{\sigmasec-\sigmaz} \cdot e^{-\sigmaz t}\leq C_1e^{-\sigmaz t},
\end{align*}
where $C_1$ depends on $m_0, M_0$ and $\|u_0'\|_{L^\infty}$.

The minimum $\min_xv(x,t_0)$ can be bounded from below by the same
argument. 
\end{proof}

The exponential decay estimate \eqref{eq:uxdecay} on
$\|\pa_xu(\cdot,t)\|_{L^\infty}$ allows us to obtain uniform-in-time
bound on all $\dot{H}^s$ norms. Indeed, one can directly verify the global regularity condition \eqref{BKM}
for $T=+\infty:$
\[\int_0^\infty\|\pa_xu\|_{L^\infty}(1+\|\pa_xu\|_{L^\infty}^4)\,dt
\lesssim \int_0^\infty e^{-\sigmaz t}\,dt<+\infty.\]
Note that due to the instant regularization, we can get uniform bound
on all $H^{s+k}$ norms for initial data $u_0\in H^s(\S)$
\[  \sup_{t\geq0}\Big(\min \{1 , t^k\}\left\|u(\cdot,t)\right\|_{\dot{H}^{s+k}}\Big)<+\infty,\quad\forall~k\geq0,\]
and the bound only depends on $k, m_0, M_0$ and $\|u_0\|_{\dot{H}^s}$.

Through standard Sobolev embedding, we get the following
uniform-in-time bounds on higher order derivatives: 
\begin{equation}\label{ukbound}
  \sup_{t\geq 1}\left\|\partial_x^k u(\cdot,t)\right\|_{L^\infty}\leq
  \Q_k,\quad\forall~k=1,2,\cdots.
\end{equation}
We will show that the higher order derivatives of solutions 
are not only bounded, but also decay
exponentially in time. Let us start with
\begin{corollary}\label{cor:decayHu}
  Suppose $u(x,t)$ is a global smooth periodic solution of
  \eqref{ueq5}. Then,
  \begin{equation}\label{eq:decayHuetc}
    \max\big\{\|Hu(\cdot,t)\|_{L^\infty},\,
    \|\Lambda u(\cdot,t)\|_{L^\infty},\, \|\pa_x\varphi\|_{L^\infty},\,
    \|\pa_x\psi\|_{L^\infty}\big\}\lesssim e^{-\sigmaz t}.
  \end{equation}
\end{corollary}
\begin{proof}
  For $Hu$, from \eqref{Bbound} and \eqref{eq:uxdecay}, we obtain
  \[ \|Hu(\cdot,t)\|_{L^\infty}< 2\|\pa_xu(\cdot,t)\|_{L^\infty}\leq \C_1e^{-\sigmaz t}.\]
 For $\Lambda u$, we apply \eqref{lambdauest}, \eqref{eq:uxdecay} and
 \eqref{ukbound}:
 \[ \|\Lambda u(\cdot,t)\|_{L^\infty}\lesssim
   \|\pa_xu\|_{L^\infty}(1+\log_+\|\pa_{x}^2 u\|_{L^\infty})\leq C_1(1+\log_+Q_2)e^{-\sigmaz t}.\]
 The exponential decay for $\pa_x\psi$ and $\pa_x\varphi$ now follows
 directly from estimates \eqref{eq:psix} and \eqref{eq:phix}.
\end{proof}

\subsection{Exponential decay for higher derivatives}
 \begin{proposition}
  Suppose that $u(x,t)$ is a global smooth periodic solution of
  \eqref{ueq5}. Then, for any positive integer $k=1,2,\cdots,$
  \begin{equation}\label{ukdecay}
    \left\|\partial_x^k u(\cdot,t)\right\|_{L^\infty}\leq
    C_k e^{-\sigmaz t},\quad\forall~t\geq1,
  \end{equation}
where $\sigmaz$ is defined in \eqref{eq:D}, and $C_k$ are constants that may
depend only on $m_0$, $M_0$ and $\|u_0\|_{\dot{H}^s}$.
\end{proposition}
\begin{proof}
  We prove the proposition by induction. The $k=1$ case has been shown in
  Proposition~\ref{thm:ux}.
  Now, suppose the result is true for
  $k=1,\cdots, i-1$. We will show exponential decay for $\partial_x^i u$.
Consider the evolution of $\pa_x^i u:$
\[\pi\pa_t \partial_x^i u=\frac{\pa^i}{\pa x^i} (\psi\pa_xu-\varphi\Lambda u)
  =\sum_{k=0}^i{i\choose k}\Big[\partial_x^k \psi \cdot \partial_x^{i-k+1} u
  -\partial_x^k \varphi \cdot \Lambda \partial_x^{i-k} u \Big].\]
Therefore,
\[\pi\pa_t \pa_x^i u +\psi\cdot\pa_x^{i+1} u =-\varphi\cdot \Lambda
  \pa_x^i u +\sum_{k=1}^i{i\choose k}\Big[\pa_x^k \psi \cdot \pa_x^{i-k+1} u  -\pa_x^k \varphi \cdot \Lambda \pa_x^{i-k} u \Big]=:\G_i+\B_i.\]
We claim that
\begin{equation}\label{Bndecay}
  |\B_i|\lesssim e^{-\sigmaz t}.
\end{equation}
Indeed, for $k=1$, $\pa_x\psi$ and $\pa_x\varphi$ decay
exponentially \eqref{eq:decayHuetc}, while $\pa_x^i u$ and $\Lambda
\pa_x^{i-1} u$ are bounded \eqref{ukbound}.
For $k\geq 2$, $\pa_x^k \psi$ and $\pa_x^k \varphi$ are
bounded due to \eqref{ukbound}. On the other hand, the derivatives $\pa_x^{i-k+1} u$ decay exponentially by the induction
hypothesis. Same decay happens for $\Lambda \pa_x^{i-k} u$ using an
estimate similar to \eqref{lambdauest}:
\[\|\Lambda \pa_x^{i-k} u\|_{L^\infty}=\|H \pa_x^{i-k+1} u\|_{L^\infty}\lesssim
\|\pa_x^{i-k+1} u\|_{L^\infty}\left(1+\log_+\|\pa_x^{i-k+2} u\|_{L^\infty}\right)\lesssim
e^{-\sigmaz t}.\]
This proves \eqref{Bndecay}; the constant involved in $\lesssim$ can only depend on $u_0$
(and varies with $i$ through higher order norms of $u_0$).

Fix a time $t_0$, and let $x_0$ be a point where
$\max_x \pa_x^i u (x,t_0)$ is attained. Then, we apply \eqref{Bndecay} and get
\[\pi\pa_t \pa_x^i u(x_0)=\G_i(x_0)+\B_i(x_0)\leq -\phim\cdot
  \Lambda \pa_x^i u(x_0)+Ce^{-\sigmaz t}.\]
To estimate $\Lambda \pa_x^i u (x_0)$, we apply Lemma \ref{lem:nmp},
replacing $(u, v=u')$ by $(\pa_x^{i-1} u, \pa_x^i u)$.
The amplitude $v(t)$ is replaced with
\[V_{i-1}(t):=\max_x \pa_x^{i-1} u(x,t)-\min_x \pa^{i-1}_x u(x,t)\leq
  2\|\pa_x^{i-1} u(\cdot,t)\|_{L^\infty}\leq 2C_{i-1}e^{-\sigmaz t},\]
which also decays exponentially in time by the induction hypothesis.
Therefore, using the analog of \eqref{est:lamv2} and \eqref{vmaxest}, we get
\[\frac{d}{dt}\pa_x^i u_{\max}(t)\leq-\frac{\phim}{\pi}\cdot\big(\pa_x^i u_{\max}(t)-4\pi
  C_{i-1}e^{-\sigmaz t}\big)+C e^{-\sigmaz t}
  =-\sigmasec \pa_x^i u_{\max}(t) + \left(4\phim C_{i-1}+C\right)
  e^{-\sigmaz t},\]
where $\sigmasec$ is defined in \eqref{sigmasec}.
This yields the desired bound
\[
 \pa_x^i u_{\max}(t)\leq \pa_x^i u_{\max}(1)e^{-\sigmasec(t-1)}
  +\frac{4\phim C_{i-1}+C}{\sigmasec-\sigmaz}e^{-\sigmaz t}\leq
  C_i e^{-\sigmaz t},\quad \forall~t\geq1.
\]
The coefficient $C_i$ may only depend on $m_0, M_0$ and
$\|u_0\|_{\dot{H}^{s}}$.

The minimum $\min_x \pa_x^i u (x,t)$ can be bounded from below by the analogous argument.
\end{proof}

\subsection{A universal decay rate}
So far, we have shown exponential decay for $V(t)=\max u(\cdot,t)-\min
u(\cdot,t)$ and
$\|\pa_x^k u(\cdot,t)\|_{L^\infty}$, $k\geq1$, with a decay rate
$\sigmaz$ defined in \eqref{eq:D} that strongly depends on the initial data.
Now, we aim to obtain an asymptotic large time universal decay rate that only depends on $\bar u$, and
complete the proof of Theorem \ref{mainthm1}.

\begin{proof}[Proof of Theorem \ref{mainthm1}]
First, we apply \eqref{eq:Vdecay}, \eqref{eq:decayHuetc} and improve the lower bound estimate \eqref{phimin}:
\begin{equation}\label{varphiimprove}
  \varphi(x_0,t_0)\geq \frac{m(t_0)}{M(t_0)^2+\|Hu(\cdot,t_0)\|_{L^\infty}^2}\geq
  \frac{\bar u-(M_0-m_0)e^{-\sigmaz t_0}}{(\bar u+(M_0-m_0)e^{-\sigmaz
      t_0})^2+(Ce^{-\sigmaz t_0})^2}\geq \frac{1}{\bar u}-Ce^{-\sigmaz t_0},
\end{equation}
for all sufficiently large $t_0\geq t_*(u_0)$,
where the constant $C$ depends on $u_0$ and may change from line to line.
Then, we can improve the bound in \eqref{Vdynamics} as follows:
\[
  M'(t)-m'(t)\leq
  -\frac{4}{\pi^3}\left(\frac{1}{\bar u}-Ce^{-\sigmaz t}\right)(M(t)-m(t))
  =\left(-\sigma+C e^{-\sigmaz t}\right)(M(t)-m(t))
\]
for any $t\geq t_*$,
with coefficient
\[\sigma:=\frac{2}{\pi^2\bar u}\]
that only depends on $\bar u$.
For all $t \geq t_*,$ this yields
\begin{equation}\label{Vimprove}
V(t)= M(t)-m(t)\leq V(t_*)\exp\left[-\sigma (t-t_*)+C(1-e^{-\sigmaz
(t-t_*)})\right]\leq C_0e^{-\sigma t},
\end{equation}
where $C_0= (M_0-m_0)\exp\left((\sigma-\sigmaz)t_*+C\right)$
depends only on the initial data. This finishes the proof of
\eqref{meanconv1117}.

For the decay estimate on $\pa_xu$, we apply \eqref{varphiimprove} and
\eqref{Vimprove} to improve the bound \eqref{vmaxest}
\begin{align*}
  \frac{d}{dt} v_{\max}(t)\leq&\,\frac{1}{\pi}\left(-\frac{1}{\bar u}+Ce^{-\sigmaz t}+\frac{32\pi^2+64}{\pi
      m_0^2}\,V(t)\right)
      \cdot \max((v_{\max}(t)-2\pi V(t)),0)\\
  \leq&\,\left(-\frac{\pi}{2}\sigma+C e^{-\sigmaz t}\right)v_{\max}(t)
        +\frac{2C_0}{\bar u}\cdot e^{-\sigma t}.
\end{align*}
This directly implies
\begin{align*}
v_{\max}(t)\leq&\,e^{-\frac{\pi}{2}\sigma(t-t_*)+\frac{C'}{\sigmaz}(e^{-\sigmaz t_*}-e^{-\sigmaz t})}
\left[v_{\max}(t_*)+
\int_{t_*}^t e^{\frac{\pi}{2}\sigma(\tau-t_*)-\frac{C'}{\sigmaz}(e^{-\sigmaz t_*}-e^{-\sigmaz \tau})}\cdot \frac{2C_0}{\bar u}e^{-\sigma\tau}\,d\tau\right]\\
\leq&\, \Q_1\cdot e^{\frac{\pi}{2}\sigma t_*+\frac{C'}{\sigmaz}e^{-\sigmaz t_*}}
 \cdot e^{-\frac{\pi}{2}\sigma t}+
 e^{\frac{C'}{\sigmaz}e^{-\sigmaz t_*}}\cdot \frac{2C_0}{\bar u}
 \cdot\frac{1}{(\frac{\pi}{2}-1)\sigma} \cdot e^{-\sigma t}\leq C_1e^{-\sigma t},
\end{align*}
for $t\geq t_*$, as $\frac{\pi}{2}>1$.

Similar arguments can be applied to higher derivatives. This concludes
the proof of \eqref{higherder1117}.
\end{proof}

\section{Rigorous Derivation: The Setup}\label{rigdersetup}

In this section, we turn towards proving our second main result, Theorem~\ref{mainthm2}.
We begin by introducing relevant notation and overall set up
needed to rigorously connect the main equation \eqref{maineq} with the evolution of polynomial roots under differentiation.


\subsection{Measurement of error}

To start, recall the way we measure closeness between a discrete set
of roots $\{x_j\}_{j=1}^{2n}$ and a continuous distribution $u(x)$.

In the introduction, we defined the error
\begin{equation}\label{Ej}
  E_j=x_{j+1}-x_j-\frac{1}{2nu(\xb_j)},\quad j=1,\cdots, 2n; \,\,\, \|E\|_\infty = \max_{1 \leq j \leq 2n} |E_j|
\end{equation}

We will assume that initially, the roots of the polynomial $p_{2n}$ obey \eqref{Ej}
with the initial density $u_0$ and $\|E^0\|_\infty \leq \M_0 n^{-1-\epsilon}.$
We will assume that $u_0(x) \in H^s(\S),$ $s >7/2,$ and that $u_0(x) \geq u_{0,\min} >0$
for all $x \in \S,$ so that the global regularity results apply to solutions of the equation \eqref{maineq}
with initial data $u_0.$ Note that in all remaining sections we will denote the minimum of $u_0(x)$ and $u(x,t)$ by
$u_{0,\min}$ and $u_{\min}(t)$ respectively, as the letter $m$ that we used before will be reserved for a different role.
We aim to control the growth in time of the error $E_j^t$ between the solution $u(x,t)$ of \eqref{maineq} and
the zeroes of the derivative $p^{k}_{2n}(x),$ with $t=\frac{k}{2n}.$ In order to be able to do so,
we will first assume that $n$ is sufficiently large - how large will only depend
on $u_0.$ Second, some of our estimates will be done under the assumption that
\begin{equation}\label{Esmall}
  \|E^t\|_\infty=\max_{1\leq j\leq 2n}|E_j^t|\leq \M_{\max}(u_0,\M_0)(n^{-1-\epsilon} + n^{-3/2}),
\end{equation}
where $\M_{\max}$ is a constant that is independent of $n$ and
of $t,$ and only depends on $u_0$ and $\M_0.$ Such constant can certainly be found for a certain initial period of time,
and we will eventually show that with a proper choice of $\M_{\max},$ \eqref{Esmall} indeed continues to hold for all times provided that
$n \geq n_0(\tilde u_0)$ is sufficiently large.



\subsection{The decomposition}
Following the heuristic derivation in Section \ref{derivation}, let us make a related but slightly different decomposition
\begin{align}\label{ImIIm} \sum_{j=1}^{2n}\cot\frac{y_m-x_j}{2}=
  \sum_{j\in\Sm}\cot\frac{y_m-x_j}{2}+
  \sum_{j\in\Smc}\cot\frac{y_m-x_j}{2}=I_m+II_m\end{align}
  (we are doing such decomposition at any fixed time and omitting time in notation for the sake of simplicity).
  Recall that $y_m$ is the zero of the next derivative satisfying $x_m < y_m < x_{m+1}.$
The set $\Sm=\{j_-,\cdots, j_+-1\}$ consists of the indices of $x_j$ that lie
in the near field of $y_m$.
We define $j_\pm$ so that $\xb_{j_\pm-1}$ are the closest
midpoints to $y_m\pm n^{-1/2},$ respectively. Note that $j_\pm$ depend on $m,$ but
to keep notation from getting too heavy we will omit dependence. When doing two roots estimates in the following sections,
we will simply use $j_-+1$ and $j_++1$ as the cutoff indices in the decomposition for $(m+1)$st root.
Also, observe that $\Smc=\{1,\cdots,2n\}\backslash\Sm$.

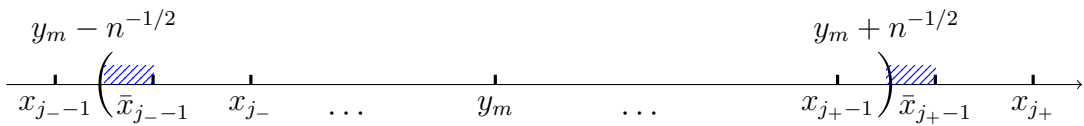
\begin{figure}[ht]
\begin{tikzpicture}[scale=1.3]
  \draw[->] (0,0) -- (11,0);
  \draw [very thick] (.5,.1) -- (.5,0) node [below] {$x_{j_--1}$};
  \draw [very thick] (1.5,.1) -- (1.5,0) node [below] {$\xb_{j_--1}$};
  \draw [very thick] (2.5,.1) -- (2.5,0) node [below] {$x_{j_-}$};
  \draw [very thick] (5,.1) -- (5,0) node [below] {$y_m$};
  \draw [very thick] (8.5,.1) -- (8.5,0) node [below] {$x_{j_+-1}$};
  \draw [very thick] (9.5,.1) -- (9.5,0) node [below] {$\xb_{j_+-1}$};
  \draw [very thick] (10.5,.1) -- (10.5,0) node [below] {$x_{j_+}$};
  \node at (3.5,-.3) {$\cdots$};
  \node at (6.5,-.3) {$\cdots$};
  \node at (1,0) {{\Huge(}};
  \node at (1, .6) {$y_m-n^{-1/2}$};
  \node at (9,0) {{\Huge)}};
  \node at (9, .6) {$y_m+n^{-1/2}$};
  \fill[pattern=north east lines, pattern color=blue] (1,0) rectangle (1.5,.2);
  \fill[pattern=north east lines, pattern color=blue] (9,0) rectangle (9.5,.2);
\end{tikzpicture}
\end{figure}

There is a mismatch between the intervals $(y_m-n^{-1/2},
y_m+n^{-1/2})$ and $(\xb_{j_--1}, \xb_{j_+-1})$ (shaded area in the
figure above). We denote the difference
\[\Dmis_m =\Dmis_{m}^-\cup \Dmis_{m}^+,\quad\text{where}\quad \Dmis_{m}^-=[
 \xb_{j_--1}, y_m-n^{-1/2}], \,\, \Dmis_{m}^+=[y_m+n^{-1/2}, \xb_{j_+-1}].\]
 Note that we do not claim that $\xb_{j_--1} \leq y_m-n^{-1/2}$ or $y_m+n^{-1/2} \leq \xb_{j_+-1}.$
The lengths of the mismatched intervals $\Dmis^\pm_m$ are of the order
$O(n^{-1})$.


\section{Rigorous Derivation: Preliminary Estimates at a Single Root}\label{rigderoneroot}

This section follows the outline of the formal derivation, but we make estimates more precise.
As we will see, single root estimates are not sufficient to rigorously connect evolution of roots under differentiation
and equation \eqref{maineq}, but these estimates will be one of the ingredients of a more in-depth argument.
Let us first estimate the far field part $II_m$ defined in \eqref{ImIIm}. 
 All estimates in the next few sections will be carried out at a fixed time, 
at which we assume \eqref{Esmall} to hold true. To lighten the presentation, we omit dependence on time in the notation until later sections where
evolution of error in time will be considered. All the constants in the estimates, including the $O$ notation, may only depend on $u_0.$
We aim to show that
\[II_m\sim 4\pi n Hu(y_m)=2n\,P.V.\int_0^{2\pi}u(y)\cot\frac{y_m-y}{2}\,dy\]
with more detailed error estimates than in heuristic derivation.
Decompose the integral into two parts and write
\begin{align}
  4\pi n Hu(y_m)-II_m=&
 ~2n\,P.V.\int_{|y-y_m|\leq n^{-1/2}}u(y)\cot\frac{y_m-y}{2}\,dy\nonumber\\
&+\left(2n\int_{|y-y_m|>n^{-1/2}}u(y)\cot\frac{y_m-y}{2}\,dy-\sum_{j\in\Smc}\cot\frac{y_m-x_j}{2}\right)\nonumber\\
  =:&\, II_{1,m}+ II_{2,m}.\label{IIm}
\end{align}
The $II_{1,m}$ term can be estimated by symmetry
\begin{equation}\label{II1m}
  \left|II_{1,m}\right|=2n\left|\int_{|y-y_m|\leq n^{-1/2}}\big(u(y)-u(y_m)\big)\cot\frac{y_m-y}{2}dy\right|
  \leq C\|\partial_x u\|_{L^\infty}n^{1/2}=O( n^{1/2}).
\end{equation}

To estimate $II_{2,m}$, we match each term in the sum of $II_m$ with the part of the integral, and
estimate the difference
\begin{align}\label{II2mj}
  &II_{2,m}^j:=\,2n\int_{\xb_{j-1}}^{\xb_j}u(y)\cot\frac{y_m-y}{2}dy-\cot\frac{y_m-x_j}{2}\\
  &=\cot\frac{y_m-x_j}{2}\left(2n\int_{\xb_{j-1}}^{\xb_j}u(y)dy-1\right)
     +2n\int_{\xb_{j-1}}^{\xb_j}u(y)\left(\cot\frac{y_m-y}{2}-\cot\frac{y_m-x_j}{2}\right)dy\nonumber
\end{align}
for $j \in S_m^c.$

Let us compute the integral
\begin{align}
  \int_{\xb_{j-1}}^{\xb_j}u(y)dy=&
  \int_{\xb_{j-1}}^{x_j}(u(\xb_{j-1})+O\big((z-\xb_{j-1}))\big)\,dz+\int_{x_j}^{\xb_j}
  \big(u(\xb_j)+O((z-\xb_j))\big)\,dz\nonumber\\
  =&\,
     \frac{x_j-x_{j-1}}{2}u(\xb_{j-1})+\frac{x_{j+1}-x_j}{2}u(\xb_j)
     + O( n^{-2})\nonumber\\
  =&\frac{1}{2n}+\frac{u(\xb_{j-1})}{2}E_{j-1}+\frac{u(\xb_j)}{2}E_j +
     O( n^{-2}).\label{intu}
\end{align}
Therefore,
\[2n\int_{\xb_{j-1}}^{\xb_j}u(y)dy-1=n\left(u(\xb_{j-1})E_{j-1}+u(\xb_j)E_j\right)
  +O( n^{-1}).\]

To estimate $\cot\frac{y_m-x_j}{2}$, we state the following
bounds on $|y_m-x_j|$.
\begin{lemma}\label{lowerbound}
  For all $n$ sufficiently large and all $j \neq m, m+1$ we have
  \begin{equation}\label{ymxj}
 \frac{|m-j|}{ u_{\min}n} \geq  |y_m-x_j|\geq\frac{|m-j|}{8 \|u\|_{L^\infty} n}. 
  \end{equation}
  Here $u_{\min} = {\rm min}_x u(x);$ we will be using this notation instead of the more compact $m(t)$ that we used
  in the first part of the paper since the variable $m$ would be too overloaded now.
\end{lemma}
\begin{proof}
  For $j\geq m+2$, estimate
\[ |y_m-x_j|\geq
  x_j-x_{m+1}=\sum_{l=m+1}^{j-1}\left(\frac{1}{2nu(\xb_l)}+E_l\right)
  \geq
  \frac{j-m-1}{4n\|u\|_{L^\infty}} \geq \frac{j-m}{8n\|u\|_{L^\infty}}.\]
 Note that in the second inequality, we use the assumption
 \eqref{Esmall}, so that $E_l\ll \frac{1}{4n u(\xb_l)}$ if $n$ is large.
 A similar argument works for $j\leq m-1$:
\[ |y_m-x_j|\geq
  x_m-x_j=\sum_{l=j}^{m-1}\left(\frac{1}{2nu(\xb_l)}+E_l\right)
  \geq
  \frac{m-j}{4n\|u\|_{L^\infty}}.\]
  The proof of the upper bound is similar.
\end{proof}
\begin{remark}
  The lower bound estimates on $|y_m-x_m|$ and $|y_{m+1}-x_{m+1}|$ are
  more subtle.  We will discuss them later in Lemma \ref{lowerbound01}. The upper bounds for $j=m,m+1$ are straightforward.
\end{remark}

Now, we have
\begin{equation}\label{cotfar}
  \left|\cot\frac{y_m-x_j}{2}\right|<\frac{2}{|y_m-x_j|}\lesssim
  \frac{n}{|j-m|}
\end{equation}
and, for $y \in (\bar x_{j-1}, \bar x_j),$
\begin{equation}\label{cotdifffar}
  \left|\cot\frac{y_m-x_j}{2}-\cot\frac{y_m-y}{2}\right|\leq
  \frac{1}{\left(\sin\frac{y_m-z}{2}\right)^2}\cdot\frac{|y-x_j|}{2}\lesssim
  \frac{n}{|j-m|^2},
\end{equation}
where $z\in(y, x_j)$ and $|y_m-z|$ has a lower bound similar to
\eqref{ymxj}.

Collecting all the estimates, we have
\[II_{2,m}^j=n\cot\frac{y_m-x_j}{2}\left(u(\xb_{j-1})E_{j-1}+u(\xb_j)E_j\right)
+O(|m-j|^{-1}+n|m-j|^{-2}).\]

To complete the estimate of $II_{2,m}$, we sum up $II_{2,m}^j$ for all
$j\in\Smc$. Note that there is a mismatch of the integral at the boundary. It
can be controlled by
\[\left|2n\int_{\Dmis_m}u(y)\cot\frac{y_m-y}{2}dy\right|\lesssim n\cdot
  n^{-1}\cdot n^{1/2}=O(n^{1/2}),\]
as the length of $\Dmis_m$ is $O(n^{-1})$, and
$y_m-y=n^{-1/2}+O(n^{-1})$ for any $y\in\Dmis_m$.

We can summarize that
\begin{equation}\label{II2m}
  II_{2,m}=\sum_{j\in\Smc}II_{2,m}^j+2n\int_{\Dmis_m}u(y)\cot\frac{y_m-y}{2}dy=4\pi
  nu(\xb_m)A_{1,m}+O(n^{1/2}),
\end{equation}
where
\begin{equation}\label{A1}
  A_{1,m}=\frac{1}{4\pi u(\xb_m)}\sum_{j\in\Smc}\cot\frac{y_m-x_j}{2}\left(u(\xb_{j-1})E_{j-1}+u(\xb_j)E_j\right).
\end{equation}
Note that a direct estimate of the term $A_{1,m}$ yields
\begin{align}\label{aux21125}
  |A_{1,m}|\leq\sum_{|m-j|\gtrsim n^{1/2}}\frac{n}{|m-j|}\cdot\|E\|_\infty =O(n\log n\|E\|_\infty).
\end{align}

Next, we turn to estimate $I_m$. The goal is to show that
\[I_m=\sum_{j\in\Sm}\cot\frac{y_m-x_j}{2}\sim 4\pi nu(\xb_m)\cot(2\pi
  n u(\xb_m)(y_m-x_m)).\]
The first step is to replace $\cot\frac{y_m-x_j}{2}$ with $\frac2{y_m-x_j}$, with
an error
\begin{equation}\label{I1}
  I_{1,m}= -\sum_{j\in\Sm}\left(\cot\frac{y_m-x_j}{2}-\frac{2}{y_m-x_j}\right).
\end{equation}
Since by \eqref{ymxj},
\[\left|\cot\frac{y_m-x_j}{2}-\frac{2}{y_m-x_j}\right| =
  O(|y_m-x_j|)=O\big((|m-j|+1)n^{-1}\big),\]
we obtain
\begin{equation}\label{I1m711} |I_{1,m}| = \sum_{|j-m|\lesssim n^{1/2}}
 O\big((|m-j|+1)n^{-1}\big) =O(1).\end{equation}

In the following lemma, we state a useful simple lower bound on
$|y_m-x_j|$, with $j=m, m+1$. Later, we will derive a more precise estimate.

\begin{lemma}\label{lowerbound01}
Suppose \eqref{Esmall} holds. Then for all sufficiently large $n$
there exists a constant $c=c(u)>0$, such that
  \begin{equation}\label{ymxm}
    \min\{y_m-x_m, x_{m+1}-y_m\}\geq c\,n^{-1}.
  \end{equation}
\end{lemma}
As follows from the proof and our earlier bounds, the threshold for $n$ depends only on the constant $\M_{\max}$ in \eqref{Esmall}, on $\| u\|_{C^1}$ and on the minimal value of $u$ - that is,
effectively, on $u_0$. A more precise value for $c(u)$ will be computed below in \eqref{ymxm122}.
\begin{proof}
  We start with writing $I_m+II_m=0$ as
  \begin{equation}\label{lbdecomp}
    \sum_{j\in\Sm}\frac{2}{y_m-x_j}-I_{1,m}+(4\pi n
    Hu(y_m)-II_{1,m}-II_{2,m})=0.
  \end{equation}
  Let us first assume $y_m$ is closer to $x_m$.
  Split the sum into three pieces
  \begin{equation}\label{I1msplit117} \sum_{j=j_-}^{j_+-1}\frac{2}{y_m-x_j}
    =\frac{2}{y_m-x_m}+
    \sum_{l=1}^{j_+-m-1}\left(\frac{2}{y_m-x_{m+l}}+\frac{2}{y_m-x_{m-l}}\right)
    +\sum_{j=j_-}^{2m-j_+}\frac{2}{y_m-x_j}.\end{equation}
  Here, we shall use the convention on summation notation that
  \begin{equation}\label{convention117}
  \sum_{j=a}^{a-1}=0 \,\,\,{\rm and}\,\,\,
  \sum_{j=a}^b=-\sum_{j=b+1}^{a-1}\,\,\, {\rm if}\,\,\, b\leq a-2.
  \end{equation}

  The second term on the right hand side of \eqref{I1msplit117} matches the $m\pm l$ terms. Compute
  \[\frac{2}{y_m-x_{m+l}}+\frac{2}{y_m-x_{m-l}}=
    \frac{2y_m-x_{m+l}-x_{m-l}}{(y_m-x_{m+l})(y_m-x_{m-l})}.\]
  For the numerator, observe that
  \begin{align}x_{m+l}+x_{m-l}=&\,2x_m+\sum_{k=1}^l\left(\frac{1}{2nu(\xb_{m+k-1})}+E_{m+k-1}-\frac{1}{2nu(\xb_{m-k})}-E_{m-k}\right)\nonumber\\
    =&\,2x_m+O( l^2n^{-2}+l\|E\|_\infty).\label{xsym}
  \end{align}
  For the denominator, we use Lemma \ref{lowerbound} for $|y_m-x_j|$
  with $j\neq m+1$. For $|y_m-x_{m+1}|$, as we assume $y_m$ is closer
  to $x_m$, we get
  \[|y_m-x_{m+1}|\geq \frac{x_{m+1}-x_m}{2}\geq
    \frac{1}{4nu(\xb_m)}-\|E\|_\infty.\]
  Therefore, we can bound the second term on the right hand side of \eqref{I1msplit117} by
  \begin{align}
    \sum_{l=1}^{j_+-m-1}\left|\frac{2}{y_m-x_{m+l}}+\frac{2}{y_m-x_{m-l}}\right|\leq
    &\,\sum_{l=1}^{\sim n^{1/2}}O\left(\frac{(y_m-x_m)+
      l^2n^{-2}+l\|E\|_{L^\infty}}{l^2n^{-2}}\right)\nonumber\\
    =&\,O(n+n^2\log n\|E\|_\infty),\label{sumcancel}
  \end{align}
  which is $O(n)$ if \eqref{Esmall} is assumed.

  The third part on the right hand side of \eqref{I1msplit117} represents the mismatched terms, as there might not
  be precisely the same number of points in $\Sm$ on each side of $m$, namely
  $m-(j_--1)\neq j_+-m$. However, we can estimate the number of mismatched indices. 
  Indeed, from the definition of $\Sm$, we have
\begin{align}\label{aux11125} x_{j_+}-x_{m+1}=n^{-1/2}+O(n^{-1}),\quad x_m-x_{j_--1}=n^{-1/2}+O(n^{-1}).\end{align}
On the other hand,
\begin{align*}
  x_{j_+}-x_{m+1}=&\sum_{l=m+1}^{j_+-1}\left(\frac{1}{2nu(\xb_l)}+E_l\right)
  =\frac{j_+-1-m}{2nu(\xb_m)}+O(n^{1/2}\|E\|_\infty+n^{-1}),\\
  x_m-x_{j_--1}=&\sum_{l=j_--1}^{m-1}\left(\frac{1}{2nu(\xb_l)}+E_l\right)
  =\frac{m-j_-+1}{2nu(\xb_m)}+O(n^{1/2}\|E\|_\infty+n^{-1}).
\end{align*}
Taking the difference and using \eqref{aux11125}, we get
\begin{equation}\label{mismatch}
  2m-j_+-j_-=O(1+n^{3/2}\|E\|_\infty).
\end{equation}
So, the third term can be controlled by
\begin{equation}\label{3t117}
\left|\sum_{j=j_-}^{2m-j_+}\frac{2}{y_m-x_j}\right|\leq O(1+n^{3/2}\|E\|_\infty) \cdot\frac{2}{O(n^{-1/2})}=O(n^{1/2}+n^2 \|E\|_\infty).
\end{equation}

Putting together all the estimates \eqref{II1m}, \eqref{II2m}, \eqref{aux21125}, \eqref{I1m711}, \eqref{sumcancel}, \eqref{3t117} into \eqref{lbdecomp}, we obtain
\[\frac{2}{y_m-x_m}=-4\pi nHu(y_m)+O(n+n^2\log n\|E\|_\infty)=O(n).\]
This implies a lower bound on $|y_m-x_m|\gtrsim n^{-1}$.

In the case when $y_m$ is closer to $x_{m+1}$, we split the sum
differently, singling out the $(m+1)$st term. Then, a similar argument
yields the same lower bound.
\end{proof}

Next, we approximate $\frac2{y_m-x_j}$ by
$\frac2{y_m-\tilde{x}_j}$, where $\tilde{x}_j=x_m+\frac{j-m}{2n u(\xb_m)}$ are equally distributed
points defined in \eqref{xtilde}. Recall the cotangent identity
\[\sum_{j\in\mathbb{Z}}\frac{2}{y_m-\tilde{x}_j}=
  \sum_{j\in\mathbb{Z}}\frac{2}{y_m-x_m-\frac{j-m}{2nu(\xb_m)}}=
  4\pi nu(\xb_m)\cot(2\pi n u(\xb_m)(y_m-x_m)).\]
Split the sum into two parts
\[\sum_{j\in\mathbb{Z}}\frac{2}{y_m-\tilde{x}_j}=
\sum_{j\in\mathbb{Z}\backslash\Sm}\frac{2}{y_m-\tilde{x}_j}+\sum_{j\in\Sm}\frac{2}{y_m-\tilde{x}_j},\]
and group the second term with $\sum_{j\in\Sm}\frac2{y_m-x_j}$. Define
\begin{equation}\label{I2}
  I_{2,m}=\sum_{j\in\mathbb{Z}\backslash\Sm}\frac{2}{y_m-\tilde{x}_j}
\end{equation}
and
\begin{equation}\label{I3}
  I_{3,m}=\sum_{j\in\Sm}\left(\frac{2}{y_m-\tilde{x}_j}-\frac{2}{y_m-x_j}\right)=\sum_{j\in\Sm\backslash\{m\}}\frac{-2(x_j-\tilde{x}_j)}{(y_m-\tilde{x}_j)(y_m-x_j)}.
\end{equation}
Note that we can exclude the term $j=m$ as $\tilde{x}_m=x_m$.
This yields
\begin{equation}\label{Im}
  I_m=4\pi nu(\xb_m)\cot(2\pi n u(\xb_m)(y_m-x_m))-I_{1,m}-I_{2,m}-I_{3,m}.
\end{equation}

It remains to estimate $I_{2,m}$ and $I_{3,m}$.

For $I_{2,m}$, we match the terms $m\pm k$ over the summation range
\begin{equation}\label{I2match}
  \left|\frac{2}{y_m-x_m-\frac{k}{2nu(\xb_m)}}+\frac{2}{y_m-x_m+\frac{k}{2nu(\xb_m)}}\right|\lesssim
  \frac{|y_m-x_m|}{k^2n^{-2}}.
\end{equation}
As before, the number of mismatched indices is at most $O(1+n^{3/2}\|E\|_\infty)$ (see \eqref{mismatch}).
Thus we obtain
\begin{align}
  I_{2,m}=&\,\sum_{k=j_+-m}^\infty\left(\frac{2}{y_m-\tilde{x}_{m+k}}+\frac{2}{y_m-\tilde{x}_{m-k}}\right)+\sum_{j=2m-j_++1}^{j_--1}\frac{2}{y_m-\tilde{x}_j} \\
  =&\,\sum_{k\gtrsim n^{1/2}}O\left(\frac{n^{-1}}{k^2n^{-2}}\right) + O(1+n^{3/2}\|E\|_\infty))\cdot\frac{2}{O(n^{-1/2})}=O(n^{1/2}+n^2\|E\|_\infty). \label{I2m117a}
\end{align}
Finally, let us estimate $I_{3,m}$. Compute (using convention \eqref{convention117} for $j<m$ case)
\begin{align}
  x_j-\tilde{x}_j=&
 \sum_{l=m}^{j-1}\left(\frac{1}{2nu(\xb_l)}+E_l\right)-\frac{j-m}{2nu(\xb_m)}=\sum_{l=m}^{j-1}E_l+\frac{1}{2n}\sum_{l=m}^{j-1}\left(\frac{1}{u(\xb_l)}-\frac{1}{u(\xb_m)}\right)\nonumber\\
  =&\sum_{l=m}^{j-1}E_l+O( n^{-2}|m-j|^2). \label{xtildexj}
\end{align}

For the denominators, we make use of the lower bound estimates
\eqref{ymxj} and \eqref{ymxm}.
\[|y_m-x_j|>c|m-j|n^{-1},\quad\forall~j\neq m.\]
Similar estimates are also valid for $|y_m-\tilde{x}_j|$. Indeed,
for $j\leq m-1$,
\[y_m-\tilde{x}_j>x_m-\tilde{x}_j=\frac{m-j}{2nu(\xb_m)},\]
for $j\geq m+2$,
 \[\tilde{x}_j -y_m>\frac{j-m}{2nu(\xb_m)}-(x_{m+1}-x_m)
 >\frac{j-m-1}{2nu(\xb_m)}-\|E\|_\infty>\frac{j-m}{6nu(\xb_m)},\]
and for $j=m+1$,
\[\tilde{x}_{m+1}-y_m=x_{m+1}-y_m-E_m>cn^{-1}-\|E\|_\infty>\frac{c}{2}n^{-1}\]
for all sufficiently large $n.$

Collecting all the estimates and using the convention \eqref{convention117}, we arrive at
\begin{align}
  I_{3,m}=&\, \sum_{j\in\Sm\backslash\{m\}}
  \left(\frac{-2\sum_{l=m}^{j-1}E_l}{(y_m-\tilde{x}_j)(y_m-x_j)}+O\left(\frac{
            n^{-2}|m-j|^2}{n^{-2}|m-j|^2}\right)\right)\nonumber\\
  =&\,4\pi nu(\xb_m) A_{2,m}+O( n^{1/2}), \label{I3mfinal117}
\end{align}
where we specify the expression $A_{2,m}$
\begin{equation}\label{A2}
  A_{2,m}=\frac{1}{2\pi nu(\xb_m)}\left(\sum_{j=j_-}^{m-1}\sum_{l=j}^{m-1}\frac{E_l}{(y_m-\tilde{x}_j)(y_m-x_j)}
  -\sum_{j=m+1}^{j_+-1}\sum_{l=m}^{j-1}\frac{E_l}{(y_m-\tilde{x}_j)(y_m-x_j)}\right).
\end{equation}
A direct estimate of this term leads to the bound
\begin{align}\label{aux51125} A_{2,m}=\sum_{1\leq|m-j|\lesssim n^{1/2}}
  O\left(n^{-1}\cdot\frac{|m-j|\|E\|_\infty}{n^{-2}|m-j|^2}\right)=O(n\log n\|E\|_\infty).\end{align}

To summarize our computations, from \eqref{IIm} and \eqref{Im}, we
have
\begin{equation}\label{single}
  4\pi nu(\xb_m)\cot(2\pi n u(\xb_m)(y_m-x_m)
  =-4\pi nHu(y_m) +I_{1,m}+I_{2,m}+I_{3,m}+II_{1,m}+II_{2,m}.
\end{equation}
Applying estimates \eqref{II1m}, \eqref{II2m}, \eqref{I1m711}, \eqref{I2m117a} and \eqref{I3mfinal117} we get
\begin{equation}\label{singleest}
  \cot(2\pi nu(\xb_m)(y_m-x_m))=- \frac{Hu(y_m)}{u(\xb_m)}+
 A_{1,m}+A_{2,m}+O(n^{-1/2}+n\|E\|_\infty),
\end{equation}
where $A_{1,m}$ and $A_{2,m}$ are given by \eqref{A1} and \eqref{A2} and are of order $O(n\log n\|E\|_\infty)$.

Without the error terms, the equation becomes \eqref{formalupdate}.

We can now return to the question addressed in Lemma \ref{lowerbound01} and derive a more precise estimate on $y_m-x_m$ and $x_{m+1}-y_m$
that we will need later.

\begin{lemma}\label{xmymlem117}
With the choice of the branch of $\arccot x$ with values in $(0,\pi),$ we have
\begin{equation}\label{ymxm117}
 y_m-x_m=\frac{1}{2\pi nu(\xb_m)}\arccot\left(- \frac{Hu(\xb_m)}{u(\xb_m)}\right)+
              O(n^{-3/2}+\log n\|E\|_\infty)
\end{equation}
and
\begin{equation}\label{lower}
    x_{m+1}-y_m=\frac{1}{2\pi nu(\xb_m)}\arccot\left(\frac{Hu(\xb_m)}{u(\xb_m)}\right)+
              O(n^{-3/2}+\log n\|E\|_\infty).
  \end{equation}
\end{lemma}
\begin{proof}
It follows directly from \eqref{singleest} and estimates for $A_{1,m}$ and $A_{2,m}$ that
 \begin{align}\label{aux1211} y_m-x_m= \frac{1}{2\pi nu(\xb_m)}\arccot\left(- \frac{Hu(y_m)}{u(\xb_m)}+
              O(n^{-1/2}+n\log n\|E\|_\infty)\right). \end{align}
Since the derivative of $\arccot$ is globally bounded, 
the estimate \eqref{ymxm117} follows. Note that $Hu(y_m)$ can be replaced in \eqref{aux1211} by $Hu(\xb_m)$ creating a difference of order $O(n^{-1})$
that can be absorbed in the error.

The second estimate could be derived in a similar way by repeating the computations taking $x_{m+1}$
as the center of approximation. But it is much simpler to observe that
\[ x_{m+1}-x_m = \frac{1}{2n u(\bar x_m)} + E_m, \]
while $\arccot (\theta) + \arccot (-\theta) = \pi$ for all $\theta.$
\end{proof}

Note that Lemma \ref{xmymlem117} yields a sharp bound for the constant $c(u)$ in Lemma \ref{lowerbound01}:
\begin{align}\label{ymxm122} c(u) = \frac{1}{2\pi
    u(\xb_m)} {\rm min} \left(\arccot\left(\frac{Hu(\xb_m)}{u(\xb_m)}\right), \pi-\arccot\left(\frac{Hu(\xb_m)}{u(\xb_m)}\right) \right)
    +  O(n^{-3/2}+\log n\|E\|_\infty). \end{align}

To determine how the errors change over a single time step, we can subtract $y_m-x_m$ from $y_{m+1}-x_{m+1},$ obtaining
\begin{align} 
E_m^{t+\Delta t} - E_m^t =& -\frac{1}{2nu(\bar y_{m}, t+\Delta t)}+\frac{1}{2nu(\xb_m,t)}+ \frac{1}{2\pi nu(\xb_{m+1})}\arccot\left(- \frac{Hu(\xb_{m+1})}{u(\xb_{m+1})}\right) \\
& - \frac{1}{2\pi nu(\xb_m)}\arccot\left(-\frac{Hu(\xb_m)}{u(\xb_m)}\right)+ O(n^{-3/2}+\log n\|E\|_\infty).   \label{Eev1p117}
\end{align}
We could now try to use the evolution equation on $u$ to absorb the main terms - but the error estimates in \eqref{Eev1p117} are too crude to yield anything useful.
In fact, the difference of the $\arccot$ terms, which are the dominant ones in \eqref{ymxm117}, is already of the order $O(n^{-2})$ in \eqref{Eev1p117}, smaller than the error.
So, in a sense, from the one point estimates it is not even clear for sure that \eqref{maineq} is the right equation.
If we tried to propagate $O(n^{-3/2})$ errors for $\sim n$ steps, a unit time, the accumulation of error could be $\sim n^{-1/2}.$ This is much larger than the spacing between
roots and would definitely destroy the approximation argument using cotangent identity in the near field. To overcome these obstacles, we need to go deeper and consider two point estimates,
where additional cancellations will yield much more favorable bounds on errors.

\section{Estimates for Pairs of Roots: General Setup}\label{rigdertworootssetup}

Before going into details, let us recap from the introduction what form of the error propagation equation we are aiming at, as well as outline informally the general
plan of the argument.
We are going to show that
\begin{equation}\label{erreq117} E_m^{t+\Delta t} - E^t_m = \sum_{j \ne m} \kappa(j,m) (E_j^t - E_m^t) + {\rm errors}, \end{equation}
where $\kappa(j,m) \sim |m-j|^{-2}.$ Essentially, the evolution is a discretization of a nonlinear fractional heat equation
$\partial_t E^t = {\mathcal L}^t E^t + n \cdot {\rm errors}$ at a scale $\sim n^{-1},$ where ${\mathcal L}^t$ is a dissipative operator
that in the main order is similar to $-(-\Delta)^{1/2}.$ In order to follow the estimates below, it is useful to
understand how one can classify the error terms in \eqref{erreq117}. The most senior inhomogeneous errors in \eqref{erreq117} will have the order $O(n^{-5/2}),$
but in fact any forcing term containing no $E^t$ in \eqref{erreq117} would prevent us from proving convergence of the error to zero as time goes to infinity as stated in Theorem \ref{mainthm2}.
A crucial observation will be that all such error terms have factors involving derivatives of $u$ in front of them. By \eqref{higherder1117}
we know that all derivatives of $u$ decay exponentially in time, allowing for much better estimates.

The linear errors of the order $O(n^{-1}\|E^t\|_\infty)$ can be thought of as critical. These errors can lead to growth by a constant factor over $\sim n$ iterations, that is time $\sim 1.$
Any larger linear in $E^t$ errors are supercritical: a brutal absolute value estimate on them will lead to only exploding upper bound on the error.
The expressions $A_{1,m}$ and $A_{2,m}$ that we saw above in \eqref{A1}, \eqref{A2} are examples of supercritical terms (or, rather, they lead to supercritical terms of order $O(n^{-1}\log n \|E^t\|_\infty)$
in the final error propagation equation). We will have to use the detailed structure of the supercritical terms in order to be able to handle them; it will turn out that any such terms can be
absorbed into the dissipative sum in \eqref{erreq117}. But even critical terms could lead to exponential growth in time - all such errors, however, turn out to also have factors involving derivatives of $u$ that decay
exponentially in time.  Thus, given a threshold, in finite time that only depends on $u_0$ and this threshold, the coefficients in front of the
linear critical terms will become smaller than the threshold. This will allow us to use a sort of spectral gap estimate on dissipation to dominate these errors.
To deploy the variant of a spectral gap bound, we will need to control the mean, the zeroth mode, of $E^t.$ This will be done by using the definition of the error and conservation of $\int_{\S} u(x,t)\,dx.$

Among the nonlinear in $E$ errors, the most dangerous term will have order $O(n^{1/2}\|E^t\|_\infty^2).$ The criticality of nonlinear terms is determined by how large $\|E^t\|_\infty$ can become.
At the very least, even when the initial errors are small or zero, the inhomogeneous error terms will only allow an upper bound of $\|E^t\|_\infty \lesssim n^{-3/2}$ for $t \sim 1.$
At this level, the $O(n^{1/2}\|E\|^2_\infty)$ term is critical; but if linear critical terms allow to maintain $\|E^t\|_\infty \lesssim n^{-3/2}$ bound for arbitrary long finite time with worsening constant,
the nonlinear one would only allow a fixed finite time $\sim 1.$
However, we will be able to control the error term of $O(n^{1/2}\|E^t\|_\infty^2)$ type by absorbing it into dissipation.
There will also be subcritical errors, that will be easier to handle by making sure that $n$ is sufficiently large.

There will be many places in the argument that will generate inhomogeneous and critical errors, and to keep track of the decaying factors in front of them, we define
\begin{equation}\label{uxdelta}
\ux(t) = \max \Big\{\|\partial_x u(x,t)\|_{L^\infty}, \,\,\|\partial_x^2 u(x,t)\|_{L^\infty}, \,\,\|\partial_x^3 u(x,t)\|_{L^\infty}\Big\}.
\end{equation}
In the argument, we will usually omit the dependence of $\ux$ on $t$ to save space. Note that of course this $\ux$ has nothing to do with the $\delta$ from \eqref{modcondef}.


Now we begin with two point estimates.
Consider $x_m, x_{m+1}$ and $x_{m+2}$ at some time $t$.
Let $y_m\in(x_m, x_{m+1})$ and $y_{m+1}\in(x_{m+1},x_{m+2})$ be roots
after one more differentiation.
To estimate $y_{m+1}-y_m-x_{m+1}+x_m$, we start with \eqref{single}.

Recall that
\begin{equation}\label{single3}
  \cot(2\pi nu(\xb_m)n(y_m-x_m))=
  -\frac{Hu(y_m)}{u(\xb_m)}+\frac{1}{4\pi nu(\xb_m)} G_m,
\end{equation}
where we define
\begin{equation}\label{Gm118} G_m=I_{1,m}+I_{2,m}+I_{3,m}+II_{1,m}+II_{2,m},\end{equation}
with $I_{1,m},$ $I_{2,m},$ $I_{3,m},$ $II_{1,m}$ and $II_{2,m}$
given by \eqref{I1}, \eqref{I2}, \eqref{I3} and \eqref{IIm} respectively.
A completely analogous equality holds for $m+1;$ here we naturally choose
$S_{m+1} = \{ j_-+1, \dots, j_+ \}$ in the splitting of near field and far field terms.
All estimates that we derived for $I_{i,m}$ and $II_{i,m}$ extend to
$I_{i,m+1}$ and $II_{i,m+1}.$

Applying the mean value theorem,
we get
\begin{align*}&\cot(2\pi nu(\xb_{m+1})(y_{m+1}-x_{m+1}))-\cot(2\pi
  nu(\xb_m)(y_m-x_m))\\
  &=-\frac{1}{\sin^2z}\big(2\pi nu(\xb_{m+1})(y_{m+1}-x_{m+1})
  -2\pi nu(\xb_m)(y_m-x_m)\big),
\end{align*}
where
\begin{equation}\label{zint}
  z\in[2\pi nu(\xb_m)(y_m-x_m), 2\pi
nu(\xb_{m+1})(y_{m+1}-x_{m+1})].
\end{equation}
Therefore,
\begin{align}
  &y_{m+1}-y_m-x_{m+1}+x_m=-\frac{u(\xb_{m+1})-u(\xb_m)}{u(\xb_m)}(y_{m+1}-x_{m+1})\nonumber\\
  &\,\,+\frac{\sin^2z}{2\pi
    nu(\xb_m)}\left(\frac{Hu(y_{m+1})}{u(\xb_{m+1})}-\frac{Hu(y_m)}{u(\xb_m)}\right)
    -\frac{\sin^2z}{8\pi^2n^2u(\xb_m)^2}\big(G_{m+1}-G_m\big)\label{yyxxdiff1}
\end{align}
The first term can be replaced by
\begin{align}
  &-\frac{u(\xb_{m+1})-u(\xb_m)}{u(\xb_m)}(y_{m+1}-x_{m+1})
  =-\frac{\partial_x u(\xb_m)}{u(\xb_m)}(\xb_{m+1}-\xb_m)(y_{m+1}-x_{m+1})+O(\ux n^{-3})\nonumber\\
 &\quad=-\frac{\partial_x u(\xb_m)}{2nu^2(\xb_m)}(y_{m+1}-x_{m+1})+O(\ux n^{-3}, \ux n^{-1}\|E\|_\infty).  \label{aux1118}
\end{align}
Next, we observe that $\sin^2z=\frac{1}{1+\cot^2z}$.
Since $z$ lies in the interval appearing in \eqref{zint}, we focus on the two
endpoints. The formula \eqref{singleest} provides a useful expression for $\cot(2\pi nu(\xb_m)n(y_m-x_m)),$
while the analogous formula for $\cot(2\pi nu(\xb_{m+1})n(y_{m+1}-x_{m+1}))$ is
\begin{align}
  &\cot(2\pi nu(\xb_{m+1})n(y_{m+1}-x_{m+1}))\\
  &=- \frac{Hu(y_{m+1})}{u(\xb_{m+1})}+
    A_{1,m+1}+A_{2,m+1}+O(n^{-1/2}+n\|E\|_{\infty})\\
   & =- \frac{Hu(y_m)}{u(\xb_m)}+
 A_{1,m}+A_{2,m}+O(n^{-1/2}+n\|E\|_{\infty})\\
  &\,\quad-\left(\frac{Hu(y_{m+1})}{u(\xb_{m+1})}-\frac{Hu(y_m)}{u(\xb_m)}\right)
    +(A_{1,m+1}-A_{1,m})+(A_{2,m+1}-A_{2,m}).  \label{singleest1212}
\end{align}
Let us verify that the terms in the last line of the estimate above
are small.
For the first term, let us replace $Hu(y_m)$ and $Hu(y_{m+1})$ by
$Hu(\xb_m)$ and $Hu(\xb_{m+1})$ respectively. The difference is
\begin{align}
  &\frac{Hu(y_{m+1})-Hu(\xb_{m+1})}{u(\xb_{m+1})}-\frac{Hu(y_m)-Hu(\xb_m)}{u(\xb_m)}\nonumber\\
  &=\frac{1}{u(\xb_m)}\big(Hu(y_{m+1})-Hu(\xb_{m+1})-Hu(y_m)+Hu(\xb_m)\big)+O(\ux
    n^{-2}) \nonumber\\
  &=\frac{\pa_xHu(\xb_m)}{u(\xb_m)}\big(y_{m+1}-y_m
    -\xb_{m+1}+\xb_m\big)+O(\ux n^{-2}) \nonumber\\
  &=\frac{\pa_xHu(\xb_m)}{u(\xb_m)}\left(y_{m+1}-y_m
    -x_{m+1}+x_m+\frac{1}{4nu(\xb_m)}-\frac{1}{4nu(\xb_{m+1})}+\frac{E_m-E_{m+1}}{2}\right)+O(\ux
    n^{-2}) \nonumber\\
  &=O(\ux|y_{m+1}-y_m-x_{m+1}+x_m|)+O(\ux n^{-2}+\ux\|E\|_\infty).\label{Huxy}
\end{align}
On the other hand,
\begin{align}\label{aux31125}
\left| \frac{Hu(\xb_{m+1})}{u(\xb_{m+1})}-\frac{Hu(\xb_m)}{u(\xb_m)} \right| \leq \left\| \partial_x \left(\frac{Hu}{u}\right) \right\|_\infty |\bar x_{m+1} - \bar x_m| \leq
C(\|\partial_x u\|_\infty+\|\partial^2_x u\|_\infty)n^{-1}= O(\ux n^{-1}).
\end{align}
due to global regularity.

Next, for the term $A_{1,m+1}-A_{1,m}$, we start with a simplification
of $A_{1,m}$ by replacing $u(\xb_{j-1})$ and $u(\xb_j)$ in \eqref{A1} with
$u(\xb_m)$. The difference is bounded by $O(\ux n^{-1}|m-j|)$. We obtain after summation
\begin{equation}\label{A1simp}
  A_{1,m}=\frac{1}{4\pi}\sum_{j\in\Smc}\cot\frac{y_m-x_j}{2}(E_{j-1}+E_j)+O(\ux
  n\|E\|_\infty).
\end{equation}
Then, we can telescope the difference $A_{1,m+1}-A_{1,m}$ as follows
\begin{align*}
  &A_{1,m+1}-A_{1,m}\\&=
  \frac{1}{4\pi}\sum_{j\in\Smc}\left(\cot\frac{y_{m+1}-x_{j+1}}{2}(E_j+E_{j+1})
  -\cot\frac{y_m-x_j}{2}(E_{j-1}+E_j)\right) +O(\ux
  n\|E\|_\infty)\\
  &=\frac{1}{4\pi}\sum_{j\in\Smc}\left(\cot\frac{y_{m+1}-x_{j+1}}{2}+\cot\frac{y_{m+1}-x_{j}}{2}-\cot\frac{y_{m}-x_{j+1}}{2}-\cot\frac{y_m-x_j}{2}\right)E_j \\
  &\qquad+O(n^{1/2}\|E\|_\infty+ \ux
  n\|E\|_\infty)\\
  &\leq \sum_{|j-m|\gtrsim n^{1/2}}O\left(\frac{n^2}{|m-j|^2}\cdot
    n^{-1}\cdot\|E\|_\infty\right) +O(n^{1/2}\|E\|_\infty+\ux n\|E\|_\infty)\\
 & =O( \delta n\|E\|_\infty + n^{1/2}\|E\|_\infty).
\end{align*}
Here, the $O(n^{1/2}\|E\|_\infty)$ term in the second equality encodes
the boundary terms from the telescoped sum. In the third
step, we use the bound $|y_{m+1}-y_m| = O(n^{-1})$.

Now, for $A_{2,m+1}-A_{2,m}$, we start with a simplification  on
$A_{2,m}$ by replacing $y_m-\tilde{x}_j$ in \eqref{A2} by
$\frac{m-j}{2nu(\xb_m)}$. The difference is $y_m-x_m=O(n^{-1})$, and
it leads to an estimate
\begin{align}
  A_{2,m}=&\,\frac{1}{\pi}\left(\sum_{j=j_-}^{m-1}\sum_{l=j}^{m-1}\frac{E_l}{(m-j)(y_m-x_j)}
  +\sum_{j=m+1}^{j_+-1}\sum_{l=m}^{j-1}\frac{E_l}{(j-m)(y_m-x_j)}\right)\label{A2simp}\\
  &+n^{-1}\sum_{j\in\Sm\backslash\{m\}}O\left(n^{-1}\cdot
    |m-j|~\|E\|_\infty\cdot \frac{n^3}{|m-j|^3}\right)\nonumber\\
  =&\,\frac{1}{\pi}\left(\sum_{l=j_-}^{m-1}\sum_{j=j_-}^{l}\frac{E_l}{(m-j)(y_m-x_j)}
  +\sum_{l=m}^{j_+-2}\sum_{j=l+1}^{j_+-1}\frac{E_l}{(j-m)(y_m-x_j)}\right)+O(n\|E\|_\infty).
     \nonumber
\end{align}
We are ready to estimate the difference $A_{2,m+1}-A_{2,m}$. Let us
work on the first sum and telescope
\begin{align*}
  &\sum_{l=j_-+1}^{m}E_{l}\sum_{j=j_-}^{l-1}\frac{1}{(m-j)(y_{m+1}-x_{j+1})}
    -\sum_{l=j_-}^{m-1}E_l\sum_{j=j_-}^{l}\frac{1}{(m-j)(y_m-x_j)}\\
  &=\sum_{l=j_-+1}^{m-1}E_l\sum_{j=j_-}^{l-1}\frac{-(y_{m+1}-y_m-x_{j+1}+x_j)}{(m-j)(y_m-x_j)(y_{m+1}-x_{j+1})}+O(n\|E\|_\infty)\\
  &=\sum_{l=j_-+1}^{m-1}E_l\sum_{j=j_-}^{l-1}O\left(\frac{n^{-1}}{n^{-2}|m-j|^3}\right)+O(n\|E\|_\infty)=O(n\|E\|_\infty).
\end{align*}
The second sum can be treated similarly.
Note that the estimates above on $A_{1,m+1}-A_{1,m}$ and
$A_{2,m+1}-A_{2,m}$ are not optimal, but they, along with \eqref{Huxy}, \eqref{aux31125}, \eqref{singleest} and \eqref{singleest1212}, are sufficient to obtain
the following bound on $\cot z$:
\[\cot z=- \frac{Hu(\xb_m)}{u(\xb_{m})}+
    A_{1,m}+A_{2,m}+O(n^{-1/2}+n\|E\|_\infty).\]
Here we also replaced $y_m$ with $\xb_m$ in the argument of the Hilbert transform generating $O(n^{-1})$ error. Then, using \eqref{Esmall}, \eqref{aux21125} and \eqref{aux51125},
we have
\begin{align}
  \sin^2z=&\,\frac{1}{1+\cot^2z}\nonumber\\
  =&\,\frac{1}{1+\left(\frac{Hu(\xb_m)}{u(\xb_m)}\right)^2-2\frac{Hu(\xb_m)}{u(\xb_m)}(A_{1,m}+A_{2,m})+O(n^{-1/2}+n\|E\|_\infty+n^2 (\log n)^2 \|E\|_\infty^2)}\nonumber\\
  =&\,\frac{u(\xb_m)^2}{u(\xb_m)^2+Hu(\xb_m)^2}
     +\frac{2 u(\xb_m)^3Hu(\xb_m)}{(u(\xb_m)^2+H
     u(\xb_m)^2)^2}(A_{1,m}+A_{2,m}) \nonumber\\
     &\,+O(n^{-1/2}+n\|E\|_\infty+n^2 (\log n)^2  \|E\|_\infty^2).\label{sinz}
\end{align}

Now we derive from \eqref{yyxxdiff1}, \eqref{Huxy} and \eqref{sinz} that
\begin{align}
  y_{m+1}&\,-y_m-x_{m+1}+x_m=-\frac{u(\xb_{m+1})-u(\xb_m)}{u(\xb_m)}(y_{m+1}-x_{m+1})\nonumber\\
  &\, +\frac{u(\xb_m)}{2\pi
    n(u(\xb_m)^2+H
    u(\xb_m)^2)}\left(\frac{Hu(\xb_{m+1})}{u(\xb_{m+1})}-\frac{Hu(\xb_m)}{u(\xb_m)}\right) \nonumber\\
  &\,+\frac{u(\xb_m)^2Hu(\xb_m)}{\pi
    n (u(\xb_m)^2+H u(\xb_m)^2)^2}(A_{1,m}+A_{2,m})\left(\frac{Hu(\xb_{m+1})}{u(\xb_{m+1})}-\frac{Hu(\xb_m)}{u(\xb_m)}\right) \nonumber\\
  &\,+O(\ux
    n^{-1}|y_{m+1}-y_m-x_{m+1}+x_m|+\ux n^{-3}+\ux n^{-1}\|E\|_\infty) \nonumber\\
 &\,+O\Big((n^{-1/2}+n\|E\|_\infty+n^2 (\log n)^2 \|E\|_\infty)\cdot
   n^{-1}\cdot \ux n^{-1}\Big)
  +\frac{\sin^2z}{8\pi^2n^2u(\xb_m)^2}(G_m-G_{m+1}). \nonumber
\end{align}
Here in the penultimate term on the right hand side we used \eqref{aux31125}.
Using \eqref{aux1118}, we can further simplify
\begin{align}
 y_{m+1}&\,-y_m-x_{m+1}+x_m  = -\frac{\pa_xu(\xb_m)}{2nu(\xb_m)^2}(y_{m+1}-x_{m+1}) \\
 &\,   +\frac{1}{4\pi
     n^2(u(\xb_m)^2+Hu(\xb_m)^2)}\pa_x\left(\frac{Hu}{u}\right)(\xb_m) \nonumber\\
  &\,+\frac{u(\xb_m)Hu(\xb_m)}{2\pi
    n^2 (u(\xb_m)^2+H u(\xb_m)^2)^2}(A_{1,m}+A_{2,m})
    \pa_x\left(\frac{Hu}{u}\right)(\xb_m) \nonumber\\
  &\, + O(\ux n^{-1}|y_{m+1}-y_m-x_{m+1}+x_m|+\ux n^{-5/2}+\ux
    n^{-1}\|E\|_\infty+ \ux (\log n)^2 \|E\|_\infty^2) \\ &\,+\frac{\sin^2z}{8\pi^2n^2u(\xb_m)^2}(G_m-G_{m+1}).\label{yyxxdiff}
\end{align}

The equation \eqref{yyxxdiff} will be our starting point in derivation of the error propagation estimates.
The key task is to obtain sufficiently strong control over $G_m-G_{m+1}$ (recall that $G_m$ is defined in \eqref{Gm118}).
Before starting this task, let us record the following corollary of \eqref{yyxxdiff} that we will need later.
\begin{lemma}\label{lem:yyxxrough}
  We have
  \begin{equation}\label{yyxxrough}
    |y_{m+1}-y_m-x_{m+1}+x_m| = O(n^{-3/2}+\log n\|E\|_\infty+ (\log n)^2 \|E\|_\infty^2).
  \end{equation}
\end{lemma}
\begin{proof}
  Let us recall that $A_{1,m}$ and $A_{2,m}$ have order $O(n\log
  n\|E\|_\infty)$, and note that it follows from  \eqref{Gm118}, \eqref{single} and \eqref{singleest} that
  \[G_m=4\pi
    nu(\xb_m)(A_{1,m}+A_{2,m})+O(n^{1/2}+n^2\|E\|_\infty)=O(n^{1/2}+n^2\log
    n\|E\|_\infty). \]
  Similarly, $G_{m+1}=O(n^{1/2}+n^2\log n\|E\|_\infty)$.
  Inserting these estimates into \eqref{yyxxdiff}, we obtain
  \begin{align*}
   &y_{m+1}-y_m-x_{m+1}+x_m=O(\ux n^{-2})+O(\ux
            n^{-1}\log n\|E\|_\infty)\\
   &\qquad+O(\ux n^{-1}|y_{m+1}-y_m-x_{m+1}+x_m|+\ux n^{-5/2}+\ux
     n^{-1}\|E\|_\infty+\ux (\log n)^2 \|E\|_\infty^2)\\
   &\qquad+O(n^{-3/2}+\log n\|E\|_\infty).
  \end{align*}
  Move the $O(\ux n^{-1}|y_{m+1}-y_m-x_{m+1}+x_m)|$ term to the left
  hand side and simplify the equation. This yields the desired bound
  \begin{align*}
    y_{m+1}-y_m-x_{m+1}+x_m=&\,\left(1+O(\ux n^{-1})\right)
                              \cdot O\left(n^{-3/2}+\log n\|E\|_\infty+(\log n)^2 \|E\|_\infty^2\right)\\
    =&\,O\left(n^{-3/2}+\log n\|E\|_\infty+ (\log n)^2 \|E\|_\infty^2\right).
  \end{align*}
\end{proof}

The estimate \eqref{yyxxrough} is suboptimal, and cannot be used to establish sufficient control
over propagation of error. We will only use it as an ingredient for deriving more precise estimates.
 The improvements are possible because in the proof of Lemma
\ref{lem:yyxxrough}, we did not make use of cancellations in the
difference $G_m-G_{m+1}$.
In the next two sections, we focus on obtaining improved error estimates of the
major term $G_m-G_{m+1}$.


\section{Estimates for pairs of roots: near field errors}\label{nearfieldtworoots}

Let us first estimate $I_{1,m}-I_{1,m+1}$; recall \eqref{I1} which defined
\[ I_{1,m}= -\sum_{j\in\Sm}\left(\cot\frac{y_m-x_j}{2}-\frac{2}{y_m-x_j}\right). \]
\begin{lemma}\label{lem:I1diff}
We have
\begin{equation}\label{I1diff}
  I_{1,m}-I_{1,m+1}=O\big(n^{1/2}|y_{m+1}-y_m-x_{m+1}+x_m|+\ux n^{-1}+n^{1/2}\|E\|_\infty \big).
\end{equation}
\end{lemma}
\begin{proof}
  We start with a Laurent series expansion for $\cot$ near zero: 
\begin{align}\label{cotexp1110} \cot\frac{y_m-x_j}{2}= \sum_{k=0}^\infty \frac{(-1)^k 2^{2k} B_{2k}}{(2k)!} \left(\frac{y_m-x_j}{2}\right)^{2k-1},
\end{align}
where $B_{2k}$ are the Bernoulli numbers, \[ B_{2k} \sim (-1)^{k-1} 4 \sqrt{\pi k} \left( \frac{k}{\pi e} \right)^{2k} \] for large $k$.
Observe that for any integer $p>0,$
\begin{align} \left| \sum_{j \in \Sm} \left( (y_m-x_j)^p - (y_{m+1}-x_{j+1})^p \right) \right| \leq \sum_{j \in \Sm} |y_m-x_j-y_{m+1}+x_{j+1}|  p C^{p-1} \left( \frac{|j-m|^{p-1}}{n^{p-1}} \right)  \\
=p C^{p-1}\sum_{j \in \Sm} \left| (y_m -x_m -y_{m+1}+x_{m+1}) -\frac{1}{2n u(\xb_m)}-E_m + \frac{1}{2nu(\xb_j)}+E_j \right| \left( \frac{|j-m|^{p-1}}{n^{p-1}} \right) \\ \leq
 C^{p-1} \left( n^{-\frac{p}{2} +1}|y_{m+1}-y_m-x_{m+1}+x_m|+\ux n^{-\frac{p+1}{2}} +n^{-\frac{p}{2} +1}\|E\|_\infty \right),
\end{align}
where in the second step the constant $C$ comes from Lemma \ref{lowerbound} and may only depend on $u_0.$
Therefore,
\begin{align}
\left| I_{1,m}-I_{1,m+1} \right| = \left| \sum_{j \in \Sm} \sum_{k=1}^\infty \frac{(-1)^k 2 B_{2k}}{(2k)!} \left((y_m-x_j)^{2k-1}-(y_{m+1}-x_{j+1})^{2k-1}\right) \right| \\
\leq 2\sum_{k=1}^\infty \frac{ C^{2k-2} |B_{2k}|}{(2k)!} \left( n^{-\frac{2k-3}{2}}|y_{m+1}-y_m-x_{m+1}+x_m|+\ux n^{-k} +n^{-\frac{2k-3}{2}}\|E\|_\infty \right) \\
\leq C_1 \left( n^{1/2} |y_{m+1}-y_m-x_{m+1}+x_m| + \ux n^{-1} + n^{1/2} \|E\|_\infty \right)
\end{align}
for all sufficiently large $n \geq n_0(u_0)$, and with a constant $C_1$ that may only depend on $u_0.$
\end{proof}
Next, for $I_{2,m}-I_{2,m+1}$, we denote $\{\tilde{x}_j\}$ and
$\{\hat{x}_{j}\}$ equally distributed points centered at $x_m$ and
$x_{m+1}$, respectively. Recall that $I_{2,m}$ is defined by \eqref{I2}, and that
\[\tilde{x}_j=x_m+\frac{j-m}{2n u(\xb_m)},\quad
  \hat{x}_j=x_{m+1}+\frac{j-(m+1)}{2n u(\xb_{m+1})}.\]
Then, we can express
\begin{align}
  I_{2,m}-I_{2,m+1}=&\,\sum_{j\in \mathbb{Z}\backslash\Sm}\frac{2}{y_m-\tilde{x}_j}
  -\sum_{j\in\mathbb{Z}\backslash\Smp}\frac{2}{y_{m+1}-\hat{x}_j}\nonumber\\
  =&\,\sum_{j\in \mathbb{Z}\backslash\Sm}
     \left(\frac{2}{y_m-\tilde{x}_j}-\frac{2}{y_{m+1}-\hat{x}_{j+1}}\right)\nonumber\\
  =&\,\sum_{j\in \mathbb{Z}\backslash\Sm}
     \frac{2(y_{m+1}-y_m-x_{m+1}+x_m)+\frac{j-m}{n}\left(\frac{1}{u(\xb_m)}-\frac{1}{u(\xb_{m+1})}\right)}{(y_m-\tilde{x}_j)(y_{m+1}-\hat{x}_{j+1})}\label{I2diffsum}
\end{align}

\begin{lemma}\label{lem:I2diff}
  The following estimate holds:
  \begin{equation}\label{I2diff}
  I_{2,m}-I_{2,m+1}=O\big(n^{3/2}|y_{m+1}-y_m-x_{m+1}+x_m|+\ux
  n^{-1/2} +\ux n \|E\|_{\infty}\big).
\end{equation}
\end{lemma}

\begin{proof}
  We estimate the two parts in \eqref{I2diffsum} one by one.
  The first part can be estimated by
  \begin{align*}
    \sum_{j\in \mathbb{Z}\backslash\Sm}
    \frac{2(y_{m+1}-y_m-x_{m+1}+x_m)}{(y_m-\tilde{x}_j)(y_{m+1}-\tilde{x}_{j+1})}=&\,
    (y_{m+1}-y_m-x_{m+1}+x_m)\sum_{j\in
      \mathbb{Z}\backslash\Sm}O(n^2|m-j|^{-2})\\=&\,O\left(n^{3/2}|y_{m+1}-y_m-x_{m+1}+x_m|\right).
      \end{align*}
For the second part in \eqref{I2diffsum}, we match the $j=m\pm k$
terms over the summation range, similarly to \eqref{I2match}:
\begin{align*}
  &\frac{k}{(y_m-\tilde{x}_{m+k})(y_{m+1}-\hat{x}_{m+k+1})}+
  \frac{-k}{(y_m-\tilde{x}_{m-k})(y_{m+1}-\hat{x}_{m-k+1})}\\
  &=k\cdot\frac{2(y_m-x_m)\frac{k}{2nu(\xb_{m+1})}+2(y_{m+1}-x_{m+1})\frac{k}{2nu(\xb_m)}}{(y_m-\tilde{x}_{m+k})(y_{m+1}-\hat{x}_{m+k+1}) (y_m-\tilde{x}_{m-k})(y_{m+1}-\hat{x}_{m-k+1})}\\
  &=O(k^2\cdot n^{-2}\cdot n^4k^{-4})=O(n^2k^{-2}).
\end{align*}
Sum over $k$ and get
\[\frac{1}{n}\left(\frac{1}{u(\xb_m)}-\frac{1}{u(\xb_{m+1})}\right)\sum_{k\gtrsim
  n^{1/2}}O(n^2k^{-2})=O(\ux n^{-2})\cdot O(n^{3/2})=O(\ux n^{-1/2}).\]
As before (see \eqref{mismatch}), there could be at most $O(1+n^{3/2}\|E\|_\infty)$ mismatched terms, leading to the error
\[O(1+n^{3/2}\|E\|_\infty))\cdot O(\ux n^{-2})\cdot O\left(n^{1/2}\cdot(n^{1/2})^2\right)=O(\ux n^{-1/2}+\ux n\|E\|_\infty)).\]

Collecting all the estimates, we conclude the proof of \eqref{I2diff}.
\end{proof}
Finally, let us estimate $I_{3,m}-I_{3,m+1}$. Recall that \eqref{I3} defined
\[  I_{3,m}=\sum_{j\in\Sm}\left(\frac{2}{y_m-\tilde{x}_j}-\frac{2}{y_m-x_j}\right). \]
\begin{lemma}\label{lem:I3diff}
  Let $I_{3,m}$ be as defined in \eqref{I3}. Then,
  \begin{align}
  I_{3,m}-I_{3,m+1}=&\,D_{1,m}+B_{1,m}+B_{2,m}
  +O\big((n^3\|E\|_\infty+\ux
  n\log n) |y_{m+1}-y_m-x_{m+1}+x_m|\big)\nonumber\\
  &\,+O(\ux n^{-1/2}+\ux n\|E\|_\infty+\ux n^2\log n\|E\|_\infty^2), \label{I3diff}
\end{align}
where $D_{1,m}$, $B_{1,m}$, $B_{2,m}$ are defined below in \eqref{D1},
\eqref{B1} and \eqref{B2} respectively.
\end{lemma}

\begin{remark}
$D_{1,m}, B_{1,m}, B_{2,m}$ all have form
  \[D_{1,m} = \sum_{j \in \Sm} d_1(j,m) (E_j-E_m),\quad B_{i,m}=\sum_{j \in \Sm}
    b_i(j,m)(E_j-E_m).\]
In particular, one can see from \eqref{d1} that
$d_1(j,m)>0$. Therefore, $D_{1,m}$ is a dissipative term. We will show
later in Section \ref{properr} that $B_{1,m}$ and  $B_{2,m}$ are supercritical but can be absorbed into the dissipation.
\end{remark}

\begin{proof}
Let us consider the difference between the expressions appearing in the $(j+1)$st term in $I_{3,m+1}$
and the $j$th term in $I_{3,m}$. Note that $j=m$ summands vanish identically.  We have
\begin{equation}\label{cotapp1}
  \frac{2}{y_{m+1}-x_{j+1}}-\frac{2}{y_m-x_j}=
  \frac{-2(y_{m+1}-y_m-x_{m+1}+x_m)+\frac{1}{n}(\frac{1}{u(\xb_j)}-\frac{1}{u(\xb_m)})+2(E_j-E_m)}{(y_{m+1}-x_{j+1})(y_m-x_j)},
\end{equation}
and
\begin{equation}\label{cotapp2}
-\frac{2}{y_{m+1}-\hat{x}_{j+1}}+\frac{2}{y_m-\tilde{x}_j}=
\frac{2(y_{m+1}-y_m-x_{m+1}+x_m)-\frac{j-m}{n}(\frac{1}{u(\xb_{m+1})}-\frac{1}{u(\xb_m)})}{(y_{m+1}-\hat{x}_{j+1})(y_m-\tilde{x}_j)}.
\end{equation}

The last summand of \eqref{cotapp1} produces a dissipative term that is crucial for the entire argument. We
denote it as
\begin{equation}\label{D1}
  D_{1,m}=\sum_{j\in\Sm\backslash\{m\}}\frac{2(E_j-E_m)}{(y_{m+1}-x_{j+1})(y_m-x_j)}
  =:\sum_{j\in\Sm\backslash\{m\}}d_1(j,m)(E_j-E_m),
\end{equation}
where from the estimates \eqref{ymxj} and \eqref{ymxm}, we
know that the coefficients $d_1(j,m)$ satisfy
\begin{equation}\label{d1}
  d_1(j,m)=\frac{2}{(y_{m+1}-x_{j+1})(y_m-x_j)}\sim \frac{n^2}{|m-j|^2}>0,\quad\forall~j\in\Sm\backslash\{m\}.
\end{equation}
Here $\sim$ as usual means lower and upper bounds with constants that may only depend on $u_0.$

We combine the first parts on the right hand side in \eqref{cotapp1} and \eqref{cotapp2}. The corresponding term is
\[T_{1,m}^j=2(y_{m+1}-y_m-x_{m+1}+x_m)\left(\frac{1}{(y_{m+1}-\hat{x}_{j+1})(y_m-\tilde{x}_j)}-\frac{1}{(y_{m+1}-x_{j+1})(y_m-x_j)}\right).\]
Now, we estimate the difference (using the summation convention \eqref{convention117}):
\begin{align}
  &~\frac{1}{(y_{m+1}-\hat{x}_{j+1})(y_m-\tilde{x}_j)}-\frac{1}{(y_{m+1}-x_{j+1})(y_m-x_j)} \pm \frac{1}{(y_{m+1}-\hat{x}_{j+1})(y_m-x_j)} \nonumber\\
  =&~\frac{\tilde{x}_j-x_j}{(y_{m+1}-\hat{x}_{j+1})(y_m-\tilde{x}_j)(y_m-x_j)}
    +\frac{\hat{x}_{j+1}-x_{j+1}}{(y_{m+1}-\hat{x}_{j+1})(y_{m+1}-x_{j+1})(y_m-x_j)}\nonumber\\
  =&~-\frac{\sum_{l=m}^{j-1}E_l+O(\ux n^{-2}|m-j|^2)}{(y_{m+1}-\hat{x}_{j+1})(y_m-\tilde{x}_j)(y_m-x_j)}
    -\frac{\sum_{l=m+1}^jE_l+O(\ux n^{-2}|m-j|^2)}{(y_{m+1}-\hat{x}_{j+1})(y_{m+1}-x_{j+1})(y_m-x_j)}\nonumber\\
  \lesssim&~\frac{|m-j|\|E\|_\infty+\ux(|j-m|/n)^2}{(|m-j|/n)^3}
  =O\left(n^3|m-j|^{-2}\|E\|_\infty+\ux n |m-j|^{-1}\right).\label{diff1oy2}
\end{align}
Here, we use the lower bound estimates \eqref{ymxj}, \eqref{ymxm} for $|y_m-x_j|$ and
$|y_{m+1}-x_{j+1}|$, and a variant of \eqref{xtildexj} for $x_j-\tilde{x}_j$ and
$x_{j+1}-\hat{x}_{j+1}$ taking into account the presence of the $\ux$ factor. Now, summing up over $j\in\Sm\backslash\{m\}$
we get
\[\sum_{j\in\Sm\backslash\{m\}}T_{1,m}^j=O\big((n^3\|E\|_\infty+\ux n\log n) |y_{m+1}-y_m-x_{m+1}+x_m|\big).\]

Next, we combine the second parts on the right hand side of \eqref{cotapp1} and
\eqref{cotapp2}. The corresponding term is
{\small\begin{align}T_{2,m}^j=&\,\frac{\frac{1}{n}\left(\frac{1}{u(\xb_j)}-\frac{1}{u(\xb_m)}\right)}{(y_{m+1}-x_{j+1})(y_m-x_j)} -\frac{\frac{j-m}{n}\left(\frac{1}{u(\xb_{m+1})}-\frac{1}{u(\xb_m)}\right)}{(y_{m+1}-\hat{x}_{j+1})(y_m-\tilde{x}_j)}\\
  =&\,\frac{\frac{1}{n}\left(\frac{1}{u(\xb_j)}-\frac{1}{u(\xb_m)}\right)
     -\frac{j-m}{n}\left(\frac{1}{u(\xb_{m+1})}-\frac{1}{u(\xb_m)}\right)}{(y_{m+1}-x_{j+1})(y_m-x_j)}\\
   &+\frac{j-m}{n}\left(\frac{1}{u(\xb_{m+1})}-\frac{1}{u(\xb_m)}\right)
     \left(\frac{1}{(y_{m+1}-x_{j+1})(y_m-x_j)}-\frac{1}{(y_{m+1}-\hat{x}_{j+1})(y_m-\tilde{x}_j)}\right)\\
  =:&\,T_{2,m}^{j,(1)}+T_{2,m}^{j,(2)}. \label{Tj12}
\end{align}}
For the first term $T_{2,m}^{j,(1)}$, let us estimate the numerator
\begin{align*}
  &\,\frac{1}{n}\left(\frac{1}{u(\xb_j)}-\frac{1}{u(\xb_m)}-\frac{j-m}{u(\xb_{m+1})}+\frac{j-m}{u(\xb_m)}\right)\\
  =&\,\frac{1}{n}\left[-\frac{\partial_x u(\xb_m)}{u(\xb_m)^2}(\xb_j-\xb_m)+O(\ux
     |\xb_j-\xb_m|^2)\right.\\
  &\qquad\left.+\frac{\partial_x u(\xb_m)(j-m)}{u(\xb_m)^2}(\xb_{m+1}-\xb_m)+O(\ux
    (\xb_{m+1}-\xb_m)^2)\right]\\
  =&\,-\frac{\partial_x u(\xb_m)}{nu(\xb_m)^2}\left(\xb_j-\xb_m-(j-m)
     (\xb_{m+1}-\xb_m)\right)
     +O(\ux n^{-3}|m-j|^2).
\end{align*}
Observe that
\begin{align*}
  &\xb_j-\xb_m-(j-m)(\xb_{m+1}-\xb_m)\\
  =&\,\sum_{l=m}^{j-1}\left(\frac{1}{4nu(\xb_l)}+\frac{1}{4nu(\xb_{l+1})}+\frac{E_l+E_{l+1}}{2}\right)\\
  &\,-(j-m)
    \left(\frac{1}{4nu(\xb_m)}+\frac{1}{4nu(\xb_{m+1})}+\frac{E_m+E_{m+1}}{2}\right)\\
  =&\,\sum_{l=m}^{j-1}\left(O(\ux n^{-1}\cdot n^{-1}|l-m| )+
     \frac{(E_l-E_m)+(E_{l+1}-E_{m+1})}{2}\right)\\
  =&\,O(\ux n^{-2}|m-j|^2)+\sum_{l=m}^{j-1}\left(\frac{(E_l-E_m)+(E_{l+1}-E_{m+1})}{2}\right).
\end{align*}
Note that the summation can be replaced by $\sum_{l=m+1}^{j-1}$ as the
term $l=m$ is zero. When $j\leq m-1$, the sum represents $-\sum_{j}^{m-1}$, in accordance with the summation convention \eqref{convention117}.

Then for $T_{2,m}^{(1)}=\sum_{j\in\Sm\backslash\{m\}} T_{2,m}^{j,(1)}$  we obtain
\begin{align*}
  T_{2,m}^{(1)}=&\,\sum_{j\in\Sm\backslash\{m\}}
  \left[O\big(n^2|m-j|^{-2}\cdot \ux n^{-3}|m-j|^2\big)-\sum_{l=m+1}^{j-1}
              \frac{\partial_x u(\xb_m)(E_l-E_m+E_{l+1}-E_{m+1})}{2nu(\xb_m)^2(y_{m+1}-x_{j+1})(y_m-x_j)}\right].\\
  =&\, B_{1,m}+O(\ux n^{-1/2}),
\end{align*}
where the term $B_{1,m}$ can be expressed as
\begin{align}
  B_{1,m}=&\,\frac{-\pa_xu(\xb_m)}{2nu(\xb_m)^2}\left[
   \sum_{j\in\Sm\backslash\{m\}}\sum_{l=m+1}^{j-1}\frac{(E_l-E_m)+(E_{l+1}-E_m)}{(y_{m+1}-x_{j+1})(y_m-x_j)}\right.\nonumber\\
    &\left.-(E_{m+1}-E_m) \sum_{j\in\Sm\backslash\{m\}}\sum_{l=m+1}^{j-1}\frac{1}{(y_{m+1}-x_{j+1})(y_m-x_j)}\right]\label{aux1119}\\
  =:&\sum_{l\in\Sm\backslash\{m\}}b_1(l,m)(E_l-E_m).\label{B1}
\end{align}
Let us estimate the coefficients $\{b_1(l,m)\}$.
For $l\geq m+2$, we change the order of summation:
\begin{equation}\label{aux1118a} \sum_{j=m+1}^{j_+-1}\sum_{l=m+1}^{j-1}\frac{1}{(y_{m+1}-x_{j+1})(y_m-x_j)}
  =\sum_{l=m+1}^{j_+-2}\sum_{j=l+1}^{j_+-1}\frac{1}{(y_{m+1}-x_{j+1})(y_m-x_j)}.\end{equation}
Using the lower bound estimate \eqref{ymxj}, we can estimate the sum in \eqref{aux1118a} and then the coefficient $b_1(l,m)$ by
\[  O(\ux
  n^{-1})\sum_{j=l+1}^{j_+-1}\frac{1}{|m-j|^2n^{-2}}=O(\ux n
  |m-l|^{-1}).\]
A similar calculation can be done for $l\leq m-1$.
For $l=m+1$, there is a contribution similar to \eqref{aux1118a} (where we now also use \eqref{ymxm}), and an additional summand from \eqref{aux1119} that is equal to
\begin{align*}
  \frac{\partial_x u(\xb_m)}{2nu(\xb_m)^2}\left(\sum_{j=j_-}^{m-1}\frac{j-m}{(y_{m+1}-x_{j+1})(y_m-x_j)}+\sum_{j=m+1}^{j_+-1}\frac{j-m-1}{(y_{m+1}-x_{j+1})(y_m-x_j)}\right)\\
  =\,O(\ux n^{-1})\sum_{1\leq|j-m|\lesssim
     n^{1/2}}O\left(\frac{|m-j|}{|m-j|^2n^{-2}}\right)=O(\ux n\log n).
\end{align*}
To summarize, the coefficients
\begin{equation}\label{b1}
  b_1(l,m)=\begin{cases} O(\ux n |m-l|^{-1}),& l\in\Sm\backslash{\{m,m+1\}}\\
    O(\ux n\log n),& l=m+1.
  \end{cases}
\end{equation}
Comparing with \eqref{d1}, one sees that
\begin{equation}\label{b1d1}
  |b_1(j,m)| \lesssim n^{-1/2}d_1(j,m),\quad\forall~j\in\Sm\backslash\{m\}.
\end{equation}
Hence, for all sufficiently large $n,$ $B_{1,m}$ can be absorbed into $D_{1,m}$ with bounds \eqref{d1} remaining valid.

For the second term $T_{2,m}^{j,(2)}$ in \eqref{Tj12}, we apply a calculation as in \eqref{diff1oy2} and obtain
\begin{align}\label{T22}
 T_{2,m}^{j,(2)}=  &\sum_{j\in\Sm\backslash\{m\}}\frac{j-m}{n}\left(\frac{1}{u(\xb_{m+1})}-\frac{1}{u(\xb_m)}\right)\times\\
 &\times\left(\frac{\sum_{l=m}^{j-1}E_l+O(\ux n^{-2}|j-m|^2)}{(y_{m+1}-\hat{x}_{j+1})(y_m-\tilde{x}_j)(y_m-x_j)}
   +\frac{\sum_{l=m+1}^jE_l+O(\ux n^{-2}|j-m|^2)}{(y_{m+1}-\hat{x}_{j+1})(y_{m+1}-x_{j+1})(y_m-x_j)}\right).\nonumber
\end{align}
The parts of this expression corresponding to $O(\ux n^{-2}|j-m|^2)$ in the numerator lead after summation in $j$
to the error of the order
\begin{equation}\label{easyerr119}
\lesssim \delta \sum_{j\in\Sm\backslash\{m\}} \frac{|j-m|}{n^2} \frac{n^3}{|j-m|^3} n^{-2}|j-m|^2 = O(\delta n^{-1/2}).
\end{equation}
To further simplify the remaining part, we first replace
$y_{m+1}-\hat{x}_{j+1}$ by
$x_{m+1}-\hat{x}_{j+1}=\frac{m-j}{2nu(\xb_{m+1})}$.
Note that by \eqref{ymxj}, \eqref{ymxm} we have
\begin{align}\label{aux11128} \frac{1}{y_{m+1}-\hat{x}_{j+1}}-\frac{1}{\frac{m-j}{2nu(\xb_{m+1})}}
  =\frac{-(y_{m+1}-x_{m+1})}{(y_{m+1}-\hat{x}_{j+1})\cdot\frac{m-j}{2nu(\xb_{m+1})}}
  =O\left(\frac{n}{|m-j|^2}\right).\end{align}
The difference in the remaining part of \eqref{T22} introduced by making such change does not exceed
\begin{align}\label{aux21128} \sum_{1\leq|j-m|\lesssim n^{1/2}}O\left(\frac{|j-m|}{n}\cdot\ux
    n^{-1}\cdot |m-j|\|E\|_\infty\cdot \frac{n}{|m-j|^2}\cdot
   \frac{n^2}{ |m-j|^2}\right)=O(\ux n\|E\|_\infty).\end{align}
   We can similarly replace $y_m - \tilde{x}_j$ with $\frac{m-j}{2n u(\bar x_m)}$ incurring
   the error of the same order.
Therefore, $T_{2,m}^{(2)}=\sum_{j\in\Sm\backslash\{m\}} T_{2,m}^{j,(2)}$ after a straightforward calculation
can be represented as 
\begin{align}
T_{2,m}^{(2)}=  &\frac{\pa_xu(\xb_m)}{nu(\xb_m)^2}\sum_{j\in\Sm\backslash\{m\}}
\left(\frac{\sum_{l=m}^{j-1}E_l}{\frac{m-j}{2nu(\xb_m)}\cdot (y_m-x_j)}
+\frac{\sum_{l=m+1}^jE_l}{(y_{m+1}-x_{j+1})(y_m-x_j)}\right)\nonumber\\
 &+O(\ux n^{-1/2}+ \ux n\|E\|_\infty+\ux n^2\log n\|E\|_\infty^2).\label{T22s}
\end{align}
To absorb the remaining supercritical term into $D_{1,m}$, we replace $E_l$ by $E_l-E_m$, and define
\begin{align}
  B_{2,m}=&\,\frac{\pa_xu(\xb_m)}{nu(\xb_m)^2}\sum_{j\in\Sm\backslash\{m\}}
\left(\frac{\sum_{l=m}^{j-1}(E_l-E_m)}{\frac{m-j}{2nu(\xb_m)}\cdot(y_m-x_j)}
    +\frac{\sum_{l=m+1}^j(E_l-E_m)}{(y_{m+1}-x_{j+1})(y_m-x_j)}\right)\nonumber\\
 =:&\sum_{j\in\Sm\backslash\{m\}}b_2(j,m)(E_j-E_m),\label{B2}
\end{align}
where the coefficients $b_2(j,m)$ can be estimated similarly to
$b_1(j,m)$, yielding
\begin{equation}\label{b2}
  b_2(j,m)=O(\ux n|m-j|^{-1})\lesssim n^{-1/2}d_1(j,m),\quad \forall~j\in\Sm\backslash\{m\}.
\end{equation}
Therefore, $B_{2,m}$ can be absorbed by $D_{1,m}$ as well. Notice that there is a partial cancellation between
$B_{2,m}$ and $B_{1,m}$ but we do not need to pursue it here, as it does not eliminate the supercritical term.

We are left with the difference between the sum in \eqref{T22s} and
$B_{2,m}$, which is given by
\begin{equation}\label{B2diff}
  E_m\cdot\frac{\pa_xu(\xb_m)}{nu(\xb_m)^2}\sum_{j\in\Sm\backslash\{m\}}
    \left(-\frac{2nu(\xb_m)}{y_m-x_j}+\frac{j-m}{(y_{m+1}-x_{j+1})(y_m-x_j)}\right).
\end{equation}
For the first term of the sum in \eqref{B2diff}, we make use of the cancellation
between the $j= m\pm k$ terms, as described in \eqref{sumcancel} and \eqref{3t117}.
This leads
to the error
\[O\left(\|E\|_\infty\cdot \ux n^{-1}\cdot n\cdot
  (n+n^2 \log n \|E\|_\infty)\right)=O(\ux n\|E\|_\infty + \ux n^2 \log n \|E\|_\infty^2).\]
The second summand in \eqref{B2diff} becomes the same as the first one if we
replace $y_{m+1}-x_{j+1}$ by $y_{m+1}-\hat{x}_{j+1},$ and then by $x_{m+1}-\hat{x}_{j+1}.$
The latter replacement produces the same error as in \eqref{aux11128}, \eqref{aux21128}. We calculate an
error generated by the former replacement as follows. Using an estimate parallel to \eqref{xtildexj} we get
\begin{align*}
  &\frac{1}{y_{m+1}-\hat{x}_{j+1}}-\frac{1}{y_{m+1}+x_{j-1}}
  =\frac{\hat{x}_{j+1}-x_{j+1}}{(y_{m+1}-\hat{x}_{j+1})(y_{m+1}-x_{j+1})}\\
  &=O\left(\frac{|m-j|\|E\|_\infty+\ux
  n^{-2}|m-j|^2}{(|m-j|/n)^2}\right)=O(|m-j|^{-1}n^2\|E\|_\infty+\ux).
\end{align*}
This leads to an error
\[O\left(\|E\|_\infty\cdot \ux n^{-1} \cdot
 \sum_{1\leq |j-m|\lesssim
      n^{1/2}}\left(\frac{n^2\|E\|_\infty}{|m-j|}+\ux\right)\cdot \frac{n|j-m|}{|j-m|}\right)=O\left(\ux
  n^{1/2}\|E\|_\infty+\ux n^2\log n\|E\|_\infty^2\right).\]

Collecting all the estimates, we arrive at \eqref{I3diff}.
\end{proof}

 In summary, collecting the bounds \eqref{I1diff}, \eqref{I2diff}
 and \eqref{I3diff},
the contribution into $G_m-G_{m+1}$ from the near
field is
 \begin{align}
   &(I_{1,m}-I_{1,m+1})+(I_{2,m}-I_{2,m+1})+(I_{3,m}-I_{3,m+1})\nonumber\\
  &=D_{1,m}+B_{1,m}+B_{2,m} +O\left((n^{3/2}+n^3\|E\|_\infty)|y_{m+1}-y_m-x_{m+1}+x_m|\right)\nonumber\\
   & \quad+O(\ux n^{-1/2}+\ux n\|E\|_\infty+\ux n^2\log n\|E\|_\infty^2).\label{Idiff}
 \end{align}

\section{Estimates for pairs of roots: far field errors}\label{farfieldtworoots}

We start with $II_{1,m}-II_{1,m+1}$. Recall \eqref{IIm}:
\[ II_{1,m} =2n P.V. \int_{|y-y_m| \leq n^{-1/2}} u(y) \cot \frac{y_m-y}{2} \,dy. \]
\begin{lemma}\label{lem:II1mdiff}
  Let $II_{1,m}$ be defined as in \eqref{IIm}. Then,
  \begin{equation}\label{II1mdiff}
    II_{1,m}-II_{1,m+1}=O(\ux n^{-1/2}).
  \end{equation}
\end{lemma}
\begin{proof}
  We have
  \begin{align*}
  &II_{1,m}-II_{1,m+1}=2n\,P.V.\int_{-n^{-1/2}}^{n^{-1/2}}\Big(u(y_{m+1}+z)-u(y_m+z)\Big)\cot\frac{z}{2}\,dz\\
  &=2n \int_0^{n^{-1/2}}\Big(u(y_{m+1}+z)-u(y_{m+1}-z)-u(y_m+z)+u(y_m-z)\Big)\cot\frac{z}{2}\,dz\\
  &=2n\int_0^{n^{-1/2}}\Big(2u'(y_{m+1})z-2u'(y_m)z+O(\ux z^3)\Big)\cot\frac{z}{2}\,dz\\
  &=2n\cdot \left((y_{m+1}-y_m)\cdot O(\ux n^{-1/2})+O(\ux n^{-3/2})\right)
    =O(\ux n^{-1/2}).
\end{align*}
\end{proof}

Next, we focus on $II_{2,m}-II_{2,m+1}$. The estimate of this term is the most involved one,
and will take us a while to complete.
Let us decompose $II_{2,m}$ as in \eqref{II2m}. We will discuss the matching
summands $II_{2,m}^j$ and the mismatched part one by one. The definition of $II_{m,2}^j$ is recalled in \eqref{II2mjr} right after the statement of the lemma.
\begin{lemma}\label{lem:II2mjdiff}
  Let $II_{2,m}^j$ be defined as in \eqref{II2mj}. Then,
\begin{align}
  &\sum_{j\in\Smc}\left(II_{2,m}^j-II_{2,m+1}^{j+1}\right)=
  D_{2,m}+H_{1,m}+H_{2,m}+H_{3,m}\nonumber\\
&+nu(\xb_{j_+-1})\cot\frac{y_m-x_{j_+}}{2}(E_{j_+-1}-E_m)+nu(\xb_{j_+})\cot\frac{y_m-x_{j_++1}}{2}(E_{j_+}-E_m) \nonumber\\
&-nu(\xb_{j_--})\cot\frac{y_m-x_{j_--1}}{2}(E_{j_-}-E_m)-nu(\xb_{j_-1})\cot\frac{y_m-x_{j_--2}}{2}(E_{j_--1}-E_m) \nonumber\\
  &+O\big((n^{5/2}\|E\|_\infty+n^{1/2})|y_{m+1}-y_m-x_{m+1}+x_m|\big)\nonumber\\
  &+O(\ux n^{-1/2}+\ux n\|E\|_\infty+\ux n^2\log n\|E\|_\infty^2)
    +O(n^{1/2}\|E\|_\infty+n^2\|E\|_\infty^2),\label{II2mjdiff}
\end{align}
where $D_{2,m}, H_{1,m}, H_{2,m}$ and $H_{3,m}$ are defined below in \eqref{D2},
  \eqref{H1}, \eqref{H2} and  \eqref{H3} respectively.
\end{lemma}
\begin{remark}
$D_{2,m}, H_{1,m}, H_{2,m}$ and $H_{3,m}$ will all have form
  \[D_{2,m} = \sum_j d_2(j,m) (E_j-E_m),\quad H_{i,m}=\sum_j
    h_i(j,m)(E_j-E_m).\]
In particular, we will discover \eqref{d2} that
$d_2(j,m)>0$. Therefore, $D_{2,m}$ is a dissipative term. We will show
later that the supercritical terms $H_{1,m}, H_{2,m}, H_{3,m}$ can be absorbed into the dissipation.
\end{remark}

\begin{proof}
We first recall the definition of $II_{2,m}^j$ in \eqref{II2mj}:
\begin{equation}\label{II2mjr}
  \cot\frac{y_m-x_j}{2}\left(2n\int_{\xb_{j-1}}^{\xb_j}u(y)dy-1\right)
     +2n\int_{\xb_{j-1}}^{\xb_j}u(y)\left(\cot\frac{y_m-y}{2}-\cot\frac{y_m-x_j}{2}\right)dy.
\end{equation}
We will obtain more precise estimates on $II_{2,m}^j$, and consider the
difference $II_{2,m}^j-II_{2,m+1}^{j+1}$ only when we need to take advantage of the additional
cancellation between the two terms. In simpler contributions where we do not need to consider the difference,
the $m$th and $(m+1)$st estimates are identical and we only provide the arguments for $II_{2,m}^j.$

We start with the first term in \eqref{II2mjr}.
A refined compared with \eqref{intu} calculation yields
\begin{align*}
  \int_{\xb_{j-1}}^{\xb_j}u(y)dy=&\,
  \int_{\xb_{j-1}}^{x_j}(u(\xb_{j-1})+u'(\xb_{j-1})(z-\xb_{j-1})+O\big(\ux(z-\xb_{j-1})^2)\big)\,dz\\&
  +\int_{x_j}^{\xb_j}
  \big(u(\xb_j)+u'(\xb_j)(z-\xb_j)+O(\ux(z-\xb_j)^2)\big)\,dz\\
  =&\,\frac{1}{2n}+\frac{u(\xb_{j-1})}{2}E_{j-1}+\frac{u(\xb_j)}{2}E_j\\
  &\,   +u'(\xb_{j-1})\left(\frac{x_j-x_{j-1}}{2}\right)^2
     -u'(\xb_j)\left(\frac{x_{j+1}-x_j}{2}\right)^2+O(\ux n^{-3}).
\end{align*}
Apply the estimate to the first term of \eqref{II2mjr}, and sum over
$j \in \Smc$. We obtain
\begin{align}
  &\sum_{j\in\Smc}\cot\frac{y_m-x_j}{2}\left(2n\int_{\xb_{j-1}}^{\xb_j}u(y)dy-1\right)
    =n\sum_{j\in\Smc} \cot\frac{y_m-x_j}{2} \big(u(\xb_{j-1})E_{j-1}+u(\xb_j)E_j\big)\nonumber\\
  &+ \frac{n}{2}\sum_{j\in\Smc}\cot\frac{y_m-x_j}{2}
  \Big(u'(\xb_{j-1})(x_j-x_{j-1})^2
    -u'(\xb_j)(x_{j+1}-x_j)^2\Big)+O(\ux n^{-1}\log n).\label{II2term1}
\end{align}
For the  second sum on the right hand side of \eqref{II2term1}, we have the following
cancellation:
\begin{align}
  &u'(\xb_{j-1})(x_j-x_{j-1})^2-u'(\xb_j)(x_{j+1}-x_j)^2\nonumber\\
  &=u'(\xb_j)\left[\left(\frac{1}{2nu(\xb_{j-1})}+E_{j-1}\right)^2
    -\left(\frac{1}{2nu(\xb_j)}+E_j\right)^2\right]+O(\ux n^{-3})\nonumber\\
  &=u'(\xb_j)\left[\frac{1}{4n^2}
    \left(\frac{1}{u(\xb_{j-1})^2}-\frac{1}{u(\xb_j)^2}\right)
    +(E_{j-1}^2-E_j^2)+\left(\frac{E_{j-1}}{nu(\xb_{j-1})}-\frac{E_j}{nu(\xb_j)}\right)
    \right]+O(\ux n^{-3})\nonumber\\
  &=\frac{u'(\xb_j)}{nu(\xb_j)}(E_{j-1}-E_j)+O(\ux n^{-3}+\ux \|E\|_\infty^2+\ux
    n^{-2}\|E\|_\infty).\label{intuimprove}
\end{align}
This leads to the estimate
\begin{align}\label{aux31128}
  &\frac{n}{2}\sum_{j\in\Smc}\cot\frac{y_m-x_j}{2}
  \big(u'(\xb_{j-1})(x_j-x_{j-1})^2-u'(\xb_j)(x_{j+1}-x_j)^2\big)\\
  &=\sum_{j\in\Smc}\frac{u'(\xb_j)}{2u(\xb_j)}\cot\frac{y_m-x_j}{2}(E_{j-1}-E_j)
    +  O(\ux n^{-1}\log n +\ux
    \log n \|E\|_\infty+\ux n^2 \log n \|E\|_\infty^2).
\end{align}
For the remaining sum on the right side of \eqref{aux31128}, we can telescope and get
\begin{align*}
  &\sum_{j\in\Smc}\frac{u'(\xb_j)}{2u(\xb_j)}\cot\frac{y_m-x_j}{2}(E_{j-1}-E_j)\\
   &=\sum_{j\in\Smc\backslash\{j_+\}}\left(\frac{u'(\xb_{j+1})}{2u(\xb_{j+1})}\cot\frac{y_m-x_{j+1}}{2}
  -\frac{u'(\xb_j)}{2u(\xb_j)}\cot\frac{y_m-x_j}{2}\right)E_j
  +O(\ux n^{1/2}\|E\|_\infty)\\
   &=\sum_{|m-j|\gtrsim n^{1/2}}\left(O(\ux n^{-1}\cdot
     n|m-j|^{-1})+O(\ux n|m-j|^{-2})\right)E_j
     +O(\ux n^{1/2}\|E\|_\infty)\\
  &=O(\ux n^{1/2}\|E\|_\infty),
\end{align*}
where we have used \eqref{cotfar} and a variant of \eqref{cotdifffar} in the second
equality.

Now, we focus on the first term on the right in \eqref{II2term1},
\[n\sum_{j\in\Smc} \cot\frac{y_m-x_j}{2} \big(u(\xb_{j-1})E_{j-1}+u(\xb_j)E_j\big).\]
This is a major term. We need to work on the difference with the
corresponding term in $II_{m+1}^{j+1}:$
\begin{align}
  &n\sum_{j\in\Smc} \left[\cot\frac{y_m-x_j}{2}
  \big(u(\xb_{j-1})E_{j-1}+u(\xb_j)E_j\big)
  -\cot\frac{y_{m+1}-x_{j+1}}{2}
    \big(u(\xb_j)E_j+u(\xb_{j+1})E_{j+1}\big)\right]\nonumber\\
  &=n\sum_{j\in\Smc}\left(\cot\frac{y_m-x_j}{2}-\cot\frac{y_{m+1}-x_{j+1}}{2}\right)
     \big(u(\xb_j)E_j+u(\xb_{j+1})E_{j+1}\big)\nonumber\\
  &\quad+n\sum_{j\in\Smc} \cot\frac{y_m-x_j}{2}
    \big(u(\xb_{j-1})E_{j-1}-u(\xb_{j+1})E_{j+1}\big).\label{IImaindiff}
\end{align}
For the first term on the right hand side of \eqref{IImaindiff}, an estimate using mean value theorem yields
\begin{align}&\cot\frac{y_m-x_j}{2}-\cot\frac{y_{m+1}-x_{j+1}}{2}=
\frac{1}{2\sin^2\frac{z_j}{2}}\cdot(y_{m+1}-y_m-x_{j+1}+x_j)\\\label{3dec1110}
&=\frac{1}{2\sin^2\frac{z_j}{2}}\cdot\left[
(y_{m+1}-y_m-x_{m+1}+x_m)+\frac{1}{2n}\left(\frac{1}{u(\xb_m)}-\frac{1}{u(\xb_j)}\right)+(E_m-E_j)\right]
\end{align}
where $z_j\in(y_m-x_j, y_{m+1}-x_{j+1})$.
This further decomposes the first term on the right in \eqref{IImaindiff} into three parts.
The first part is of the order
\begin{equation}\label{E2err1110}
O\Big(n^{5/2}\|E\|_\infty |y_{m+1}-y_m-x_{m+1}+x_m|\Big).
\end{equation}
Direct estimates of contributions corresponding to the second and third summands in \eqref{3dec1110}
yield a supercritical $O(\ux n \log n \|E\|_\infty)$ and a dangerous quadratic $O(n^{5/2}\|E\|_\infty^2)$ errors
respectively. These estimates are not sufficient and need further treatment.

The contribution corresponding to the second summand in \eqref{3dec1110} can be represented as follows:
\begin{align}
&   \sum_{j\in\Smc}\frac{1}{4\sin^2\frac{z_j}{2}}\left(\frac{1}{u(\xb_m)}-\frac{1}{u(\xb_j)}\right) \Big(u(\xb_j)(E_j-E_m)+u(\xb_{j+1})(E_{j+1}-E_m)\Big)\nonumber\\
&+E_m \sum_{j\in\Smc}\frac{1}{4\sin^2\frac{z_j}{2}}\left(\frac{1}{u(\xb_m)}-\frac{1}{u(\xb_j)}\right) \Big(u(\xb_j)+u(\xb_{j+1})\Big)\label{IImainsecond}
\end{align}
We combine the first sum in \eqref{IImainsecond} and the contribution corresponding to the third summand
 in \eqref{3dec1110} that is given by
\begin{equation}\label{IImainthird}
  \sum_{j\in\Smc}\frac{-n(u(\xb_j)E_j+u(\xb_{j+1})E_{j+1})}{2\sin^2\frac{z_j}{2}}(E_j-E_m),
\end{equation}
and denote the resulting expression $H_{1,m}$. It takes form
\begin{equation}\label{H1}
  H_{1,m}=\sum_{j\in\Smc\cup\{j_-\}}h_1(j,m)(E_j-E_m),
\end{equation}
with the coefficients
\begin{align}
  h_1(j,m)=&\begin{cases}
    \frac{u(\xb_j)-u(\xb_m)}{4u(\xb_m)\sin^2\frac{z_j}{2}}
  +\frac{(u(\xb_{j-1})-u(\xb_m))u(\xb_j)}{4u(\xb_m)u(\xb_{j-1})\sin^2\frac{z_{j-1}}{2}}
  -\frac{n(u(\xb_j)E_j+u(\xb_{j+1})E_{j+1})}{2\sin^2\frac{z_j}{2}},& j\in\Smc\backslash\{j_+\}\\
    \frac{u(\xb_{j_+})-u(\xb_m)}{4u(\xb_m)\sin^2\frac{z_{j_+}}{2}}
  -\frac{n(u(\xb_{j_+})E_{j_+}+u(\xb_{j_++1})E_{j_++1})}{2\sin^2\frac{z_{j_+}}{2}},& j=j_+\\
  \frac{(u(\xb_{j_--1})-u(\xb_m))u(\xb_{j_-})}{4u(\xb_m)u(\xb_{j_--1})\sin^2\frac{z_{j_--1}}{2}},& j=j_-
  \end{cases}\nonumber\\
  =&\,O(n^3 |m-j|^{-2}\|E\|_\infty+\ux n|m-j|^{-1}). \label{h1}
\end{align}
We now consider the second term in \eqref{IImainsecond}.
We argue that this term is harmless.
To see this, we shall make several simplifying replacements. First, replace
$u(\xb_j)+u(\xb_{j+1})$ by $2u(\xb_m)$. The error introduced by such change is
\begin{align}\label{err21110} E_m\sum_{j\in\Smc}O\left(\frac{n^2}{|m-j|^2}\cdot
    \frac{\ux |m-j|}{n}\cdot \frac{\ux |m-j|}{n}\right)=O(\ux n\|E\|_\infty).\end{align}
Next, replace $\frac{1}{u(\xb_m)}-\frac{1}{u(\xb_j)}$ by $\frac{u'(\xb_m)}{u^2(\xb_m)}(\xb_j-\xb_m),$ using that
\[\frac{1}{u(\xb_m)}-\frac{1}{u(\xb_j)}=\frac{u'(\xb_m)}{u^2(\xb_m)}(\xb_j-\xb_m)
  +O(\ux n^{-2}|m-j|^2)\]
The contribution from such change to the error is
\[E_m\sum_{j\in\Smc}O\left(\frac{n^2}{|m-j|^2}\cdot
    \frac{\ux |m-j|^2}{n^2}\right)=O(\ux n\|E\|_\infty).\]
We also replace $\frac{1}{4\sin^2\frac{z_j}{2}}$ by $\frac{1}{(\xb_j-\xb_m)^2}$.
Note that $z_j\in (y_m-x_j, y_{m+1}-x_{j+1})$, and
$|(y_{m+1}-x_{j+1})-(y_m-x_j)| = O(n^{-1})$. Then,
$|z_j-(y_m-x_j)| = O(n^{-1})$.
Since $|(y_m-x_j)-(\xb_j-\xb_m)|=O(n^{-1})$, we get
$|z_j-(\xb_j-\xb_m)|=O(n^{-1})$ as well. Therefore, we have
\begin{align*}
  \frac{1}{4\sin^2\frac{z_j}{2}}
=&\,\frac{1}{4\sin^2\left(\frac{\xb_j-\xb_m}{2}+O(n^{-1})\right)}
=\frac{1}{4\sin^2\frac{\xb_j-\xb_m}{2}}+O(n^2|m-j|^{-3})\\
=&\,\frac{1}{(\xb_j-\xb_m)^2}+O(1)+O(n^2|m-j|^{-3}).
\end{align*}
Thus difference created by such substitution leads to an error
\[E_m\sum_{j\in\Smc}O\left(\left(1+\frac{n^2}{|m-j|^3}\right)\cdot
    \frac{\ux |m-j|}{n}\right)=O(\ux n\|E\|_\infty).\]
Hence, we brought the second term in \eqref{IImainsecond} to the form
\[\frac{u'(\xb_m)E_m}{u(\xb_m)}\sum_{j\in\Smc}\frac{1}{\xb_j-\xb_m}+
O(\ux n\|E\|_\infty).\]
For the remaining sum, we can pair the $j= m\pm k$ summands and make
use of the cancellation similar to \eqref{xsym}
\begin{equation}\label{Empair}
  \frac{1}{\xb_{m+k}-\xb_m}+\frac{1}{\xb_{m-k}-\xb_m}
    =\frac{(\xb_{m+k}-\xb_m)+(\xb_{m-k}-\xb_m)}{(\xb_{m+k}-\xb_m )(\xb_{m-k}-\xb_m)}
    =O\left(\frac{k^2n^{-2}+k\|E\|_\infty}{k^2n^{-2}}\right),
\end{equation}
which after summation in $k$ yields $O(n+n^2 \log n \|E\|_{\infty}).$
There is at most $O(1+n^{3/2}\|E\|_{\infty})$ mismatched terms at the boundary, leading to
$O(n^{1/2}+n^2 \|E\|_{\infty})$ contribution to the sum. Finally, we arrive at a desired
bound on the second term in \eqref{IImainsecond}:
\begin{equation}\label{Emcancel}
  O\left(\ux\|E\|_\infty\cdot (n+n^2\log n\|E\|+n^{1/2})\right)+O(\ux
  n\|E\|_\infty)
  =O(\ux n\|E\|_\infty+\ux n^2\log n\|E\|_\infty^2).
\end{equation}
To sum up, the bound we obtained for the first term
 in \eqref{IImaindiff} is given by
\begin{align}\label{1term1110}
H_{1,m} + O(n^{5/2}\|E\|_\infty |y_{m+1}-y_m-x_{m+1}+x_m|+\ux n\|E\|_\infty+\ux n^2\log n\|E\|_\infty^2).
\end{align}

Now, let us turn to the second term in \eqref{IImaindiff}.
Replace $E_{j\pm1}$ in this term by $E_{j\pm1}-E_m$. The difference leads to an
error
\begin{equation}\label{errshift}
  nE_m\sum_{j\in\Smc}\cot\frac{y_m-x_j}{2}\big(u(\xb_{j-1})-u(\xb_{j+1})\big).
\end{equation}
Again, we shall argue that the error \eqref{errshift} is harmless. Compute
\begin{align*}
  u(\xb_{j-1})-u(\xb_{j+1})=&\,
 -u'(\xb_j)(\xb_{j+1}-\xb_{j-1})+O(\ux n^{-2})\\
  =&\,-u'(\xb_m)(\xb_{j+1}-\xb_{j-1})+O(\ux n^{-2}|m-j| )\\
  =&\,-u'(\xb_m)\left(\frac{1}{nu(\xb_m)}+O(\ux n^{-2}|m-j|+\|E\|_\infty)\right)+O(\ux n^{-2} |m-j|)\\
  =&\, -\frac{u'(\xb_m)}{nu(\xb_m)}+O(\ux n^{-2}|m-j|+\ux \|E\|_\infty).
\end{align*}
Then the error term \eqref{errshift} becomes
\begin{align}\label{err31110}  -\frac{u'(\xb_m)
    E_m}{u(\xb_m)}\sum_{j\in\Smc}\cot\frac{y_m-x_j}{2}+
O\big(\ux n\|E\|_\infty+\ux n^2\log n\|E\|_\infty^2\big).\end{align}
Using \eqref{cotexp1110} and then pairing up the $j= m \pm k$ leading terms similarly
to \eqref{Empair}, estimating the rest of the series and any mismatched terms, we have
\begin{equation}\label{cotsum}
  \sum_{j\in\Smc}\cot\frac{y_m-x_j}{2}=O(n+n^2\log n\|E\|_\infty).
\end{equation}
It follows that the error term \eqref{errshift} has the same order as in \eqref{Emcancel}.

The remainder of the second term in \eqref{IImaindiff} can be telescoped
\begin{align}\label{D2bddy}
  &\,n\sum_{j\in\Smc} \cot\frac{y_m-x_j}{2}
    \big(u(\xb_{j-1})(E_{j-1}-E_m)-u(\xb_{j+1})(E_{j+1}-E_m)\big)\\
=&\,\sum_{j\in\Smc\backslash\{j_+, j_--1\}}nu(\xb_j)\left(\cot\frac{y_m-x_{j+1}}{2}-\cot\frac{y_m-x_{j-1}}{2}\right)(E_j-E_m) \nonumber\\
&\,+nu(\xb_{j_+-1})\cot\frac{y_m-x_{j_+}}{2}(E_{j_+-1}-E_m)+nu(\xb_{j_+})\cot\frac{y_m-x_{j_++1}}{2}(E_{j_+}-E_m) \nonumber\\
&\,-nu(\xb_{j_-})\cot\frac{y_m-x_{j_--1}}{2}(E_{j_-}-E_m)-nu(\xb_{j_--1})\cot\frac{y_m-x_{j_--2}}{2}(E_{j_--1}-E_m). \nonumber
\end{align}
Denote the first term on the right hand side of \eqref{D2bddy} as
\begin{equation}\label{D2}
  D_{2,m}=\sum_{j\in\Smc\backslash\{j_+, j_--1\}}d_2(j,m)(E_j-E_m),
\end{equation}
with positive coefficients
\begin{align}
  d_2(j,m)=&\,nu(\xb_j)\left(\cot\frac{y_m-x_{j+1}}{2}-\cot\frac{y_m-x_{j-1}}{2}\right)=\frac{nu(\xb_j)}{2\sin^2\frac{\tilde{z}_j}{2}}(x_{j+1}-x_{j-1})\nonumber\\
  =&\,\frac{1}{2\sin^2 \frac{y_m-x_j}{2}}\left(1+O(n^{-1/2}+n\|E\|_\infty)\right)\sim
\frac{n^2}{|m-j|^2}, \label{d2}
\end{align}
where $\tilde{z}_j \in (y_m-x_{j+1},y_m-x_{j-1})$ and we performed several elementary simplifications in the third step.
The expression $D_{2,m}$ is our main dissipative term in the far field.
The remaining four boundary terms in \eqref{D2bddy} are kept in
\eqref{II2mjdiff}, and will be handled later. This completes the treatment of the first term in
\eqref{II2mjr}, and summing up all the bounds we see that it has been represented as the sum of the last two lines in
\eqref{D2bddy} and
\begin{align}\label{err41110}
D_{2,m}+H_{1,m}+O(n^{5/2}\|E\|_\infty |y_{m+1}-y_m-x_{m+1}+x_m|+\ux n\|E\|_\infty+\ux n^2\log n\|E\|_\infty^2).
\end{align}

Next, we consider the second term in \eqref{II2mjr},
\begin{align}\label{aux11130} 2n\int_{\xb_{j-1}}^{\xb_j}u(y)\left(\cot\frac{y_m-y}{2}-\cot\frac{y_m-x_j}{2}\right)dy.\end{align}
First, we can replace $u(y)$ by $u(x_j)$, or any $u(z)$ where $z$ is
$O(n^{-1})$ distance from $y$, for instance, $u(\xb_{j-1})$ or
$u(\xb_j)$. The resulting difference sums up to
\begin{equation}\label{changeu}
  \sum_{|j-m|\gtrsim n^{1/2}}O(n\cdot n^{-1}\cdot \ux n^{-1}\cdot
  n|m-j|^{-2})=O(\ux n^{-1/2}).
\end{equation}
Then, we expand $\cot\frac{y_m-y}{2}$ into Taylor series around $\frac{y_m-x_j}{2}$:
\begin{align}\label{taylor1130}
\cot\frac{y_m-y}{2} = \sum_{k=0}^\infty \frac{(-1)^k \cot^{(k)}\frac{y_m-x_j}{2}}{2^k k!}(y-x_j)^k.
\end{align}
We have
\begin{align}
\int_{\xb_{j-1}}^{\xb_j} (y-x_j)^k \, dy & = \frac{1}{k+1}\left(  \left(\frac{x_{j+1}-x_{j}}{2} \right)^{k+1} -\left(\frac{x_{j-1}-x_{j}}{2} \right)^{k+1} \right) \\
& = \frac{1}{(k+1)2^{k+1}} \left( (x_{j+1}-x_j)^{k+1}+(-1)^k (x_j-x_{j-1})^{k+1}    \right). \label{inttaylor1130}
\end{align}
When $k$ is odd, the expression in brackets appearing in \eqref{inttaylor1130} for all sufficiently large $n$ can be estimated by
\begin{align}
(k+1) (x_{j+1}+x_{j-1}-2x_j) O(n^{-k}) & = (k+1) \left( \frac{1}{2n u(\xb_{j})} - \frac{1}{2n u(\xb_{j-1})}+E_{j}-E_{j-1} \right) O(n^{-k})\\
& = (k+1) O(\ux n^{-k-2} +\|E\|_\infty n^{-k}). \label{oddk1130}
\end{align}
As usual, the constants involved in $O$ may only depend on $u_0,$ but not on $n$ or $k.$
Using the Laurent series for $\cot$ \eqref{cotexp1110} and \eqref{ymxj}, it is not difficult to show that
\begin{align}\label{dercotest1130}
\left| \cot^{(k)}\frac{y_m-x_j}{2} \right| \leq \frac{C^k k! n^{k+1}}{|m-j|^{k+1}}
\end{align}
with some constant $C>0$ that may only depend on $u_0.$ 
Taking into account the factor $n$ in front of the integral in \eqref{aux11130},
we obtain that a contribution to \eqref{aux11130} from any odd $k$ in \eqref{taylor1130}
can be estimated as
\begin{align}\label{aux21130}
O\left( \frac{(k+1) C^k n^{k+2}}{|m-j|^{k+1}} \cdot (\ux n^{-k-2} +\|E\|_\infty n^{-k}) \right)= O ( C^k \ux  |m-j|^{-k-1} + C^k n^2 |m-j|^{-k-1} \|E\|_\infty ),
\end{align}
where we can absorb $k+1$ into $C^k$ by slightly adjusting the constant.
Summing up over $j \in \Smc$ we get the bound of the order
\begin{align}\label{aux31130}
O(C^k \ux n^{-\frac{k}{2}} + C^k  n^{-\frac{k}{2}+2} \|E\|_\infty).
\end{align}
For all $n$ sufficiently large, the first expression in \eqref{aux31130} can be summed up over $k \geq 1$, yielding $O(\ux n^{-1/2}).$ The second expression in \eqref{aux31130} can be summed up over $k \geq 3$
leading to $O(n^{1/2}\|E\|_\infty)$ contribution. The case $k=1$ for the second error will now be considered in more detail.

Going back to \eqref{taylor1130}, \eqref{inttaylor1130} and \eqref{oddk1130}, we see that the part of the linear term $k=1$ leading to the large error is given by
\begin{align}
\frac{n u(x_j)}{8\sin^2 \frac{y_m-x_j}{2}} (E_j-E_{j-1}) (x_{j+1}-x_{j-1}) & = \frac{1}{8\sin^2 \frac{y_m-x_j}{2}} \left( (E_j-E_{j-1})+ n u(x_j)(E_j^2-E_{j-1}^2) \right)\\
& +O(n |j-m|^{-2} \|E\|_\infty).\label{aux41130}
\end{align}
For the second summand in the brackets on the right side of \eqref{aux41130}, we telescope when summing over $j$ and get
\begin{align*}
  &\sum_{j\in\Smc}\frac{nu(x_j)}{8\sin^2\frac{y_m-x_j}{2}}(E_j^2-E_{j-1}^2)
  =\sum_{j\in\Smc\backslash\{j_+\}}
  \left(\frac{nu(x_j)}{8\sin^2\frac{y_m-x_j}{2}}
  -\frac{nu(x_{j+1})}{8\sin^2\frac{y_m-x_{j+1}}{2}}\right)E_j^2+O(n^2\|E\|_\infty^2)\\
  &=\sum_{j\in\Smc\backslash\{j_+\}}O\left(\frac{n^3}{|m-j|^2}\cdot\ux
    n^{-1}+\frac{n^4}{|m-j|^3}\cdot n^{-1} \right)E_j^2+O(n^2\|E\|_\infty^2)
    =O(n^2\|E\|_\infty^2).
\end{align*}

Summing up the rest of \eqref{aux41130} in $j,$ and adding the earlier bounds, we obtain that summation over all odd $k$ in \eqref{taylor1130} results in the
following contribution
to \eqref{aux11130}:
\begin{align}\label{err51110}
\sum_{j\in\Smc}\frac{E_j-E_{j-1}}{8\sin^2\frac{y_m-x_j}{2}}+O(\ux  n^{-1/2}+n^{1/2}\|E\|_\infty+n^2\|E\|_\infty^2).\end{align}
Telescoping in the sum that remains in \eqref{err51110} would yield a critical error of the order
$O(n\|E\|_\infty)$ without $\ux$ factor, which should be avoided if we want to
control the evolution of errors for all times. Let us consider the
difference of the sum in \eqref{err51110} with the corresponding sum in $II_{2,m+1}^{j+1}$:
\begin{equation}\label{H2}
  H_{2,m}:=\sum_{j\in\Smc}\frac{E_j-E_{j-1}}{8\sin^2\frac{y_m-x_j}{2}}
  -\sum_{j\in\Smcp}\frac{E_j-E_{j-1}}{8\sin^2\frac{y_{m+1}-x_j}{2}}
    =\sum_{j\in\Smc\cup\{j_+-1, j_-\}}h_2(j,m) (E_j-E_m).
\end{equation}
Here we write $E_j-E_{j-1}=(E_j-E_m)-(E_{j-1}-E_m)$ and reorganize
the sum. The coefficients $h_2(j,m)$ are given by
\begin{align}\label{h2}
  &h_2(j,m)=\\ &\begin{cases}
    \frac{1}{8\sin^2\frac{y_m-x_j}{2}}-\frac{1}{8\sin^2\frac{y_m-x_{j+1}}{2}}
    -\frac{1}{8\sin^2\frac{y_{m+1}-x_j}{2}}+\frac{1}{8\sin^2\frac{y_{m+1}-x_{j+1}}{2}},&
    j\in\Smc\backslash\{j_--1, j_+\}\\
    \frac{1}{8\sin^2\frac{y_m-x_{j_--1}}{2}}
    -\frac{1}{8\sin^2\frac{y_{m+1}-x_{j_--1}}{2}}+\frac{1}{8\sin^2\frac{y_{m+1}-x_{j_-}}{2}},&
    j=j_--1\\
    -\frac{1}{8\sin^2\frac{y_{m+1}-x_{j_-}}{2}},& j=j_-\\
    -\frac{1}{8\sin^2\frac{y_m-x_{j_+}}{2}},& j=j_+-1\\
    \frac{1}{8\sin^2\frac{y_m-x_{j_+}}{2}}-\frac{1}{8\sin^2\frac{y_m-x_{j_++1}}{2}}
   +\frac{1}{8\sin^2\frac{y_{m+1}-x_{j_++1}}{2}},&
    j=j_+
  \end{cases}\nonumber
\end{align}
In the bulk when $ j\in\Smc\backslash\{j_--1, j_+\}$, we have
\[\frac{1}{8\sin^2\frac{y_m-x_j}{2}}-\frac{1}{8\sin^2\frac{y_m-x_{j+1}}{2}}=O(n^3|m-j|^{-3}\cdot n^{-1})=O(n^2|m-j|^{-3}),\]
and similarly for the summands containing $y_{m+1}.$
Therefore, the coefficients satisfy
\begin{align}\label{h2out1116} h_2(j,m)=O(n^2|m-j|^{-3}),\quad \text{ for }  j\in\Smc\backslash\{j_--1, j_+\}.\end{align}
At the boundary, there is no such cancellation.
Since $x_{j_+}-y_m=n^{-1/2}+O(n^{-1})$, we have
\begin{equation}\label{sinest}
  \sin^2\frac{x_{j_+}-y_m}{2}=\frac{1}{4n} + O(n^{-3/2}).
\end{equation}
Similar estimates hold for all terms in \eqref{h2} at the boundary.
Therefore, we have
\begin{equation}\label{h2boundary}
  h_2(j,m)=\begin{cases}
    -\frac{n}{2}+O(n^{1/2}),& j=j_-, j_+-1\\
    \frac{n}{2}+O(n^{1/2}),& j=j_--1, j_+
  \end{cases}
\end{equation}

We now consider the case of even $k$ in \eqref{taylor1130}.
In this case, we will be matching the $m$th and $(m+1)$st contributions. From \eqref{taylor1130} and \eqref{inttaylor1130},
we see that the key expression to estimate is
\begin{align}\label{aux51130}
\cot^{(k)} \frac{y_m-x_j}{2}  \left( (x_{j+1}-x_j)^{k+1}+ (x_j-x_{j-1})^{k+1}\right) - \\
\cot^{(k)} \frac{y_{m+1}-x_{j+1}}{2}  \left( (x_{j+2}-x_{j+1})^{k+1}+ (x_{j+1}-x_{j})^{k+1}\right).
\end{align}
Let us telescope, and first observe that
\begin{align}\label{aux61130}
\left| \cot^{(k)} \frac{y_m-x_j}{2} - \cot^{(k)} \frac{y_{m+1}-x_{j+1}}{2} \right| \leq \frac12 |\cot^{(k+1)} z_j|
\cdot |y_m - x_j - y_{m+1} +x_{j+1}|,
\end{align}
where $z_j \in (\frac{y_m-x_j}{2}, \frac{y_{m+1}-x_{j+1}}{2}).$ But
\[ y_{m+1} - x_{j+1} - y_{m} +x_{j} = y_{m+1}-x_{m+1}-y_m +x_m +\frac{1}{2n u(\xb_m)}-\frac{1}{2n u(\xb_j)} +E_m -E_j, \]
and taking into account \eqref{dercotest1130} and \eqref{ymxj}, we obtain that the contribution to \eqref{aux11130} coming from the
difference of cotangents in \eqref{aux51130} has the order of
\begin{align}\label{aux71130}
O \left( n \cdot \frac{C^{k+1} (k+1)! n^{k+2}}{(k+1)! |m-j|^{k+2}} \cdot n^{-k-1} \cdot \left(|y_{m+1}-x_{m+1}-y_m +x_m| + \ux |m-j| n^{-2} + \|E\|_\infty \right)\right).
\end{align}
Simplifying this expression and summing over $j \in \Smc,$ we obtain
\begin{align}
\sum_{j \in \Smc} O \left( C^{k+1} n^2 |m-j|^{-k-2} (|y_{m+1}-x_{m+1}-y_m +x_m| + \ux |m-j| n^{-2} + \|E\|_\infty)\right) \\ = O\left(C^{k+1}\left(n^{-\frac{k-3}{2}} |y_{m+1}-x_{m+1}-y_m +x_m| +
\ux n^{-\frac{k}{2}} + n^{-\frac{k-3}{2}} \|E\|_\infty \right)\right).  \label{aux81130}
\end{align}
Given that the zeroth term in \eqref{taylor1130} is cancelled identically, we sum \eqref{aux81130} over even $k \geq 2,$ obtaining the contribution of the order
\begin{align}\label{cotdiff1130}
 O( n^{1/2} |y_{m+1}-x_{m+1}-y_m +x_m| + \ux n^{-1} + n^{1/2} \|E\|_\infty).
\end{align}
Next, let us estimate the second telescoped difference
\begin{equation}\label{aux91130}
\cot^{(k)} \frac{y_m-x_j}{2} \cdot \left( (x_{j+1}-x_j)^{k+1}+ (x_j-x_{j-1})^{k+1} -(x_{j+2}-x_{j+1})^{k+1}- (x_{j+1}-x_{j})^{k+1}\right).
\end{equation}
Note that the factor in brackets in \eqref{aux91130} equals
\begin{align}
& (x_j-x_{j-1})^{k+1} -(x_{j+2}-x_{j+1})^{k+1}= (x_j-x_{j-1}-x_{j+2}+x_{j+1}) \cdot  O(k n^{-k})  \nonumber\\
&= O \left(k n^{-k} \cdot \left( \frac{1}{2n u(\xb_{j-1})} - \frac{1}{2n u(\xb_{j+1})} +E_{j-1}-E_{j+1} \right)\right)\nonumber\\
&= O(k n^{-k} (\ux n^{-2} + \|E\|_\infty )). \label{aux101130}
\end{align}
Combining \eqref{aux101130} with \eqref{dercotest1130} and \eqref{taylor1130}, we find that the contribution of the difference of $m$th and $(m+1)$st
$k$th summand from \eqref{taylor1130} for even $k$ has the order
\[  O\left( n\frac{ C^k k! n^{k+1}}{k! |m-j|^{k+1}} \cdot k n^{-k} (\ux n^{-2} + \|E\|_\infty) \right) = O(C^k |m-j|^{-k-1}(\ux + n^2 \|E\|_\infty)). \]
Summing up over $j \in \Smc,$ we obtain
\begin{align}\label{aux121130}
\sum_{j \in \Smc} O\left(C^k |m-j|^{-k-1}(\ux + n^2 \|E\|_\infty)\right) = O\left(C^k (\ux n^{-\frac{k}{2}} + n^{-\frac{k}{2}+2}\|E\|_\infty)\right).
\end{align}
We can sum the first expression on the right hand side of \eqref{aux121130} over even $k \geq 2,$ and the second one over $k \geq 4$, yielding the
contribution of the order $O( \ux n^{-1} + \|E\|_\infty).$
For the $k=2$ case of the second expression in \eqref{aux121130}, we get the contribution of the order $O(n \|E\|_\infty)$. This is a critical error without $\ux$ factor and so we will consider it
in more detail. This error arises from the difference $E_{j-1}-E_{j+1}$ in \eqref{aux101130}, and equals 
\begin{align}\label{aux141130}
 \sum_{j \in \Smc} \frac{nu(x_j)\cos\frac{y_m-x_j}{2}}{\sin^3\frac{y_m-x_j}{2}} \cdot \left((E_{j-1}-E_m)-(E_{j+1}-E_m)\right) \cdot O(n^{-2}), 
\end{align}
where $O(n^{-2})$ contains a constant factor and a factor coming from \eqref{aux101130}.
We represent \eqref{aux141130} as
\begin{align}
H_{3,m}:= \,\sum_{j\in\Smc\cup\{j_-, j_+-1\}} h_3(j,m)(E_j-E_m), \label{H3}
\end{align}
where the coefficients $h_3(j,m)$ can be estimated by
\begin{equation}\label{h3}
  h_3(j,m)=O\left(\frac{n^4}{|m-j|^3}\right)\cdot
  O(n^{-2})=O\left(\frac{n^2}{|m-j|^3}\right).
\end{equation}
A better estimate is possible in the bulk, taking into account the
cancellation - but we do not need it.

Collecting all the estimates, we obtain \eqref{II2mjdiff}.
\end{proof}

Next, we work on the mismatched part in \eqref{II2m}
\begin{equation}\label{IImis}
  II_{\text{Mis}, m}:=2n\int_{\Dmis_m} u(y)\cot\frac{y_m-y}{2}\,dy=
 2n \int_{\Dmis_m^+}u(y)\cot\frac{y_m-y}{2}\,dy
  +2n \int_{\Dmis_m^-}u(y)\cot\frac{y_m-y}{2}\,dy.
\end{equation}
Recall that
\begin{align}\label{Dmis1110} \Dmis_m =\Dmis_{m}^-\cup \Dmis_{m}^+,\quad\text{where}\quad \Dmis_{m}^-=[
 \xb_{j_--1}, y_m-n^{-1/2}], \,\, \Dmis_{m}^+=[y_m+n^{-1/2}, \xb_{j_+-1}].\end{align}
  The endpoints in the definition of the intervals $\Omega_m^\pm$ record the accurate
 limits of integration in \eqref{IImis} (which influences the sign of the contributions).
 The correct order of the limits of integration can be inferred from \eqref{IIm}, \eqref{II2mj} and \eqref{II2m}.

\begin{lemma}\label{lem:mismatch}
  Let $II_{\text{Mis}, m}$ be as defined in \eqref{IImis}. Then,
  \begin{align}
    &II_{\text{Mis}, m}-II_{\text{Mis}, m+1}=
    nu(\xb_{j_+})\cot\frac{z_+}{2}\big((E_{j_+-1}-E_m)+(E_{j_+}-E_m)\big)\nonumber\\
    &\qquad+nu(\xb_{j_-})\cot\frac{z_-}{2}\big((E_{j_--1}-E_m)+(E_{j_-}-E_m)\big)\nonumber\\
    &\qquad+O\big(n^{3/2}(y_{m+1}-y_m-x_{m+1}+x_m)+\ux n^{-1/2}+\ux
    n\|E\|_\infty\big), \label{IImismatch}
  \end{align}
  where $z_+, z_-$ are such that $z_\pm=n^{-1/2}+O(n^{-1})$.
\end{lemma}

\begin{proof}
We focus our discussion on the first integral (over $\Dmis^+_m$) in
\eqref{IImis}. The second integral can be treated using the analogous argument.
Decompose the integral as follows:
\begin{equation}\label{misdecomp}
  2nu(\xb_{j_+-1})\int_{\Dmis_m^+}\cot\frac{y_m-y}{2}\,dy
+2n\int_{\Dmis_m^+}\left(u(y)-u(\xb_{j_+-1})\right)
\cot\frac{y_m-y}{2}\,dy.
\end{equation}
The second term in \eqref{misdecomp} can be estimated by
\[O\big(n\cdot n^{-1}\cdot \ux n^{-1}\cdot n^{1/2}\big)=O(\ux n^{-1/2}).\]
Here, we have used the fact that $|\Dmis_m^+|=O(n^{-1})$, and
$y_m-y=n^{-1/2}+O(n^{-1})$ for any $y\in\Dmis_m^+$.
Let us offset the first term in \eqref{misdecomp} by its counterpart from $II_{2,m+1}:$
\begin{align}\nonumber
& 2nu(\xb_{j_+-1})\int_{y_{m}+n^{-1/2}}^{\xb_{j_+-1}}\cot\frac{y_m-y}{2}\,dy-2nu(\xb_{j_+})\int_{y_{m+1}+n^{-1/2}}^{\xb_{j_+}}\cot\frac{y_{m+1}-y}{2}\,dy\\
& =-2nu(\xb_{j_+-1})\int_{n^{-1/2}}^{\xb_{j_+-1}-y_m}\cot\frac{z}{2}\,dz+2nu(\xb_{j_+})\int_{n^{-1/2}}^{\xb_{j_+}-y_{m+1}}\cot\frac{z}{2}\,dz. \label{err71110}
\end{align}
Let us replace $u(\xb_{j_+-1})$ in the first summand in \eqref{err71110} with $u(\xb_{j_+}).$ This generates an error of the order
\[ O\big( \ux n^{-1}\cdot n\cdot n^{-1}\cdot n^{1/2}\big) = O(\ux n^{-1/2}).  \]
Then the difference in \eqref{err71110} can be written as
\begin{align}\label{err81110} 2nu(\xb_{j_+}) \int_{\xb_{j_+-1}-y_m}^{\xb_{j_+}-y_{m+1}} \cot\frac{z}{2}\,dz = 2nu(\xb_{j_+})\cot \frac{z_+}{2} \cdot (\xb_{j_+}-y_{m+1}-\xb_{j_+-1}+y_m), \end{align}
where $z_+\in(\xb_{j_+-1}-y_m, \xb_{j_+}-y_{m+1})$, and clearly
$z_+=n^{-1/2}+O(n^{-1})$.
Further decomposition yields
\begin{align}
  \xb_{j_+}-y_{m+1}-\xb_{j_+-1}+y_m=&\,
    -(y_{m+1}-y_m-x_{m+1}+x_m)+\left(\frac{E_{j_+-1}+E_{j_+}}{2}-E_m\right)\nonumber\\
 &+\frac{1}{2n}\left(\frac{1}{2u(\xb_{j_+-1})}+\frac{1}{2u(\xb_{j_+})}-\frac{1}{u(\xb_m)}\right).\label{mismatchlength}
\end{align}
The first term in \eqref{mismatchlength} leads to an error
\begin{align}\label{err11114} O\big(n^{3/2}|y_{m+1}-y_m-x_{m+1}+x_m|\big).\end{align}
The second term  in \eqref{mismatchlength} produces dissipation at the boundary
\begin{align}\label{bnddiss1114} nu(\xb_{j_+})\cot\frac{z_+}{2}(E_{j_+-1}-E_m)
  +nu(\xb_{j_+})\cot\frac{z_+}{2}(E_{j_+}-E_m).\end{align}
We keep this term in \eqref{IImismatch}.
The last term in \eqref{mismatchlength} leads to a contribution
\begin{align}\label{err21114} u(\xb_{j_+})\cot\frac{z_+}{2}\cdot\left(\frac{1}{2u(\xb_{j_+-1})}+\frac{1}{2u(\xb_{j_+})}-\frac{1}{u(\xb_m)}\right).\end{align}
Let us make the following simplifications here. Replace $z_+$ by
$n^{-1/2}$. As $z_+-n^{-1/2}=O(n^{-1})$, the error generated by such substitution is
\[O\big(n\cdot n^{-1}\cdot \ux n^{-1/2}\big)=O(\ux n^{-1/2}).\]
Replace $u(\xb_{j_+-1})$ by $u(\xb_{j_+})$, the error is
\[O\big( n^{1/2} \ux n^{-1}\big)=O(\ux n^{-1/2}).\]
Then \eqref{err21114} becomes
\begin{align}\label{err31114} \cot\frac{n^{-1/2}}{2}\cdot \frac{u(\xb_m)-u(\xb_{j_+})}{u(\xb_m)}.\end{align}
Such contribution produces a supercritical error of order $O(\ux)$.
We will group \eqref{err31114} with the analogous term coming from
the integral in $\Dmis^-_m,$ leading to
\begin{align}\label{err41114} \cot\frac{n^{-1/2}}{2}\cdot \frac{u(\xb_m)-u(\xb_{j_+})}{u(\xb_m)}
  +\cot\frac{n^{-1/2}}{2}\cdot \frac{u(\xb_m)-u(\xb_{j_-})}{u(\xb_m)}.
\end{align}
To explain the contribution from integration over $\Dmis^-_m$ to \eqref{err41114}, let us quickly trace through
the corresponding computations. The analogs of the estimates in \eqref{misdecomp}, \eqref{err71110}, \eqref{err81110}
lead to the principal contribution
\begin{align}\label{omneg1114}
2nu(\xb_{j_-}) \int_{y_{m+1}-\xb_{j_-}}^{y_{m}-\xb_{j_--1}} \cot\frac{z}{2}\,dz = 2nu(\xb_{j_-})\cot \frac{z_-}{2} \cdot (y_{m}-\xb_{j_--1}+\xb_{j_-}-y_{m+1})
\end{align}
where $z_-\in(y_{m+1}-\xb_{j_-}, y_{m}-\xb_{j_--1})$. In the analog of \eqref{mismatchlength},
the second term leads to the dissipative contribution
\[  nu(\xb_{j_-})\cot\frac{z_-}{2}(E_{j_--1}-E_m)
  +nu(\xb_{j_-})\cot\frac{z_-}{2}(E_{j_-}-E_m), \]
that appears in \eqref{IImismatch}.  The analog of the last term in \eqref{mismatchlength} leads to a contribution
\begin{align}\label{aux1213}  u(\xb_{j_-})\cot\frac{z_-}{2}\cdot\left(\frac{1}{2u(\xb_{j_--1})}+\frac{1}{2u(\xb_{j_-})}-\frac{1}{u(\xb_m)}\right) \end{align}
that after simplifications yields the second term in \eqref{err41114}.

Coming back to \eqref{err41114}, we aim to exploit additional cancellations in the numerator:
\begin{equation}\label{err51114}
  \big(u(\xb_m)-u(\xb_{j_+})\big)+\big(u(\xb_m)-u(\xb_{j_-})\big)
  =\partial_x u(\xb_m)\big((\xb_m-\xb_{j_+})+(\xb_m-\xb_{j_-})\big)+O(\ux n^{-1}).
\end{equation}
The higher order expression $O(\ux n^{-1})$ in \eqref{err51114} leads to an error $O(\ux n^{-1/2})$. For the
remaining first summand on the right side of \eqref{err51114}, write
\begin{equation}\label{err71114}
  \xb_m-\xb_{j_+}=-\frac{1}{4nu(\xb_m)}-\sum_{l=1}^{j_+-m-1}\frac{1}{2nu(\xb_{m+l})}-\frac{1}{4nu(\xb_{j_+})}-\frac{E_m}{2}-\sum_{l=1}^{j_+-m-1}E_{m+l}-\frac{E_{j_+}}{2},
\end{equation}
\begin{equation}\label{err81114}
\xb_m-\xb_{j_-}=\frac{1}{4nu(\xb_m)}+\sum_{l=1}^{m-j_--1}\frac{1}{2nu(\xb_{m-l})}+\frac{1}{4nu(\xb_{j_-})}+\frac{E_m}{2}+\sum_{l=1}^{m-j_--1}E_{m-l}+\frac{E_{j_-}}{2}.
\end{equation}
We can estimate the sum
\begin{align}\label{err61114} \sum_{l=1}^{\sim n^{1/2}}
  \left(-\frac{1}{2nu(\xb_{m+l})}+\frac{1}{2nu(\xb_{m-l})}\right)
=\sum_{l=1}^{\sim n^{1/2}}O\big(n^{-1}\cdot\ux ln^{-1}\big)=O(\ux
n^{-1}).\end{align}
Compared with \eqref{err61114}, in the corresponding sums in \eqref{err71114}, \eqref{err81114} there could be at most
$O(1+n^{3/2}\|E\|_{\infty}))$ mismatched summands. Each of these mismatched summands is of the order $O(n^{-1})$, leading to the error $O(\ux n^{-1/2}+\ux n \|E\|_\infty)$ due to the $\partial_x u$ factor in \eqref{err51114}
and $\cot$ factor in \eqref{err41114}.
The rest of the summands in \eqref{err71114}, \eqref{err81114} involving $\{E_j\}$ add up to
$O(n^{1/2}\|E\|_\infty)$, leading to $O(\ux n\|E\|_\infty)$ error. Putting everything together, we obtain
\[\cot\frac{n^{-1/2}}{2}\cdot \frac{u(\xb_m)-u(\xb_{j_+})}{u(\xb_m)}
  +\cot\frac{n^{-1/2}}{2}\cdot \frac{u(\xb_m)-u(\xb_{j_-})}{u(\xb_m)}
  =O\big(\ux n^{-1/2}+\ux n\|E\|_\infty\big).
\]
Collecting all bounds, we obtain \eqref{IImismatch}.
\end{proof}

Combining \eqref{II2mjdiff} and \eqref{IImismatch}, we conclude that
\begin{align}
  &II_{2,m}-II_{2,m+1}=
  D_{2,m}+D_{3,m}+H_{1,m}+H_{2,m}+H_{3,m}\nonumber\\
  &+O\big((n^{5/2}\|E\|_\infty+n^{3/2})|y_{m+1}-y_m-x_{m+1}+x_m|\big)\nonumber\\
  &+O(\ux n^{-1/2}+\ux n\|E\|_\infty+\ux n^2\log n\|E\|_\infty^2)
    +O(n^{1/2}\|E\|_\infty+n^2\|E\|_\infty^2),\label{II2mdiff}
\end{align}
where we collect the boundary terms and define $D_{3,m}$ as
\begin{align}
  D_{3,m}:=&\,  nu(\xb_{j_+})\cot\frac{z_+}{2}\big((E_{j_+-1}-E_m)+(E_{j_+}-E_m)\big)\nonumber\\
    &+nu(\xb_{j_-})\cot\frac{z_-}{2}\big((E_{j_--1}-E_m)+(E_{j_-}-E_m)\big)\nonumber\\
    &-nu(\xb_{j_+})\cot\frac{x_{j_+}-y_m}{2}(E_{j_+-1}-E_m)-nu(\xb_{j_+})\cot\frac{x_{j_++1}-y_m}{2}(E_{j_+}-E_m) \nonumber\\
          &-nu(\xb_{j_-})\cot\frac{y_m-x_{j_--1}}{2}(E_{j_-}-E_m)-nu(\xb_{j_-})\cot\frac{y_m-x_{j_--2}}{2}(E_{j_--1}-E_m)
            \nonumber\\
   =&\,\sum_{j\in\{j_--1,~ j_-,~ j_+-1,~ j_+\}}d_3(j,m)(E_j-E_m).\label{D3}
\end{align}
We adjusted some arguments of $u$ for simplicity; these changes generate a subcritical error of the order $O(\ux n^{1/2} \|E\|_\infty).$

The following lemma shows that the coefficients $d_3(j,m)$ are
positive, and hence $D_{3,m}$ produces dissipation at the interface. This will turn out to be only a part
of dissipation on the interface; we will bring all parts together later in Section \ref{properr}.

\begin{lemma}\label{lem:d3}
  The coefficients $d_3(j,m)$ in \eqref{D3} are positive, and they
  satisfy
  \begin{equation}\label{d3}
    d_3(j,m)=
    \begin{cases}
      \frac{n}{2}+O(n^{1/2}+n^2\log n\|E\|_\infty+n^2(\log n)^2\|E\|_\infty^2),& j=j_-, j_+-1,\\
      \frac{3n}{2}+O(n^{1/2}+n^2\log n\|E\|_\infty+n^2(\log n)^2\|E\|_\infty^2),& j=j_--1, j_+.
    \end{cases}
  \end{equation}
\end{lemma}

\begin{proof}
  Let us first check $d_3(j_+-1,m)$ (a similar argument applies to $d_3(j_-,m)$). From \eqref{D3}, we have
  \begin{align}
    d_3(j_+-1,m)=&\,nu(\xb_{j_+})\cot\frac{z_+}{2}-nu(\xb_{j_+})\cot\frac{x_{j_+}-y_m}{2}\label{d3jm1}\\
  =& 
   = \frac{nu(\xb_{j_+})}{2\sin^2\frac{x_{j_+}-y_m}{2}}(x_{j_+}-y_m-z_+)+O(n^{1/2}).
  \end{align}
  Now, we show $x_{j_+}-y_m-z_+>0$. Recall that
  $z_+\in(\xb_{j_+-1}-y_m, \xb_{j_+}-y_{m+1})$. Let us check the two
  endpoints. Clearly,
  \begin{equation}\label{endpointest}
    (x_{j_+}-y_m)-(\xb_{j_+-1}-y_m)=\frac{1}{4nu(\xb_{j_+-1})}+
    \frac{E_{j_+-1}}{2} =
    \frac{1}{4nu(\xb_{j_+})}+O(\|E\|_\infty+n^{-2}) >0
  \end{equation}
  for all $n$ large enough due to \eqref{Esmall}.
  We adjusted the argument of $u$ from $\xb_{j_+-1}$ to $\xb_{j_+}$ incurring $O(n^{-2})$ error.

    For the other endpoint, we apply the rough estimate \eqref{yyxxrough}
  on $y_{m+1}-y_m-x_{m+1}+x_m$ in \eqref{mismatchlength} and get
\[|\xb_{j_+}-y_{m+1}-\xb_{j_+-1}+y_m| = O(n^{-3/2}+\log
  n\|E\|_\infty+(\log n)^2\|E\|_\infty^2)+O(\|E\|_\infty)+O(\ux n^{-3/2}).\]
Then, we deduce
\begin{align}
  (x_{j_+}-y_m)-(\xb_{j_+}-y_{m+1})=&\,\big((x_{j_+}-y_m)-(\xb_{j_+-1}-y_m)\big)
  -(\xb_{j_+}-y_{m+1}-\xb_{j_+-1}+y_m)\nonumber\\
  =&\,\frac{1}{4nu(\xb_{j_+})}+O(n^{-3/2}+ \log n \|E\|_\infty+(\log n)^2\|E\|_\infty^2).\label{endpoint2est}
\end{align}

Combining \eqref{endpointest} and \eqref{endpoint2est}, we conclude
\begin{equation}\label{jpm1diff}
  x_{j_+}-y_m-z_+=\frac{1}{4nu(\xb_{j_+})}+O(n^{-3/2}+\log n\|E\|_\infty+(\log n)^2\|E\|_\infty^2).
\end{equation}
Moreover, recall \eqref{sinest}:
\[\sin^2\frac{x_{j_+}-y_m}{2}=\frac{1}{4n} + O(n^{-3/2}).\]
Substituting these estimates into \eqref{d3jm1}, we get
\begin{align}\label{aux1214}  d_3(j_+-1,m) & =\frac{n^2u(\xb_{j_+})}{2nu(\xb_{j_+})}+
  O(n^{1/2}+n^2\log n\|E\|_\infty+(\log n)^2\|E\|_\infty^2) \\ & = \frac{n}{2}+O(n^{1/2}+n^2\log n\|E\|_\infty+n^2(\log n)^2\|E\|_\infty^2),\end{align}
exactly as claimed in \eqref{d3}.

Next, we check $d_3(j_+,m)$ (a similar argument applies to $d_3(j_--1,m)$). The analog of
the calculation in \eqref{d3jm1} is
\begin{align*}
  d_3(j_+,m)=&\,nu(\xb_{j_+})\cot\frac{z_+}{2}-nu(\xb_{j_+})\cot\frac{x_{j_++1}-y_m}{2}\nonumber\\
  =&\,\frac{nu(\xb_{j_+})}{2\sin^2\frac{x_{j_++1}-y_m}{2}}(x_{j_++1}-y_m-z_+)
  +O(n^{1/2}).
\end{align*}
From \eqref{jpm1diff}, it follows that
\begin{equation}\label{jpdiff}
  x_{j_++1}-y_m-z_+=\frac{3}{4nu(\xb_{j_+})}+O(n^{-3/2}+\log
  n\|E\|_\infty+(\log n)^2\|E\|_\infty^2).
\end{equation}
Hence,
\begin{align}\label{aux11214} d_3(j_+,m) & =2n^2u(\xb_{j_+})\cdot\frac{3}{4nu(\xb_{j_+})}+
  O(n^{1/2}+n^2\log n\|E\|_\infty+(\log n)^2\|E\|_\infty^2) \\ & = \frac{3n}{2}+ O(n^{1/2}+n^2\log n\|E\|_\infty+n^2(\log n)^2\|E\|_\infty^2).\end{align}
\end{proof}

\section{Propagation of errors equation}\label{properr}

Now we are going to put together our estimates and bring the propagation of errors equation to the
 form that is most convenient for further analysis.
Let us recall the estimates in \eqref{Idiff}, 
\eqref{II1mdiff} and \eqref{II2mdiff}.
These bounds lead to the characterization
\begin{align*}
  G_m-G_{m+1}=&\,
  D_{1,m}+D_{2,m}+D_{3,m}+B_{1,m}+B_{2,m}+H_{1,m}+H_{2,m}+H_{3,m}\\
  +&\,O\big((n^3\|E\|_\infty+n^{3/2})|y_{m+1}-y_m-x_{m+1}+x_m|\big)\\
   +&\,O\big(\ux(n^{-1/2}+n\|E\|_\infty+n^2\log n\|E\|_\infty^2)\big)
    +O(n^{1/2}\|E\|_\infty+n^2\|E\|_\infty^2).
\end{align*}
Substituting this into \eqref{yyxxdiff}, we have
\begin{align*}
  &y_{m+1}-y_m-x_{m+1}+x_m\\
 &=-\frac{\pa_xu(\xb_m)}{2nu(\xb_m)^2}(y_{m+1}-x_{m+1})
    +\frac{1}{4\pi
     n^2(u(\xb_m)^2+Hu(\xb_m)^2)}\pa_x\left(\frac{Hu}{u}\right)(\xb_m)\\
  &\,+\frac{u(\xb_m)Hu(\xb_m)}{2\pi
    n^2 (u(\xb_m)^2+H u(\xb_m)^2)^2}(A_{1,m}+A_{2,m})
    \pa_x\left(\frac{Hu}{u}\right)(\xb_m)\\
  &\,+\frac{\sin^2z}{8\pi^2n^2u(\xb_m)^2}(D_{1,m}+D_{2,m}+D_{3,m}+B_{1,m}+B_{2,m}+H_{1,m}+H_{2,m}+H_{3,m})\\
  &\, + O\big((n\|E\|_\infty+n^{-1/2})|y_{m+1}-y_m-x_{m+1}+x_m|\big)\\
  &\, +O\big(\ux n^{-5/2}+\ux n^{-1}\|E\|_\infty\big)
    +O(n^{-3/2}\|E\|_\infty+ (\log n)^2 \|E\|_\infty^2).
\end{align*}
The term in the penultimate line can be absorbed into the left hand
side. Provided that $n$ is sufficiently large and \eqref{Esmall} holds, we obtain
\begin{align}
  &y_{m+1}-y_m-x_{m+1}+x_m=\big(1+O(n\|E\|_\infty+n^{-1/2})\big)\times\nonumber\\
 &\times\Bigg[-\frac{\pa_xu(\xb_m)}{2nu(\xb_m)^2}(y_{m+1}-x_{m+1})
    +\frac{1}{4\pi
     n^2(u(\xb_m)^2+Hu(\xb_m)^2)}\pa_x\left(\frac{Hu}{u}\right)(\xb_m)\Bigg.\nonumber\\
  &\,+\frac{u(\xb_m)Hu(\xb_m)}{2\pi
    n^2 (u(\xb_m)^2+H u(\xb_m)^2)^2}(A_{1,m}+A_{2,m})
    \pa_x\left(\frac{Hu}{u}\right)(\xb_m)\nonumber\\
  &\,+\frac{\sin^2z}{8\pi^2n^2u(\xb_m)^2}(D_{1,m}+D_{2,m}+D_{3,m}+B_{1,m}+B_{2,m}+H_{1,m}+H_{2,m}+H_{3,m})\nonumber\\
  &\, \Bigg.+O\big(\ux n^{-5/2}+\ux n^{-1}\|E\|_\infty\big)
    +O\left(n^{-3/2}\|E\|_\infty+(\log n)^2 \|E\|_\infty^2\right)\Bigg].\label{yyxx}
\end{align}

Now, we link the estimate above to the propagation of
the error $\{E_m^t\}$.
Let $\Delta t=\frac{1}{2n}$.
To obtain sufficiently strong estimates, we need the following lemma that will let us control decay in time of the time derivatives of $u$.
Let us recall that $u(x,t)$ solves
\begin{align}\label{maineq2}
\partial_t u = -\frac{u^2}{\pi(u^2+Hu^2)}~\pa_x\left(\frac{Hu}{u}\right)= - \frac{1}{\pi}\frac{u \Lambda u -  Hu \partial_x u}{u^2+ Hu^2}.
\end{align}
\begin{lemma}\label{utlemma}
Let $u(x,t)$ be the solution of \eqref{maineq2} corresponding to $u_0 \in H^s(\S),$ $s>7/2.$ Then for every time $t,$
\begin{align}\label{uteq}
\|\partial_t u(\cdot,t)\|_{L^\infty} \leq C\|\partial_x^2 u(\cdot,t)\|_{L^\infty},\quad
  \max\big\{\|\partial^2_{tx} u(\cdot, t)\|_{L^\infty}, \|\partial_t^2 u(\cdot, t)\|_{L^\infty}\big\} \leq C\|\partial_x^3 u(\cdot, t)\|_{L^\infty},
\end{align}
with the constant $C$ that may only depend on $u_0.$ In particular,
\[ \|\partial_t u(\cdot, t)\|_{L^\infty} + \|\partial^2_t u(\cdot, t)\|_{L^\infty} +\|\partial^2_{tx} u(\cdot, t)\|_{L^\infty} \leq C \ux(t), \]
where $\ux$ is defined in \eqref{uxdelta}.
\end{lemma}
\begin{proof}
The bound for $\partial_t u$ follows directly from \eqref{maineq2}, global regularity, and the bound \eqref{Ht1124} that yields
$\|Hu\|_{L^\infty} \leq C\|\partial_x u\|_{L^\infty},$ as well as
\begin{equation}\label{aux1212} \|\Lambda u \|_{L^\infty} = \|H \partial_x u\|_{L^\infty} \leq C(\alpha)\|\pa_x u\|_{C^{\alpha}} \leq C\|\partial^2_x u\|_{L^\infty}, \end{equation}
for every $\alpha>0.$
Note that for a sufficiently regular function $u$ on $\S$ we have $\|\partial_x u\|_\infty \leq C\|\partial_x^2 u \|_\infty.$

The bound on $\| \partial_{tx}^2 u\|_{L^\infty}$ is proved by differentiating
\eqref{maineq2} and carrying out estimates very similar to the above one.
Additional term that we need to estimate for $\| \partial_{t}^2 u\|_{L^\infty}$ takes form
\begin{equation}\label{aux11212}
  \|\pa_t\Lambda u\|_{L^\infty}=\frac{1}{\pi}\left\|\Lambda \left(\frac{ u \Lambda u- Hu \partial_x u }{u^2 + (Hu)^2} \right) \right\|_{L^\infty}. \end{equation}
This term can be controlled by using \eqref{aux1212} as well as
 \[  \|\Lambda f\|_{C^\alpha}=\|H\pa_xf\|_{C^{\alpha}} \leq C(\alpha)\|\pa_xf\|_{C^\alpha}, \]
where we can set $f$ equal to the expression in the brackets in \eqref{aux11212}, and elementary calculations.
\end{proof}

By the definition of the error \eqref{Ej1122}, we have
\begin{equation}\label{error1}
  E_m^{t+\Delta t}-E_m^t=(y_{m+1}-y_m-x_{m+1}+x_m)-\frac{1}{2nu(\bar{y}_m,
    t+\Delta t)}+\frac{1}{2nu(\xb_m,t)}.
\end{equation}
Using Lemma \ref{utlemma}, we find that 
\begin{align}
  &\frac{1}{u(\bar{y}_m,t+\Delta t)}-\frac{1}{u(\xb_m,t)}\\
  &=\pa_t\left(\frac{1}{u}\right)(\xb_m,t)\cdot\Delta t+O(\ux \Delta t^2)
  +\pa_x\left(\frac{1}{u}\right)(\xb_m,t+\Delta t)\cdot(\bar{y}_m-\xb_m)
    +O(\ux |\bar{y}_m-\xb_m|^2)\\
  &=-\frac{\pa_tu(\xb_m,t)}{u(\xb_m,t)^2}\Delta t
    -\frac{\pa_xu(\xb_m,t)}{u(\xb_m,t)^2}(\bar{y}_m-\xb_m)
    +O(\ux n^{-2})\\
  &=\frac{1}{\pi(u(\xb_m,t)^2+Hu(\xb_m,t)^2)}~\pa_x\left(\frac{Hu}{u}\right)(\xb_m,t)\cdot\frac{1}{2n}\\
    &\quad-\frac{\pa_xu(\xb_m,t)}{u(\xb_m,t)^2}(y_{m+1}-x_{m+1})
      +\frac{\pa_xu(\xb_m,t)}{2u(\xb_m,t)^2}(y_{m+1}-y_m-x_{m+1}+x_m)
      +O(\ux n^{-2}), \label{ueq1114}
\end{align}
where we have used \eqref{maineq2}
in the last equality.
Let us substitute the estimate \eqref{ueq1114} into \eqref{error1} and get
\begin{align}
 & E_m^{t+\Delta
  t}-E_m^t=\,(y_{m+1}-y_m-x_{m+1}+x_m)\left(1-\frac{\pa_xu}{4nu^2}\right)\label{Eevo}\\
  &-\frac{1}{4\pi n^2(u^2+Hu^2)}~\pa_x\left(\frac{Hu}{u}\right)+\frac{\pa_xu}{2nu^2}\cdot(y_{m+1}-x_{m+1})+O(\ux n^{-3}).\nonumber
\end{align}
Here, for simplicity we no longer indicate the spatial and time dependence,
as all quantities are all evaluated at $(\xb_m, t)$.

Deploying the estimate \eqref{yyxx} in \eqref{Eevo}, we observe that the
leading terms of the order $O(n^{-2})$ cancel.
We obtain
\begin{align}
 & E_m^{t+\Delta t}-E_m^t\label{Eevo2}\\
&=\left[-\frac{\pa_xu}{2nu^2}(y_{m+1}-x_{m+1})
    +\frac{1}{4\pi
     n^2(u^2+Hu^2)}\pa_x\left(\frac{Hu}{u}\right)\right]\cdot
  O\big(n\|E\|_\infty+n^{-1/2}\big)\nonumber\\
  &+\frac{u Hu}{2\pi
    n^2 (u^2+Hu^2)^2}
    \pa_x\left(\frac{Hu}{u}\right) (A_{1,m}+A_{2,m})\cdot \big(1+O(n\|E\|_\infty+n^{-1/2})\big)\nonumber\\
  &+\big(1+O(n\|E\|_\infty+n^{-1/2})\big)\cdot\frac{\sin^2z}{8\pi^2n^2u^2}\times\nonumber\\
  &\qquad\times(D_{1,m}+D_{2,m}+D_{3,m}+B_{1,m}+B_{2,m}+H_{1,m}+H_{2,m}+H_{3,m})\nonumber\\
  &+O\big(\ux n^{-5/2}+\ux n^{-1}\|E\|_\infty\big)
    +O(\ux n^{-3}+n^{-3/2}\|E\|_\infty+(\log n)^2\|E\|_\infty^2+n(\log n)^2 \|E\|_\infty^3).\nonumber
\end{align}
Let us examine \eqref{Eevo2}. 
The first term on the right side of \eqref{Eevo2} can be estimated by
\[O\left(\ux n^{-2}\cdot (n\|E\|_\infty+n^{-1/2})\right)=O(\ux
  n^{-5/2}+\ux n^{-1}\|E\|_\infty).\]
The second term on the right side of \eqref{Eevo2} looks like
\begin{align}
&  O\left(\ux n^{-2}\cdot n\log n\|E\|_\infty\cdot
    (1+n\|E\|_\infty+n^{-1/2})\right)\\
&=O(\ux n^{-1}\log n\|E\|_\infty+\ux\log n\|E\|_\infty^2+\ux n^{-3/2}\log
  n\|E\|_\infty). \label{err11115}
\end{align}
The first error in \eqref{err11115} is supercritical and requires additional
treatment. We write the corresponding term explicitly:
\begin{equation}\label{Aterms}
  \frac{u Hu}{2\pi n^2 (u^2+Hu^2)^2}
  \pa_x\left(\frac{Hu}{u}\right) (A_{1,m}+A_{2,m}).
\end{equation}
Below, we will prove estimates that will allow to absorb this part into the dissipation, namely into the $D_{i,m}$ expressions in
\eqref{Eevo2}.

For $A_{1,m}$, decompose the sum in \eqref{A1simp}
into two parts and use  the estimate similar to \eqref{cotsum}:
\begin{align}
  A_{1,m}=&\,\frac{1}{4\pi}\sum_{j\in\Smc}\cot\frac{y_m-x_j}{2}\big((E_{j-1}-E_m)+(E_j-E_m)\big)\\
  &+\frac{E_m}{2\pi}\sum_{j\in\Smc}\cot\frac{y_m-x_j}{2}+O(\ux
    n\|E\|_\infty)\\
  =&\, H_{4,m}+O( n\|E\|_\infty+ n^2\log n\|E\|_\infty^2),\label{err41115}
\end{align}
where $H_{4,m}$ has the form
\begin{equation}\label{H4}
  H_{4,m}=\sum_{j\in\Smc\cup\{j_+-1\}}h_4(j,m)(E_j-E_m),
\end{equation}
with coefficients
\begin{equation}\label{h4}
  h_4(j,m)=\begin{cases}
    \frac{1}{4\pi}\left(\cot\frac{y_m-x_j}{2}+\cot\frac{y_m-x_{j+1}}{2}\right),
    & j\in\Smc\backslash\{j_--1\}\\
    \frac{1}{4\pi}\cot\frac{y_m-x_{j_--1}}{2},    & j=j_--1 \\
    \frac{1}{4\pi}\cot\frac{y_m-x_{j_+}}{2},    & j=j_+-1
    \end{cases}=O(n|m-j|^{-1}).
\end{equation}

Now, for $A_{2,m}$, decompose the sum in  \eqref{A2simp} into
two parts and use the estimate \eqref{sumcancel} and a straightforward bound
on at most $O(1+n^{3/2}\|E\|_\infty)$ mismatched summands to arrive at
\begin{align}
  A_{2,m}=&\,\frac{1}{\pi}\left(\sum_{j=j_-}^{m-1}\sum_{l=j}^{m-1}\frac{E_l-E_m}{(m-j)(y_m-x_j)}
  +\sum_{j=m+1}^{j_+-1}\sum_{l=m}^{j-1}\frac{E_l-E_m}{(j-m)(y_m-x_j)}\right)\\
  &+\frac{E_m}{\pi}\left(\sum_{j=j_-}^{m-1}\frac{1}{y_m-x_j}
  +\sum_{j=m+1}^{j_+-1}\frac{1}{y_m-x_j}\right)+O(n\|E\|_\infty)\\
=&\, B_{3,m}+O( n\|E\|_\infty+ n^2\log n\|E\|_\infty^2), \label{err51115}
\end{align}
where $B_{3,m}$ has the form
\begin{equation}\label{B3}
  B_{3,m}=\sum_{l\in\Sm\backslash\{m\}}b_3(l,m)(E_l-E_m),
\end{equation}
with coefficients
\begin{equation}\label{b3}
  b_3(l,m)=\begin{cases}
    \sum_{j=j_-}^{l}\frac{1}{(m-j)(y_m-x_j)},&j_-\leq l \leq m-1\\
    \sum_{j=l+1}^{j_+-1}\frac{1}{(j-m)(y_m-x_j)},&m+1\leq l \leq j_+-2
  \end{cases}=O\big(n|m-l|^{-1}\big).
\end{equation}
Observe that after multiplication by a factor in \eqref{Aterms}, the error terms
in \eqref{err41115} and \eqref{err51115} become of the order $O(\ux n^{-1}\|E\|_\infty+\ux \log n \|E\|_\infty^2).$

We combine \eqref{Aterms} with the term on the fourth and fifth lines in \eqref{Eevo2}. Let us introduce the
final dissipation term, covering all scales. It equals
\begin{equation}\label{D}
  \sum_{j\neq m}\kappa(j,m)(E_j-E_m),
\end{equation}
with the coefficients
\begin{align}
  \kappa(j,m)=&\,\big(1+O(n\|E\|_\infty+n^{-1/2})\big)\cdot\frac{\sin^2z}{8\pi^2n^2u^2}\left[\sum_{i=1}^3d_i(j,m)+\sum_{i=1}^2b_i(j,m)+\sum_{i=1}^3h_i(j,m)\right]\nonumber\\
&+\frac{u Hu}{2\pi n^2 (u^2+Hu^2)^2}\cdot
  \pa_x\left(\frac{Hu}{u}\right)\big(b_3(j,m)+h_4(j,m)\big).\label{d}
\end{align}
Then, according to our estimates, the evolution of $E_m^t$ in \eqref{Eevo2} can be summarized as
\begin{align}
  E_m^{t+\Delta t}-E_m^t
  =&\,\sum_{j\neq m}\kappa^t(j,m)(E_j^t-E_m^t)
  +O\big(\ux n^{-5/2}+\ux n^{-1}\|E^t\|_\infty\big)\nonumber\\
  & +O(n^{-3/2}\|E^t\|_\infty+(\log n)^2 \|E^t\|_\infty^2+n(\log
  n)^2\|E^t\|_\infty^3).\label{Emt}
\end{align}

Finally,
let us examine more carefully the coefficients $\kappa^t(j,m).$
\begin{theorem}\label{lem:dpositive}
The error $E_m^t$ propagates according to the evolution equation
\begin{align}\label{E122}
  &E^{t+\Delta t}-E^t = {\mathcal L}^t E^t \\
  & \quad+ O\big(\ux n^{-5/2}+\ux n^{-1}\|E^t\|_\infty\big)
  +O(n^{-3/2}\|E^t\|_\infty+(\log n)^2 \|E^t\|_\infty^2+n(\log
  n)^2\|E^t\|_\infty^3),
\end{align}
where the operator ${\mathcal L}^t$ is a diffusive operator given by
\begin{align}\label{L122}
({\mathcal L}^t E^t)_m = \sum_{j,j\neq m}\kappa^t(j,m)(E_j^t-E_m^t).
\end{align}
The diffusion coefficients $\kappa^t(j,m)$ satisfy
\begin{align}\label{kappa122}
& \kappa^t(j,m) = \frac{1}{16 \pi^2 n^2 (u^2 + Hu^2) \sin^2 \frac{x_m^{t+\Delta t}-x_j^t}{2}} \times \\
& \left(1+ O(n^{-1/2} +\ux |j-m|n^{-1}+n \log n \|E\|_\infty+ (\log n)^2 (n^2\|E\|_\infty^2+n^3\|E\|_\infty^3 +n^4\|E\|_\infty^4))\right),
\end{align}
where $u$ and $Hu$ can be evaluated at $x_m^t$, and $x_j^t,$ $x_m^{t+\Delta t}$ are the polynomial roots at time $t$ and $t+\Delta t,$ respectively. Note that $x_m^{t+\Delta t}=y_m$ in our usual notation.
In particular, assuming \eqref{Esmall}, for all $n$ exceeding a threshold $n_0(\tilde u_0)$ that may only depend on the initial data we have
\begin{equation}\label{propkappa}
\kappa^t(j,m) \sim |m-j|^{-2};
\end{equation}
the constants involved in $\sim$ can be chosen uniformly for all $t,$ $m$ and $j$ and depend only on $\tilde u_0.$
\end{theorem}

\begin{remark}
Observe that given the formula \eqref{fraclap} for the fractional Laplacian, $\Delta t = \frac{1}{2n}$ relationship, \eqref{L122} and \eqref{kappa122}, the equation \eqref{E122} looks like a discretization of the continuous in time and space PDE
that to the main order is just a modulated fractional heat equation:
\[ \partial_t E(x,t) = \frac{u(x,t)}{\pi(u^2(x,t)+Hu^2(x,t))} \Lambda E (x,t) + errors.  \]
The leading term of this equation coincides with the dissipative part of \eqref{maineq1}.
The appearance of such a simple PDE controlling the error between the evolution of roots and the PDE \eqref{maineq} is quite surprising.
In our argument, the origin of this form of the main term is not hard to see in the near field; in the far field and on the interface it is somewhat less transparent;
but to guess beforehand that all other ingredients will turn out to be either subcritical, or good critical, or supercritical but directly absorbable by dissipation
would be quite hard (at least to us).
It is interesting to explore if a parallel structure is present and can be more intuitively understood in a related continuous setting of free fractional convolution of measures.
\end{remark}
\begin{proof}
For simplicity, we will omit time dependence in notation for $\kappa(j,m)$.
Let us start with the near field $j\in\Sm\backslash\{m,j_-,j_+-1\}$, where
  \begin{align*}
    \kappa(j,m)=&\,\big(1+O(n\|E\|_\infty+n^{-1/2})\big)\cdot\frac{\sin^2z}{8\pi^2n^2u^2}\cdot\big(d_1(j,m)+b_1(j,m)+b_2(j,m)\big)\\
    &+\frac{u Hu}{2\pi n^2 (u^2+Hu^2)^2}\cdot
      \pa_x\left(\frac{Hu}{u}\right) \cdot b_3(j,m).
  \end{align*}
  Note that there are additional contributions from $D_{3,m}$ and other terms at
  the boundary $j=j_-, j_+-1$ that will be discussed later.

  First, observe that  \eqref{b3} implies
\begin{align} {\tilde b}_{3,m}(j): & =  \frac{u Hu}{2\pi n^2 (u^2+Hu^2)^2}\cdot
      \pa_x\left(\frac{Hu}{u}\right) \cdot b_3(j,m) \\ & =O(\ux n^{-2}\cdot n|m-j|^{-1}) \lesssim |m-j|^{-2}n^{-1/2} \label{b3est1115}
\end{align}
for $j \in \Sm.$
Next, we have already shown in \eqref{d1}, \eqref{b1d1} and \eqref{b2} that
  \begin{equation}\label{d11115} d_1(j,m)= \frac{2}{(y_m-x_j)(y_{m+1}-x_{j+1})} \sim \frac{n^2}{|m-j|^2};\,\, |b_1(j,m)|, |b_2(j,m)|\lesssim
    n^{-1/2}d_1(j,m).\end{equation}
To incorporate the contribution \eqref{b3est1115} from ${\tilde b}_{3,m}(j)$, let us analyze the pre-factor in the fourth line of \eqref{Eevo2}.
 Applying \eqref{sinz}, we get
\begin{align}
 & \big(1+O(n\|E\|_\infty+n^{-1/2})\big)\cdot\frac{\sin^2z}{8\pi^2n^2u^2} \label{coefest}
  =\frac{1}{8\pi^2n^2(u^2+Hu^2)} \\
& +O\left(n^{-5/2}+n^{-1}\log n\|E\|_\infty+(\log n)^2 \|E\|_\infty^2+n(\log n)^2 \|E\|_\infty^3\right).
\end{align}
Then, taking into account \eqref{b3est1115} and \eqref{d11115}, as well as \eqref{ymxj} and \eqref{ymxm}, a short computation leads to
\begin{align}
& \kappa(j,m)  = \big(1+O(n\|E\|_\infty+n^{-1/2})\big)\cdot\frac{\sin^2z}{8\pi^2n^2u^2}\cdot\big(d_1(j,m)+b_1(j,m)+b_2(j,m)\big)+ {\tilde b}_{3,m}(j) \nonumber \\
& =\frac{d_1(j,m)}{8\pi^2n^2 (u^2 +Hu^2)} 
\left( 1+ O(n^{-1/2}+n \log n \|E\|_\infty+ n^2 (\log n)^2 \|E\|^2_\infty+ n^3 (\log n)^2 \|E\|^3_\infty)\right) \nonumber\\
& \sim |m-j|^{-2}(1+O(n^{-1/2} +n \log n \|E\|_\infty+ n^2 (\log n)^2 \|E\|^2_\infty+ n^3 (\log n)^2 \|E\|^3_\infty)) \label{dissI1115}
\end{align}
for $j\in\Sm\backslash\{m, j_-, j_+-1\}.$

To pass from \eqref{dissI1115} to the form \eqref{kappa122}, note first that
\[ \frac{1}{y_m-x_j}-\frac{1}{y_{m+1}-x_{j+1}} = \frac{1}{y_m-x_j} \cdot O(n\|E\|_\infty +n^{-1}), \]
where we used \eqref{cotapp1} and an estimate $|y_{m+1}-y_m-x_{m+1}+x_m|=O(\|E\|_\infty+n^{-2})$ that follows from \eqref{yyxx}.
Here we split the numerator in the usual manner and used that from the more precise estimate \eqref{yyxx}
it follows that $|y_{m+1}-y_m-x_{m+1}+x_m| = O(\|E\|_\infty + n^{-2}).$ Thus we can replace $y_{m+1}-x_{j+1}$ in \eqref{d11115}
by $y_m-x_j$ and absorb the difference into the error.
Then observe that for $|y_m -x_j| \lesssim n^{-1/2},$
\[ \frac{2}{(y_m-x_j)^2} - \frac{1}{2 \sin^2 \frac{y_m-x_j}{2}} = \frac{1}{2 \sin^2 \frac{y_m-x_j}{2}} \cdot O (|y_m-x_j|^2) = \frac{1}{2 \sin^2 \frac{y_m-x_j}{2}} \cdot O ( n^{-1}),  \]
proving \eqref{kappa122} in near field.

Next, we consider the interface $j=j_--1, j_-, j_+-1, j_+$. The analysis for $j=j_-$ and $j_+-1$ is similar,
so let us focus on the former. According to \eqref{d3},
\begin{equation}\label{d3j-1116}
d_3(j_-,m) = \frac{n}{2}\left(1+ O(n^{-1/2}+n \log n \|E\|_\infty+n (\log n)^2 \|E\|^2_\infty) \right).
\end{equation}
We also get from \eqref{h1} that
\begin{equation}\label{err11116} h_1(j,m)=O(n^2\|E\|_\infty+\ux n^{1/2}),\end{equation}
and from \eqref{h3}, \eqref{h4} that
\begin{equation}\label{err21116} |h_3(j,m)|+|h_4(i,m)|=O(n^{1/2})\end{equation}
 for all $j = j_--1, j_-,j_+-1,j_+$.

The term $H_{2,m}$, on the other hand, yields an essential contribution: the estimate \eqref{h2boundary} shows that 
\begin{equation}\label{err31116} h_2(j_-,m) = -\frac{n}{2}\left(1+O(n^{-1/2})\right).\end{equation}
Finally, since $j_-\in\Sm$, there is also a contribution from $D_{1,m}$, with
\begin{equation}\label{d11116} d_1(j_-,m)=\frac{2}{(y_{m+1}-x_{j_-+1})(y_m-x_{j_-})} = \frac{2}{n^{-1}+O(n^{-3/2})} = 2n \left(1+O(n^{-1/2})\right). \end{equation}
Then, due to \eqref{d3j-1116}, \eqref{err31116}, and \eqref{d11116} we have
\begin{equation}\label{difcomb1116} d_1(j_-,m)+ d_3(j_-,m)+h_2(j_-,m) = 2n \left(1+O(n^{-1/2}+n \log n \|E\|_\infty+n (\log n)^2 \|E\|^2_\infty)\right) >0.\end{equation}
Combining \eqref{sinz}, \eqref{difcomb1116}, \eqref{err11116}, \eqref{err21116}, \eqref{b3est1115} and \eqref{d}, we obtain
\begin{align}\label{kappa31116} \kappa(j_-,m) = &\, \frac{1}{4\pi^2(u^2+Hu^2)n} \times \\
&\times \left(1+O\left(n^{-1/2}+n \log n \|E\|_\infty + n^2 (\log n)^2 \|E\|^2_\infty +n^3 (\log n)^2 \|E\|_\infty^3 \right) \right).\nonumber\end{align}

For $j=j_+$ (and similarly for $j=j_--1$), from \eqref{d3} and
\eqref{h2boundary} we find
\[d_3(j_+,m)+h_2(j_+,m) = 2n \left(1+O(n^{-1/2}+n \log n \|E\|_\infty+n (\log n)^2 \|E\|^2_\infty) \right).\]
Taking into account \eqref{sinz}, \eqref{err11116}, \eqref{err21116}, \eqref{h4} and \eqref{d},
we arrive at the estimate for $\kappa(j_+,m)$ identical to \eqref{kappa31116}.
Note also that the main term in \eqref{kappa31116} coincides with \eqref{kappa122}, since
$|y_m-x_j| = n^{-1/2}+ O(n^{-1})$ for $j = j_--1, j_-,j_+-1,$ and $j_+.$
It follows that $\kappa(j,m)$ at the interface satisfy \eqref{kappa122}, and in particular
\begin{equation}\label{kappainter1116} \kappa(j,m) \sim |m-j|^{-2} \end{equation}
for all these $j$ for all sufficiently large $n.$

Finally, for the far field $j\in\Smc\backslash\{j_--1, j_+\}$, we have
  \begin{align}\nonumber
    \kappa(j,m)=&\,\big(1+O(n\|E\|_\infty+n^{-1/2})\big)\cdot\frac{\sin^2z}{8\pi^2n^2u^2}\cdot\big(d_2(j,m)+h_1(j,m)+h_2(j,m)+h_3(j,m)\big)\\
    &+\frac{u Hu}{2\pi n^2 (u^2+Hu^2)^2}\cdot
      \pa_x\left(\frac{Hu}{u}\right) \cdot h_4(j,m). \label{d21116}
  \end{align}
  We have shown in \eqref{d2}, \eqref{h1}, \eqref{h2out1116} and \eqref{h3}
  that
  \begin{align}\label{d2125} d_2(j,m)=&\, \frac{1}{2\sin^2 \frac{y_m-x_j}{2}}\left(1+O(n^{-1/2}+n\|E\|_\infty)\right), \\
  \sum_{i=1}^3h_i(j,m)=&\,O\left(\frac{\ux
      n}{|m-j|}+\frac{n^3\|E\|_\infty}{|m-j|^2}+\frac{n^2}{|m-j|^3}\right).\label{d2out1116}
      \end{align}
The last two terms in the expression for $\sum_{i=1}^3h_i(j,m)$ in \eqref{d2out1116} can be absorbed into $d_2(j,m)$ since
due to \eqref{ymxj}, we have 
\begin{equation}\label{d2lower1116}  \frac{C(u_0) n^2}{|m-j|^2} \geq d_2(j,m) \geq \frac{2u_{\min}^2 n^2}{|m-j|^2}, \end{equation}
while
\begin{equation}\label{err71116} O\left(\frac{n^3\|E\|_\infty}{|m-j|^2}+\frac{n^2}{|m-j|^3}\right) = \frac{n^2}{|m-j|^2} \cdot O(n\|E\|_\infty+n^{-1/2}),\quad\forall~j\in\Smc.\end{equation}
The first term $O(\ux n|m-j|^{-1})$ is trickier, since when $|m-j|\sim n$, this term is
generally of the order $O(1)$, similarly to $d_2(j,m)$.
Due to \eqref{Aterms}, \eqref{h4}, a term of the order $O(\ux n|m-j|^{-1})$ also arises from $H_{4,m}$ in \eqref{d21116}.
We will handle these perturbations as follows. Using \eqref{d2lower1116}, let us write
\begin{equation}\label{err81116} O(\ux n|m-j|^{-1})= n^2|m-j|^{-2}O(\ux n^{-1}|m-j|) = d_2(j,m)O(\ux n^{-1}|m-j|). \end{equation}
Applying \eqref{coefest}, \eqref{d21116}, \eqref{d2125}, \eqref{err71116}, \eqref{err81116} and \eqref{d2out1116} we get that for $j\in\Smc\backslash\{j_--1, j_+\},$
\begin{align}\label{kappaout41116}
&\kappa(j,m)  = \frac{1}{16\pi^2n^2(u^2+Hu^2)\sin^2\frac{y_m-x_j}{2}}\times\\
&\quad\times\left(1+O(\ux n^{-1}|m-j| + n^{-1/2}+n \log n \|E\|_\infty   + n^2 (\log n)^2 \|E\|_\infty^2 + n^3 (\log n)^2 \|E\|_\infty^3)\right).  \nonumber
\end{align}
This already proves \eqref{kappa122}, but it remains to show \eqref{propkappa}.
Observe that there exists $\beta(u_0)>0$ such that if $|m-j| \leq \beta(u_0)n,$ then the dangerous error term in \eqref{kappaout41116} satisfies
\begin{equation}\label{kappaout51116}
\left| O(\ux n^{-1}|m-j|) \right| \leq \frac13,
\end{equation}
for all times, as the constants depending on $u_0$ that are involved in $O$ are uniformly bounded due to Theorem~\ref{mainthm1}.
On the other hand, observe that
\begin{align}\label{err91116}
 &\frac{1}{8\pi^2n^2(u^2+Hu^2)} \sum_{n \geq |m-j| > \beta(u_0)n} d_2(j,m) O(\ux n^{-1}|m-j|)(E_j-E_m) \\
 &\lesssim \ux n^{-1} \|E\|_\infty \sum_{n \geq |m-j| > \beta(u_0)n} |m-j|^{-1}
 =O(\ux n^{-1}\|E\|_\infty).
\end{align}
It follows that for $j$ such that $|m-j| > \beta(u_0)n,$ the $O(\ux n^{-1}|m-j|)$ term in \eqref{kappaout41116} can be removed and absorbed into
the favorable critical error $O(\ux n^{-1}\|E\|_\infty)$ in \eqref{Emt}. Using \eqref{ymxj}, this implies that we can arrange
so that for all sufficiently large $n,$
\begin{equation}\label{kappaout81116} \kappa(j,m) \geq  \frac{u_{\min}^2}{16\pi^2 (u^2+Hu^2)|m-j|^2} \sim |m-j|^{-2}. \end{equation}
For the rest of the paper, we will assume that this arrangement is made, and that in \eqref{Emt}
\begin{equation}\label{delerr125}
O(\ux |j-m| n^{-1}) \leq \frac13
\end{equation}
for all $j,m,t.$


\end{proof}

Next, we show an $l^1$ bound on $\{\kappa^t(\cdot,m)\}$. Since our evolution is discrete in time, we need control on the
$l^1$ norm of our kernel to ensure stability that allows an argument similar to maximum principle that we will use to prove Theorem~\ref{mainthm2}.
\begin{lemma}\label{lem:l1bound}
 Fix any time $t,$ and assume \eqref{Esmall} holds.  Then, for
  all $n \geq n_0(\tilde u_0)$ large enough, for all $m,$ we have
  \begin{equation}\label{l1bound}
    \sum_{\{j: j\ne m\}}\kappa^t(j,m)=S(t)<1-\rho.
  \end{equation}
  The constant $\rho>0$ depends only on $u_0$ and may be chosen uniformly for all times, all sufficiently large $n \geq n_0(\tilde u_0),$ and $m.$
\end{lemma}
\begin{proof}
From \eqref{propkappa}  it follows that the contribution of the sum over $j \in \Smc$ is of the order $O(n^{-1/2})$.
Then \eqref{d11115}, \eqref{dissI1115} and \eqref{kappa31116} imply
\begin{align} \sum_{j\neq m}\kappa(j,m) &=\frac{1}{8\pi^2n^2(u^2+Hu^2)}\sum_{j\in\Sm\backslash\{m\}}\frac{2}{(y_{m+1}-x_{j+1})(y_m-x_j)} \\
& + O(n^{-1/2}+n \log n \|E\|_\infty+ n^2 (\log n)^2 \|E\|^2_\infty+n^3 (\log n)^2 \|E\|^3_\infty).\label{l1a1117}\end{align}
  Let us break the sum in \eqref{l1a1117} into three regions: (i) $m+2\leq j\leq j_+-1$,
  (ii) $j_-\leq j\leq m-1$, and (iii) $j=m+1$.
  For the first region ($j\geq m+2$), we have
  \begin{align*}
    x_j-y_m=&\,x_{m+1}-y_m+\sum_{l=m+1}^{j-1}\left(\frac{1}{2nu(\xb_l)}+E_l\right)\\
    \geq&\,\frac{j-m-1}{2nu(\xb_m)}+O\left(\frac{|m-j|^2}{n^2}+|m-j|~\|E\|_\infty\right);\\
    x_{j+1}-y_{m+1}\geq&\,\frac{j-m-1}{2nu(\xb_m)}+O\left(\frac{|m-j|^2}{n^2}+|m-j|~\|E\|_\infty\right).
  \end{align*}
  Then,
  \begin{align*}
    &\frac{1}{8\pi^2n^2(u^2+Hu^2)}\sum_{j=m+2}^{j_+-1}\frac{2}{(y_{m+1}-x_{j+1})(y_m-x_j)}\\
    &\leq \frac{1}{8\pi^2n^2(u^2+Hu^2)}\sum_{j=m+2}^{j_+-1}\left[\frac{2}{\left(\frac{j-m-1}{2nu}\right)^2}+O\left(\frac{n}{|m-j|}+\frac{n^3\|E\|_\infty}{|m-j|^2}+\frac{n^4\|E\|^2_\infty}{|m-j|^2}\right)\right]\\
    &=\frac{u^2}{\pi^2(u^2+Hu^2)}\sum_{j=m+2}^{j_+-1}\frac{1}{(j-m-1)^2}
      +O\big( n^{-1}\log n+n\|E\|_\infty+n^2\|E\|_\infty^2\big)\\
    &\leq \frac{u^2}{6(u^2+Hu^2)}+O\big( n^{-1}\log n+n\|E\|_\infty+n^2\|E\|_\infty^2\big).
  \end{align*}
  Here, we have used the Euler identity $\sum_{k=1}^\infty
  k^{-2}=\frac{\pi^2}{6}$.
  A similar argument can be done for the second region $j\leq m-1$, yielding an identical bound.

  Finally, the summand corresponding to $j=m+1$ can be large if $y_m$ is close to $x_{m+1}$.
  However, we established a fairly precise control over this distance in Lemma~\ref{xmymlem117}.
  Let us apply the lower bound estimate \eqref{lower} and obtain
  \begin{align*}
      &\frac{1}{8\pi^2n^2(u^2+Hu^2)}\cdot\frac{2}{(y_{m+1}-x_{m+2})(y_m-x_{m+1})}\\
    &=\frac{1}{8\pi^2n^2(u^2+Hu^2)}\cdot\frac{2}{\left(\frac{1}{2\pi
      nu}\arccot\left(\frac{Hu}{u}\right)\right)^2}+O\big(n^{-1/2}+n\log
      n\|E\|_\infty+n^2 (\log n)^2 \|E\|_\infty^2 \big)\\
    &=\frac{u^2}{(u^2+Hu^2)\left(\arccot\left(\frac{Hu}{u}\right)\right)^2} +O\big(n^{-1/2}+n\log
      n\|E\|_\infty+n^2 (\log n)^2 \|E\|_\infty^2 \big).
  \end{align*}
  Here $u$ and $Hu$ are evaluated at $\xb_m,$ and error coming from the main term for $y_{m+1}-x_{m+2}$
  in \eqref{lower} is absorbed into the lager $O(n^{-3/2})$ error.
  Putting the three regions together, we get the following bound
  \begin{equation}\label{aux61214} \sum_{j,j\neq m}\kappa(j,m)\leq
    \frac{1}{1+\left(\frac{Hu}{u}\right)^2}\cdot\left(\frac13+\frac{1}{
        \left(\arccot \frac{Hu}{u} \right)^2}\right)+O\big(n^{-1/2}+n\log
      n\|E\|_\infty+n^2 (\log n)^2 \|E\|_\infty^2 \big).\end{equation}
  Set $a=\arccot \left(\frac{Hu}{u}\right)\in(0,\pi)$. The value of
  $a$ depends on $u$. Since $\frac{Hu}{u}$ is uniformly bounded in
  time, $a\geq a(u_0)>0$ is bounded away from zero uniformly for all times and $x_m$.
  We rewrite the bound above in terms of $a$ as follows:
  \[ \sum_{j,j\neq m}\kappa(j,m)\leq
    \sin^2a\cdot\left(\frac13+\frac{1}{a^2}\right)+o(1)=: F(a)+o(1).\]
  The $o(1)$ includes all errors from \eqref{aux61214} and is based on \eqref{Esmall}.
  Observe that $F(a)$ is a decreasing function in $(0,\pi)$, and
  clearly $\lim_{a\to0}F(a)=1$. Indeed, one can compute
  \[F'(a)=\frac{2\sin a}{3a^3}\big(a(a^2+3)\cos a-3\sin a\big).\]
  It is immediate that $F'(a)<0$ if $a\in[\frac{\pi}{2},\pi)$.
  For $a\in(0,\frac{\pi}{2})$, we claim that $a(a^2+3)\cos a-3\sin
  a<0$, or equivalently $\tan a>a+\frac{a^3}{3}$.
  But using Taylor series for $\tan a,$  we see that
  \begin{equation}\label{tancalc} \tan a - a -
    \frac{a^3}{3}= \sum_{i=2}^\infty \frac{(\tan)^{(2i+1)}(0) a^{2i+1}}{(2i+1)!} =\frac{2a^5}{15}+\frac{17a^7}{315}+\cdots
    >0,\quad\forall~ a\in\left(0,\frac\pi2\right).\end{equation}
  Indeed, the last inequality in \eqref{tancalc} follows from $\tan^{(2i)}(0)=0,$ $\tan^{(2i+1)}(0) \geq 0$ for all $i$ (this can be derived by induction using $f' = 1+ f^2$ for $f(x) =\tan x$).
  Therefore, we end up with
  \[\sum_{j\neq m}\kappa(j,m)\leq F(a(u_0))+o(1)<1\]
  for all $n \geq n_0(\tilde u_0).$
\end{proof}

\section{Proof of Theorem~\ref{mainthm2}}\label{proofT2}

It will be convenient for us to define
\begin{equation}\label{kappamm}
  \kappa^t(m,m) := 1-\sum_{\{j: j\ne m\}}\kappa^t(j,m) = 1-S(t).
\end{equation}
Then, we can express
\begin{align*}
  E_m^t + ({\mathcal L}^t E^t)_m=&\,  E_m^t+ \sum_{\{j: j\ne m\}} \kappa^t(j,m)(E^t_j-E^t_m)\\ =&\, (1-S(t)) E_m^t+\sum_{\{j: j\ne m\}}
  \kappa^t(j,m)E_j^t=\sum_{j=1}^{2n}\kappa^t(j,m)E_j^t.
\end{align*}
Therefore, the dynamics \eqref{E122} becomes
\begin{align}
  E_m^{t+\Delta t}=&\,\sum_{j=1}^{2n}\kappa^t(j,m)E_j^t
+O\big(\ux n^{-5/2}+\ux n^{-1}\|E^t\|_\infty\big)\nonumber\\
  &  +O(n^{-3/2}\|E^t\|_\infty+(\log n)^2 \|E^t\|_\infty^2+n(\log
  n)^2\|E^t\|_\infty^3).\label{Emt2}
\end{align}
The diffusion coefficients $\kappa^t(j,m)$ satisfy
\begin{equation}\label{kappasum}
  \sum_{j=1}^{2n}\kappa^t(j,m)=1
\end{equation}
due to the definition \eqref{kappamm}.
Also, from \eqref{propkappa} and Lemma \ref{lem:l1bound},
we know that $\kappa^t(j,m) \sim |m-j|^{-2}> 0$ if $j \ne m$, and $\kappa^t(m,m) \geq \rho>0.$

We continue with the following lemma, which establishes a bound that will play
for us a role similar to the mean zero condition for the Poincar\'e inequality.
\begin{lemma}\label{lem:sumzero}
  Suppose $\int_0^{2\pi}u_0(x)\,dx=1$. Then for all times $t$, we have
  \begin{equation}\label{sumzero}
    \sum_{j=1}^{2n} E_j^t u(\xb_j^t,t)=O(\ux n^{-2}).
  \end{equation}
\end{lemma}
\begin{proof}
  Observe that $\int_0^{2\pi}u(x,t)\,dx=1$ for all times. On the other
  hand, by midpoint rule
  \begin{align*}
    \int_0^{2\pi}u(x,t)\,dx=&\,
    \sum_j \left(u(\xb_j^t,t)(x_{j+1}^t-x_j^t)+O(\ux n^{-3})\right)\\
    =&\,\sum_j \left(\frac{1}{2n}+E_j^tu(\xb_j^t,t)+O(\ux
       n^{-3})\right)
    =1+\sum_j E_j^t u(\xb_j^t,t)+O(\ux n^{-2}).
  \end{align*}
  This leads to \eqref{sumzero} for all $t$.
\end{proof}

Next we obtain an improved estimate on the dissipative term $\sum_j\kappa^t(j,m)E_j^t$.
\begin{proposition}\label{diss125}
 Suppose $\int_0^{2\pi}u_0(x)\,dx=1$.
  Given any $t\geq0$ and $m\in\{1,\cdots,2n\}$, and provided that $n \geq n_0(\tilde u_0)$ is sufficiently large,
  there exists a function $h(t)$ given by \eqref{hdef125} below, such that
  \begin{equation}\label{kappaimprove}
    \sum_{j=1}^{2n}\kappa^t(j,m)E_j^t\leq \left(1-\frac{h(t)}{2n}\right)\|E^t\|_\infty+ O(\ux n^{-4}).
  \end{equation}
  The function $h(t)$ has the following properties: there exist constants $h_1(u_0),$ $h_0(u_0) > 0$ such that
  \[ h_1 \geq h(t) \geq h_0>0 \] for all $t,$ and, moreover,
  $h(t)$ converges to $1$ as $n,t \rightarrow \infty$ provided that \eqref{Esmall} remains true.
\end{proposition}
\begin{proof}
Let us define
\begin{align}\label{hdef125}
h(t) = \min_{j,m} \frac{2n \kappa^t(j,m)}{u(\xb_j,t)} \sum_{l=1}^{2n} u(\xb_l,t).
\end{align}
Here of course the roots $\xb_j$ also depend on $t,$ but we suppress this in notation so that not to make it too cumbersome.
Consider that
 \begin{align*}
    \sum_{j=1}^{2n}\kappa^t(j,m)E_j^t
    =&\,\sum_{j=1}^{2n}\left(\kappa^t(j,m)-\frac{h(t)}{2n \sum_{l=1}^{2n}u(\xb_l,t)}u(\xb_j^t,t)\right)E_j^t
       +\frac{h(t)}{2n \sum_{l=1}^{2n} u(\xb_l,t) }\sum_{j=1}^{2n} E_j^t u(\xb_j^t,t)\\
    =&\, \sum_{j=1}^{2n}\tilde\kappa^t(j,m)E_j^t + O(\ux n^{-4}),
  \end{align*}
  where we used Lemma \ref{lem:sumzero} and a straightforward upper bound $h = O(1)$ in the last step that follows from $\kappa^t(j,m) \sim |j-m|^{-2}.$
  We also defined
  \[\tilde\kappa^t(j,m):=\kappa^t(j,m) -\frac{h(t)}{2n \sum_{l=1}^{2n} u(\xb_l,t) }u(\xb_j^t,t).\]
Observe that $\tilde\kappa^t(j,m) \geq 0$   due to \eqref{hdef125}.
Then we have
  \[\sum_{j=1}^{2n}\tilde\kappa^t(j,m)E_j^t\leq
  \sum_{j=1}^{2n}\tilde\kappa^t(j,m) \|E^t\|_\infty  
    \leq\left(1-\frac{h(t)}{2n}\right)\|E^t\|_\infty,\]
    proving \eqref{kappaimprove}.

Note that due to \eqref{kappa122}, \eqref{ymxj}, \eqref{Bbound}, \eqref{delerr125} and \eqref{Esmall}, for all sufficiently large $n$ we have for all times
\begin{align}
& h(t) \geq \frac{4n^2 u_{\min}(t)}{u_{\max}(t) \cdot 16\pi^2 n^2(u_{\max}(t)^2 +Q^2)}\cdot \frac12 = \frac{u_{\min}(t)}{8 \pi^2 u_{\max}(t)(u_{\max}(t)^2 +Q^2)}, \\
& h(t) \leq \frac{4n^2 u_{\max}(t)}{u_{\min}(t) \cdot 16\pi^2 n^2 u_{\min}(t)^2} \cdot 2 = \frac{u_{\max}(t)}{2 \pi^2 u_{\min}(t)^3}. \label{hb125}
\end{align}
These estimates imply uniform in time bounds on $h(t)$ from above and below that depend only on $u_0.$
Moreover, due to \eqref{meanconv1117}, \eqref{higherder1117}, \eqref{kappa122}, and \eqref{Bbound}, we have that for large $n$ and $t,$ 
\begin{align}\label{hlim125}
 h(t) = 1 +O\left(e^{-\sigma t}+ n^{-1/2}+n \log n \|E^t\|_\infty+ (\log n)^2 (n^2 \|E^t\|^2_\infty+n^3 \|E^t\|^3_\infty+n^4 \|E\|_\infty^4)\right),
\end{align}
where used $\bar u = \frac{1}{2\pi}.$
\end{proof}

Now we are ready to finish the proof of Theorem~\ref{mainthm2}.
\begin{proof}[Proof of Theorem \ref{mainthm2}]
Observe that due to positivity of the coefficients $\kappa^t,$ for all times $t,$ we have
\begin{align}
    E_m^{t+\Delta t}\leq&\, \sum_{j=1}^{2n}\kappa^t(j,m)\|E^t\|_\infty
   +C_1\big(\ux n^{-5/2}+ \ux n^{-1}\|E^t\|_\infty\big)\nonumber\\
  & +C_2\left(n^{-3/2}\|E^t\|_\infty+(\log n)^2 \|E^t\|_\infty^2+n(\log
  n)^2\|E^t\|_\infty^3\right),\label{aux1124}
\end{align}
where $C_1$ and $C_2$ depend only on $u_0.$
Let us apply \eqref{kappaimprove} to the estimate \eqref{aux1124} and get
  \begin{align}
    \|E^{t+\Delta t}\|_\infty\leq&\, \left(1-\frac{h(t)}{2n}+C_1\ux(t)n^{-1}\right)\|E^t\|_\infty
   +C_1\ux(t) n^{-5/2}\nonumber\\
  & +C_2\left(n^{-3/2}\|E^t\|_\infty+(\log n)^2 (\|E^t\|_\infty^2+n\|E^t\|_\infty^3+n^2\|E\|_\infty^4)\right).\label{EmtDtimprove}
  \end{align}
  Let $C_3=C_3(u_0)$ be a constant such that for all $n \geq n_0(\tilde u_0)$ and all $t,$
\begin{align}
& C_1\ux(t) \leq \frac{C_3}{4} e^{-\sigma t}, \\
& |h(t)-1| \leq \frac{C_3}{4}\left(e^{-\sigma t}+ n^{-1/2}+n \log n \|E^t\|_\infty+ (\log n)^2 (n^2 \|E^t\|^2_\infty+n^3 \|E^t\|^3_\infty+n^4 \|E\|_\infty^4)\right),
\end{align}
and, given also \eqref{Esmall},
\begin{align} & C_2\left|n^{-3/2}+(\log n)^2 (\|E^t\|_\infty+n\|E^t\|_\infty^2+n^2\|E\|_\infty^3)\right| \leq \frac{C_3}{12} n^{-1-\epsilon/2}, \\
& n^{-1/2}+n \log n \|E^t\|_\infty+ (\log n)^2 (n^2 \|E^t\|^2_\infty+n^3 \|E^t\|^3_\infty+n^4 \|E\|_\infty^4) \leq n^{-\epsilon/2}.
\end{align}
  Then we have that for all sufficiently large $n \geq n_0(\tilde u_0)$
 \begin{align}\label{Esimpleev}
   \|E^{t+\Delta t}\|_\infty\leq&\, \left(1-\frac{1}{2n}+\frac{C_3e^{-\sigma t}}{2n}+\frac{C_3}{3n^{1+\epsilon/2}}\right)\|E^t\|_\infty + \frac{C_3e^{-\sigma t}}{n^{5/2}}
  \end{align}
  for all $t.$
  By iterating \eqref{Esimpleev}, we find that
  \begin{align}\label{final1214}
  \|E^t\|_\infty & \leq  \M_0 n^{-1-\epsilon} \prod_{i=0}^{2nt-1} \left( 1-\frac{1}{2n}+\frac{C_3 e^{- \frac{\sigma i}{2n}}}{2n}+\frac{C_3}{3n^{1+\epsilon/2}} \right) \\
 & + \frac{C_3}{n^{5/2}} \sum_{i=0}^{2nt-1} e^{-\frac{\sigma  i}{2n}} \prod_{s=i+1}^{2nt-1} \left(1-\frac{1}{2n}+\frac{C_3 e^{- \frac{\sigma  s}{2n}}}{2n}+\frac{C_3}{3n^{1+\epsilon/2}} \right) \\
 & \leq \M_0 n^{-1-\epsilon} e^{-t+\frac{4C_3}{\sigma}+\frac{C_3t}{n^{\epsilon/2}}} + \frac{C_3}{n^{5/2}}\sum_{i=0}^{2nt-1} e^{-\frac{\sigma  i}{2n}} e^{-t+\frac{i+1}{2n}+\frac{4C_3}{\sigma}+\frac{C_3}{n^{\epsilon/2}}(t-\frac{i+1}{2n})}.
  \end{align}
  Using that $\sigma >1$ to sum the geometric progression, and performing elementary simplifications, we obtain
  \begin{align}\label{finalest1214}
   \|E^t\|_\infty \leq C \left( \M_0 n^{-1-\epsilon}+ n^{-3/2} (1-e^{-(\sigma-1)}) \right) e^{-t(1+O(n^{-\epsilon/2}))},
  \end{align}
  where $C$ is a constant that only depends on $u_0,$ and \eqref{finalest1214} holds for all $n \geq n_0(\tilde u_0)$ and all $t$ while \eqref{Esmall} remains in force.
  In fact, we can compute explicitly that $C = C_3 e^{\frac{2C_3}{\sigma}}$ would work (assuming without loss of generality that $C_3 \geq 1$).
  This is exactly \eqref{mainerrest1122}.

Now we can go back 
and note that given $u_0$ and the initial fit controlled by $\M_0 n^{-1-\epsilon},$ we can define $\M_{\max}$ in \eqref{Esmall} to be equal to
\[ \max(C(u_0)\M_0,C(u_0)), \] where $C(u_0)$ is from \eqref{finalest1214}.
With this definition of $\M_{\max}$ at hand, we can determine the threshold value $n_0(\tilde u_0)$ such that for all $n \geq n_0(\tilde u_0),$ \eqref{finalest1214} holds for all $t.$
\end{proof}

\bigskip
\noindent {\bf Acknowledgement.} \rm AK acknowledges partial support
of the NSF-DMS grant 2006372 and of the Simons Fellowship. CT acknowledges partial support
of the NSF-DMS grant 1853001. We are grateful to Stefan Steinerberger for letting us know
about the problem and educating us on the history of the question and its connections to
free probability and random matrices.

\end{document}